\documentclass[11pt,reqno,a4paper]{amsart}

\usepackage{mdframed}
\usepackage{setspace}
\usepackage{mathrsfs}

\usepackage[latin1]{inputenc}

\usepackage[linkcolor=red,colorlinks=true]{hyperref}

\usepackage{etoolbox}
\usepackage{amsmath}
\usepackage{enumerate}
\usepackage{amssymb}
\usepackage{amscd}
\usepackage{amsthm}
\usepackage{amsfonts}
\usepackage{graphicx}
\usepackage[all,cmtip]{xy}
\usepackage{enumitem}

\patchcmd{\subsection}{-.5em}{.5em}{}{}
\patchcmd{\subsubsection}{-.5em}{.5em}{}{}

\usepackage{enumitem}

\makeatother
\usepackage{hyperref}

\numberwithin{equation}{section}

\newcommand{\SL}{\operatorname{SL}}

\newcommand{\GL}{\operatorname{GL}}

\newcommand{\cC}{\mathcal{C}}

\newcommand{\cE}{\mathcal{E}}
\newcommand{\cF}{\mathcal{F}}
\newcommand{\cG}{\mathcal{G}}
\newcommand{\cH}{\mathcal{H}}
\newcommand{\cI}{\mathcal{I}}
\newcommand{\cJ}{\mathcal{J}}

\newcommand{\cL}{\mathcal{L}}
\newcommand{\cM}{\mathcal{M}}
\newcommand{\cN}{\mathcal{N}}

\newcommand{\cY}{\mathcal{Y}}

\newcommand{\bN}{\mathbb{N}}

\newcommand{\bR}{\mathbb{R}}

\newcommand{\bZ}{\mathbb{Z}}

\newcommand{\at}{a}
\newcommand{\bt}{b}
\newcommand{\ct}{c}

\newcommand{\ra}{\rightarrow}

\newcommand{\qand}{\quad \textrm{and} \quad}

\newcommand\subsetsim{\mathrel{%
\ooalign{\raise0.2ex\hbox{$\subset$}\cr\hidewidth\raise-0.8ex\hbox{\scalebox{0.9}{$\sim$}}\hidewidth\cr}}}
\newcommand{\eps}{\varepsilon}

\DeclareMathOperator{\dist}{dist}

\DeclareMathOperator{\Lip}{Lip}
\DeclareMathOperator{\supp}{supp}

\DeclareMathOperator{\Stab}{Stab}

\DeclareMathOperator{\id}{id}
\DeclareMathOperator{\Span}{span}

\DeclareMathOperator{\height}{ht}

\DeclareMathOperator{\Vol}{Vol}

\newcommand{\ul}[1]{\underline{#1}}

\definecolor{lichtgrijs}{gray}{0.95}
\mdfdefinestyle{mystyle}{ %
    backgroundcolor=lichtgrijs, %
    linewidth=0pt, %
    innertopmargin=10pt, %
    innerbottommargin=10pt,%
    nobreak=true
}

\theoremstyle{theorem}
\newtheorem{theorem}{Theorem}[section]
\newtheorem{corollary}[theorem]{Corollary}

\newtheorem{lemma}[theorem]{Lemma}

\theoremstyle{definition}
\newtheorem{definition}[theorem]{Definition}
\newtheorem{remark}[theorem]{Remark}

\usepackage{changepage}
\usepackage{mathtools}
\usepackage{graphicx}
\usepackage{calc}

\usepackage[toc,page]{appendix}

\usepackage{scalerel}

\usepackage{extarrows}

\usepackage{scalerel}
\usepackage{stackengine,wasysym}

\usepackage{scalerel,stackengine}
\stackMath
\newcommand\reallywidehat[1]{%
\savestack{\tmpbox}{\stretchto{%
  \scaleto{%
    \scalerel*[\widthof{\ensuremath{#1}}]{\kern-.6pt\bigwedge\kern-.6pt}%
    {\rule[-\textheight/2]{1ex}{\textheight}}
  }{\textheight}%
}{0.5ex}}%
\stackon[1pt]{#1}{\tmpbox}%
}

\DeclareMathSymbol{\shortminus}{\mathbin}{AMSa}{"39}
\usepackage{mathtools}

\title[Uniform metrical theorem]{A uniform metrical theorem in multiplicative Diophantine approximation}

\author[M. Bj\"orklund]{Michael Bj\"orklund}
\address{Department of Mathematics, Chalmers, Gothenburg, Sweden}
\email{micbjo@chalmers.se}

\author[R. Fregoli]{Reynold Fregoli}
\address{Department of Mathematics, University of Z\"urich, Switzerland}
\email{reynold.fregoli@math.uzh.ch}

\author[A.  Gorodnik]{Alexander Gorodnik}
\address{Department of Mathematics, University of Z\"urich, Switzerland}
\email{alexander.gorodnik@math.uzh.ch}

\thanks{MB was supported by Swedish VR-grant 11253320,  RF and  AG were supported by SNF grant 200021--182089.}

\begin{document}

\begin{abstract}
For Lebesgue generic $({x}_1,x_2)\in \bR^2$, we investigate the distribution of small values of products
$q\cdot \|qx_1\| \cdot \|qx_2\|$ with $q\in\bN$, where $\|\cdot \|$ denotes the distance to the closest integer. The main result gives an asymptotic formula for 
the number of  $1\le q< T$ such that
$$
\at <q\cdot \|qx_1\| \cdot \|qx_2\|\le \bt\qand \|qx_1\|, \|qx_2\|\le \ct
$$
for given parameters $\at,\bt, \ct$
satisfying certain growth conditions.

\end{abstract}

\maketitle

\section{Introduction}

In this paper we will be interested in the basic problem of multiplicative Diophantine approximation regarding existence of small values for products
$$
q\cdot \|qx_1\| \cdot \|qx_2\|, \quad\hbox{with $q\in\bN$,}
$$
for Lebesgue generic $\ul{x}=(x_1,x_2)\in \bR^2$,
where $\|\cdot \|$ denotes the distance to the closest integer. 
The main result in this direction
was established by Gallagher \cite{G1}, who showed that for any non-increasing function $\psi:\bN\to\bR_+$
satisfying $\sum_{q=1}^\infty \psi(q)\frac{\log q}{q}=\infty$, the inequality 
$$
q\cdot\|qx_1\|\cdot \|qx_2\|\le\psi(q)
$$
has infinitely many solutions $q\in\bN$ for almost all $\ul{x}=(x_1,x_2)\in \bR^2$. For instance, 
for almost all $\ul{x}=(x_1,x_2)\in \bR^2$, there are infinitely many solutions $q\in\bN$ to the inequality  
$$
q\cdot\|qx_1\|\cdot \|qx_2\|\le (\log q)^{-2}.
$$
The next natural question involves the distribution of the set of solutions
\[
M(\ul{x}) := 
\big\{ q \in \bN \,  : \,  q\cdot \|qx_1\| \cdot \|qx_2\|\le \psi(q)\big\}.
\]
Wang and Yu \cite{WY} derived an explicit asymptotic formula for the cardinality 
$|M(\ul{x}) \cap [1,T)|$ as $T\to\infty$. Inhomogeneous versions of these problems were also studied by
Chow and Technau \cite{CT1,CT2}. The above results are usually described as \emph{asymptotic} Diophantine
problems. 

Here we explore the analogous \emph{uniform} Diophantine problem that involves
solutions to the inequalities 
\begin{equation*}
q\cdot \|qx_1\| \cdot \|qx_2\|\le \bt,\quad 1\le q< T
\end{equation*}
for given parameters $\bt$ and $T$.
It should be noted that there is an essential difference between asymptotic problems and uniform problems.
For instance, establishing a uniform version of the classical Khinchin theorem was addressed only recently
in \cite{KW,KR1,KR2}. A somewhat different related uniform problem
was also considered by Widmer \cite{Wid}. His result  concerns
the sets
\[
N(\ul{x};b) := \left\{ q \in \bN \,  : \,  
\|qx_1\| \cdot \|qx_2\| \leq b
\right\},
\]
and it follows from \cite{Wid} that if $\bt\gg (\log T)^{4+\eps}/T$ for some $\eps>0$, then
for almost all $\ul{x}\in\bR^2$,
$$
\big|N(\ul{x};\bt) \cap [1,T)\big|=4\cdot \bt(1-\log(4\bt))\cdot T+O_{\ul{x},\eps}\Big( \bt^{2/3}(-\log \bt) \cdot T^{2/3} (\log T)^{4/3+\eps/3} \Big).
$$
The condition on the parameter $\bt$ is probably not optimal and is needed to ensure that 
the required Diophantine condition in \cite{Wid} holds for almost all $\ul{x}\in \bR^2$.
However, it is not clear whether this method can give an asymptotic formula in a larger range.\\

Our first main result gives an asymptotic formula for the growth of the sets
\[
L(\ul{x};b) := \big\{ q \in \bN \,  : \, q\cdot \|qx_1\| \cdot \|qx_2\|\le b
\big\}.
\]

\begin{theorem}
\label{Cor_Cnt}
Let $\eta\in (1,2)$. Then for every $\eps > 0$ and for almost every $\ul{x} \in \bR^2$,
\begin{itemize}
\item when
$(\ln T)^{-\eta}\le b\le (\ln T)^{-\eta/(3-\eta)}$,
\begin{align*}
\big|L(\ul{x};\bt) \cap [1,T)\big| = 2 \cdot \bt \cdot (\ln T)^2 
+O_{\ul{x},\eps}\left(
b^{1-1/\eta} \cdot (\ln \ln T)^{6+\eps}  
\ln T
\right),
\end{align*}

\item when $b\ge (\ln T)^{-\eta/(3-\eta)}$,
\begin{align*}
\big|L(\ul{x};\bt) \cap [1,T)\big| = 2 \cdot \bt \cdot (\ln T)^2 +O_{\ul{x},\eps}\left(
b^{2/\eta}\cdot (\ln \ln T)^{6+\eps} (\ln T)^2
\right).
\end{align*}
\end{itemize}
\end{theorem}

Theorem \ref{Cor_Cnt} is proved in Section \ref{Sec:Cor}.
It will be deduced from a more general result which we now describe.

Given parameters $T\ge 1$ and $a,b,c>0$ satisfying
\[
 a < b <1 \qand c < \frac{1}{2},
\]
we consider the domains
\begin{equation}
\Omega=\Omega_{[1,T)}^{(a,b],c} := \left\{ (\ul{u},y) \in \bR^2 \times [1,T) \,  : \,  
\begin{array}{c}
\max(|u_1|,|u_2|) \leq c \\[0.2cm]
a < |u_1u_2| \cdot y \leq b
\end{array}
 \right\},
\end{equation}
and, for $\ul{x} = (x_1,x_2) \in \bR^2$, the corresponding sets
\[
Q_\Omega(\ul{x}) 
:= 
\left\{ q \in [1,T) \cap \bN \,  : \,  
\begin{array}{c}
\max(\|qx_1\|,\|qx_2\|) \leq c \\[0.2cm]
a < q\cdot \|qx_1\|\cdot  \|qx_2\| \leq b
\end{array}
\right\}.
\]
It is natural to expect that the cardinality $|Q_\Omega(\ul{x})|$ grows as 
the volume of $\Omega$.
Our next theorem proves such an asymptotic formula:
\begin{theorem}
\label{Thm_Cnt}
Suppose that 
\[
\at \ge (\ln T)^{-\theta}\qand \zeta\cdot \bt \le \ct^2  
\quad\hbox{for some $\zeta,\theta>0$.}
\]
Let $\eta\in (1,2)$. Then for every $\eps > 0$ and for almost every $\ul{x} \in \bR^2$,
\begin{itemize}
\item when
$(c\cdot\ln T)^{-\eta}\le b\le (c\cdot \ln T)^{-\eta/(3-\eta)}$,
\begin{align*}
\big|Q_\Omega(\ul{x})\big| = \Vol_3(\Omega)
+O_{\ul{x},\zeta,\theta,\eps}\left(
c^{-1}b^{1-1/\eta} \cdot (\ln \ln T)^{6+\eps} 
\ln T+1
\right),
\end{align*}

\item when $b\ge (c\cdot\ln T)^{-\eta/(3-\eta)}$,
\begin{align*}
\big|Q_\Omega(\ul{x})\big| = \Vol_3(\Omega)
+O_{\ul{x},\zeta,\theta,\eps}\left(
  b^{2/\eta}\cdot (\ln \ln T)^{6+\eps} (\ln T)^2 +1
\right).
\end{align*}
\end{itemize}
\end{theorem}

To compare the above main term with the error term,
we note that according to \eqref{eq:v} in 
Appendix \ref{sec:vol},
\begin{equation}\label{eq:om_vol}
\Vol_3(\Omega) \sim 2 \cdot (\ln T)^2 \cdot (\bt - \at)\quad \hbox{as $T\to\infty$.}
\end{equation}


Due to the assumption on the parameter $a$, we cannot directly take $a = 0$ in the Theorem
\ref{Thm_Cnt} to deduce almost sure asymptotics for the sets $L(\ul{x};b)$ in Theorem \ref{Cor_Cnt}.  
Instead we analyze the existence of lattice points in thin hyperbolic strips
and show that when $\at$ decays sufficiently fast, the sets $L(\ul{x},\at)$
are empty for almost all $\ul{x}\in \bR^2$ (see Section \ref{Sec:Cor}).

Our method also allows us to study the distribution 
of the scaled set $Q_\Omega(\ul{x})/T$ in the interval $[0,1]$.  
For a compactly supported Lipschitz function $h : \bR \ra \bR$ consider the weighted counting function\[
S_\Omega h(\ul{x}) := \sum_{q \in Q_\Omega(\ul{x})} h\left(\frac{q}{T}\right),  \quad \ul{x} \in \bR^2.
\]
We shall show that $S_\Omega h(\ul{x})$ is asymptotic on average to 
the weighted mean 
\begin{equation}\label{eq:MT}
\cM_\Omega (h) := \int_{1}^T h\left(\frac{y}{T}\right) \cdot \Vol_2(\Omega(y)) \,  dy,
\end{equation}
where 
$$
\Omega(y) := \{ x \in \bR^2 \,  : \,  (x,y) \in \Omega\}.
$$
We note that $\Vol_2(\Omega(y))$ can be computed explicitly (see Lemma \ref{Lemma_VolOmegaTy}). 
When $h = 1$ on the interval $[0,1]$,  the sum $S_\Omega h(\ul{x})$ is simply $|Q_\Omega(\ul{x})|$ and
$\cM_\Omega(h) = \Vol_3(\Omega)$.  For different oscillatory functions $h$,  the
asymptotics of $S_\Omega h(\ul{x})$ describe the distribution of $Q_\Omega(\ul{x})$ inside $[1,T)$. 

We establish a bound for "sub-quadratic" moments defined in terms of the function
\begin{equation}\label{eq:theta00}
\theta_\kappa(s):=s^2/\ln(e+|s|)^{1+\kappa}\quad\hbox{with $\kappa>0$}.
\end{equation}

\begin{theorem}
\label{Thm_Main}
With the assumption as in Theorem \ref{Thm_Cnt},  for every compactly supported Lipschitz function $h : \bR \ra \bR$ and $\kappa > 0$,
\[
\int_{[0,1)^2} \theta_\kappa\Big(
S_\Omega h(\ul{x})-\cM_\Omega(h) \Big)\, d\ul{x}\ll_{h,\zeta, \theta,\kappa} c^{-2}b\cdot (\ln\ln T)^{3+\kappa} (\ln T)^2.
\]
In particular,
\[
\int_{[0,1)^2} \theta_\kappa\Big(
|Q_\Omega(\ul{x})|-\Vol_3(\Omega) \Big)\, d\ul{x}\ll_{\zeta, \theta,\kappa} c^{-2}b\cdot (\ln\ln T)^{3+\kappa} (\ln T)^2.
\]

\end{theorem}

\begin{remark}
In view of the volume asymptotics \eqref{eq:om_vol},
for certain ranges of the parameters, the  formula in Theorem \ref{Thm_Cnt} establishes  an error term with essentially "square-root cancellation".
\end{remark}

Theorem \ref{Thm_Main} is proved in Section \ref{Sec:Main}.
Then Theorem \ref{Thm_Cnt} is deduced from it in  Section \ref{Sec:Main Thm}
using a Borel-Cantelli argument.
We outline the structure of the paper in Section \ref{sec:org} below.

\medskip

\noindent {\bf Acknowledgement.}
We would like to thank Shunsuke Usuki for pointing out an issues in a previous version of the paper.

\setcounter{tocdepth}{1} 
\tableofcontents

\section{Organisation of the paper}\label{sec:org}

It will be convenient to consider a more geometric framework.  
For $\ul{x} \in \bR^2$,  
we consider the unimodular lattice
\begin{equation}\label{eq:l0}
\Lambda_{\ul{x}} := \{ (\ul{p} + q\ul{x},q) \in \bR^2 \times \bR \,  : \,  (\ul{p},q) \in \bZ^2 \times \bZ \}.
\end{equation}
Since $\ct < \frac{1}{2}$,  the map
\[
\Lambda_{\ul{x}} \cap \Omega \ra Q_\Omega(\ul{x}):  \enskip (\ul{p}+q\ul{x},q) \mapsto q
\]
is a bijection.  In  particular,  Theorem \ref{Thm_Cnt} can be viewed as an asymptotic
lattice point counting problem for the cardinality $|\Lambda_{\ul{x}} \cap \Omega|$.  

It order to produce an asymptotic formula for  $|\Lambda_{\ul{x}}\cap\Omega|$, 
one usually aims to establish a bound on the 
second moment $\int_{[0,1)^2} \big||\Lambda_{\ul{x}}\cap\Omega|-\Vol_3(\Omega)\big|^2 \, d\ul{x}$. 
This idea goes back, for instance, to the work of 
W.~Schmidt \cite{Schmidt60a}.
We also refer to the work Kleinbock, Shi, and Weiss \cite{KSW}, where this approach 
was developed for ergodic sums. 
However, our analysis in Section \ref{Sec:Height} suggests that 
a needed bound for the second moment is not attainable. 
As a substitute we establish
a bound for "sub-quadratic" moments
$\int_{[0,1)^2} \theta_\kappa\big(|\Lambda_{\ul{x}}\cap\Omega|-\Vol_3(\Omega)\big) \, d\ul{x}$,
with $\theta_\kappa$ as in \eqref{eq:theta00} and 
generalize the argument of \cite{Schmidt60,Schmidt60a}
to such sub-quadratic moments.

It will be convenient to view the domains $\Omega\subset\bR^3$ as disjoint union
of smaller domains. To construct such a decomposition, we use the action by the diagonal matrices
$$
a(t):=\hbox{diag}(e^{t_1},e^{t_2},e^{-t_1-t_2}),\quad t=(t_1,t_2)\in \bR^2_+.
$$
In Section \ref{Sec:Tesselate}, we show that 
\begin{equation}\label{eq:o}
\Omega =  \bigsqcup_{n \in \mathcal{F}_\Omega} a(n)^{-1}\Delta_{\Omega,n},
\end{equation}
where $\mathcal{F}_\Omega$ is a finite subset of $\bN_o^2$ and 
$\Delta_{\Omega,n}$ are finite subsets of $\bR^3$.
We note that this tessellation procedure is different compared to  
the one used in our previous work \cite{BG}. While in \cite{BG} 
we used varying tessellations defined for each $a(t)$-orbit,
here we construct a tessellation directly in $\bR^3$.
This has advantages and disadvantages. 
In particular, in our present construction the tiles $\Delta_{\Omega,n}$
have more complicated shape. This necessitates a new notion of regularity
($(\eps,\gamma,M)$-controlled sets), which we develop in Section \ref{Sec:SL3pert}. 

The decomposition \eqref{eq:o} allows us to write
\begin{equation}\label{eq:sum}
|\Lambda_{\ul{x}}\cap\Omega|=\sum_{n \in \mathcal{F}_\Omega} \widehat \chi_{\Delta_{\Omega,n}}(a(n)\Lambda_{\ul{x}}),
\end{equation}
where $\widehat \chi_{\Delta_{\Omega,n}}$ denote the Siegel transforms of 
the characteristic function of the sets $\Delta_{\Omega,n}$.
Hence, the original counting question is reformulated in terms of ergodic
sums for the action of the group $\{a(t):t\in\mathbb{R}^2\}$, albeit these sums are computed along a varying 
family of functions. The crucial ingredient of our proof is a quasi-independence 
estimate (Theorem~\ref{Thm_BG}) established in our previous work \cite{BGEME}.
It gives a quantitative bound for the correlations of $\varphi\circ a(n)$, $n\in \bN_o^2$,
for smooth compactly supported functions $\varphi$ on the space space of unimodular lattices. It should be noted that this bound is only useful when $\min(n_1,n_2)$
is large, and one of the hardest parts of the present paper is consists in  
 treating the part of the sum (\ref{eq:sum}) where $\min(n_1,n_2)$ is small.

In Section \ref{Sec:Height} we establish various non-divergence estimates for lattices
$a(t)\Lambda_{\ul{x}}$  with $t\in\bR_+^2$ and $\ul{x}\in [0,1)^2$ (and somewhat more general lattices). Although this question fits in the general framework of non-divergence
of unipotent flows (cf. \cite{D,S,KM11}), we need much more precise information
about the height functions along $a(t)\Lambda_{\ul{x}}$ (see, for instance, Corollary~\ref{Cor_Heightintegrals}). The results of Section \ref{Sec:Height}
will be used in the proof of the main theorem to estimate the part of the sum with small
$\min(n_1,n_2)$ (this corresponds to the term $\cC^{(1)}_\Omega$ in Section \ref{Sec:Main}) and also in the construction of smooth approximations.

In Section \ref{Sec:Smooth} we build smooth compactly supported approximations
for the functions $\widehat \chi_{\Delta_{\Omega,n}}$. Our main technical result here, used in the proof of the main theorem, is Lemma \ref{Lemma_Approx1}
(this corresponds to the term $\cC^{(2)}_\Omega$ in Section \ref{Sec:Main}).
The proof of this lemma uses the notion of 
$(\eps,\gamma,M)$-controlled sets and the non-divergence estimates from Section \ref{Sec:Height}.
Ultimately, its proof reduces to an arithmetic problem of estimating the number
 of points from $a(t)\Lambda_{\ul{x}}$ contained in an $(\eps,\gamma,M)$-controlled set (see Lemma \ref{Lemma_L2Siegel})
 and to an estimate for correlations of certain lattice counting functions (see Lemma \ref{Lemma_MeanDt}).
Those are handled in Sections \ref{Sec_Cor}--\ref{Sec:MeanCnt}.

The proof of Theorem \ref{Thm_Main} is completed in Section \ref{Sec:Main}.
In Section \ref{Sec:Main Thm}  we deduce Theorem \ref{Thm_Cnt} from Theorem \ref{Thm_Main} using a Borel-Cantelli argument.
Finally, Theorem \ref{Cor_Cnt} is deduced from \ref{Thm_Cnt}
in Section \ref{Sec:Cor}.

\section{Tessellations}
\label{Sec:Tesselate}

In this section we show how to decompose the "hyperboloid strips" 
\begin{equation}
\label{Def_OmegaT}
\Omega :=
\left\{ 
(\ul{x},y) \in \bR^2 \times \bR \,  : \, 
\begin{array}{c}
\at < |x_1 x_2| \cdot y \leq \bt,  \\[0.3cm]
\enskip \max(|x_1|,|x_2|) \leq \ct,  \\[0.3cm]
\enskip T_0 \leq y < T
\end{array}
\right\}.
\end{equation}
defined for $1\le T_0<T$, $0<a<b$, and $c> 0$. We use the action of the semigroup of diagonal matrices of the form
\begin{equation}
\label{Def_at}
a(t)
:=
\left(
\begin{matrix}
e^{t_1} & & \\
& e^{t_2} & \\
& & e^{-(t_1 + t_2)}
\end{matrix}
\right),  \quad \textrm{for $t = (t_1,t_2) \in \bR^2_{+}$}.
\end{equation}
Our aim below is to show that $\Omega$ can be written as a union of certain 
$a(t)$-translates of smaller pieces $\Delta_{\Omega,n}$,  where $n = (n_1,n_2) \in \bN_o^2$,  and whose dependence
on the parameters is rather mild.  More precisely,  let
\begin{equation}
\label{Def_DeltaTn}
\Delta_{\Omega,n} := \left\{ (\ul{x},y) \,  : \,  
\begin{array}{c}
\at < |x_1 x_2| \cdot y \leq \bt \\[0.3cm]
\ct \,  e^{-1} < |x_i| \leq \ct,  \enskip \textrm{for $i=1,2$} \\[0.3cm]
T_0 e^{-(n_1 + n_2)} \leq y < T e^{-(n_1 + n_2)}
\end{array} 
\right\}
\end{equation}
and
\begin{equation}
\label{Def_FT}
\cF_\Omega := \left\{ n \in \bN_o^2 \,  : \,  \alpha_\Omega 
\leq n_1 + n_2 < \beta_\Omega \right\},
\end{equation}
where
\begin{equation}
\label{Def_AB}
\alpha_\Omega := \ln\left(\frac{T_0\ct^2}{\bt e^2}\right) \qand \beta_\Omega :=  \ln\left(\frac{T\ct^2}{\at}\right).
\end{equation}
Then, we have the following lemma:

\begin{lemma}
\label{Lemma_tesselate}
For every $n \in \bN_o^2$,  
\vspace{0.1cm}
\[
\Delta_{\Omega,n} \neq \emptyset \implies n \in \cF_\Omega
\]
and
\[
\Omega =  \bigsqcup_{n \in \mathcal{F}_\Omega} a(n)^{-1}\Delta_{\Omega,n},
\]
Furthermore,  for all $n \in \cF_\Omega$, 
\[
\Delta_{\Omega,n} 
\subset
 \left[\shortminus \ct,\ct\right]^2
 \times 
 \left(\frac{\at}{\ct^2},\frac{\bt e^2}{\ct^2}\right].
\]
\end{lemma}

\begin{proof}
We first note that for every $n = (n_1,n_2) \in \bN_o^2$,
\vspace{0.1cm}
\[
a(n)^{-1}\Delta_{\Omega,n} = \left\{ (\ul{x},y) \,  : \,  
\begin{array}{c}
\at < |x_1 x_2|y \leq \bt \\[0.3cm]
\ct \,  e^{-(n_i+1)} < |x_i| \leq \ct e^{-n_i},   \enskip \textrm{for $i=1,2$} \\[0.3cm]
T_0 \leq y < T 
\end{array} 
\right\},
\]
so in particular,  $a(n)^{-1}\Delta_{\Omega,n} \subset \Omega$ and the sets $\{ a(n)^{-1}\Delta_{\Omega,n}\}_{n \in \bN^d}$ are disjoint.  Pick a point $(\ul{x},y) \in \Omega$.  Then there
is clearly a unique $n \in \bN_o^2$ such that
\[
\ct\cdot  e^{-(n_i+1)} < |x_i| \leq \ct \cdot e^{-n_i},  \quad \textrm{for both $i=1,2$},
\]
and thus $(\ul{x},y) \in a(n)^{-1}\Delta_{\Omega,n}$.  Since $(\ul{x},y)$ is arbitrary,  this shows that 
\begin{equation}
\label{firstunion}
\Omega =  \bigsqcup_{n \in \bN_o^2} a(n)^{-1}\Delta_{\Omega,n}.
\end{equation}
However,  most of the sets $\Delta_{\Omega,n}$ in this union are empty.  Indeed,  suppose that
$\Delta_{\Omega,n} \neq \emptyset$,  and pick $(\ul{x},y) \in \Delta_{\Omega,n}$.  Then,
\[
\at < |x_1 x_2| \cdot y < \ct^2 \cdot T \cdot e^{-(n_1 + n_2)},
\]
and thus $n_1 + n_2 < \ln\left(\frac{T \ct^2}{\at}\right)$.  Similarly,
\[
\ct^2 \cdot T_0\cdot e^{-2} \cdot e^{-(n_1 + n_2)} \leq |x_1 x_2|y \leq \bt,
\]
and thus $\ln\left(\frac{T_0 \ct^2}{\bt e^2}\right) \leq n_1 + n_2$.  \\
In other words,  $\Delta_{\Omega,n}$ is non-empty only if $n \in \cF_\Omega$,  and thus the union in \eqref{firstunion} can without loss of generality be restricted to $\cF_\Omega$.  

Finally,  suppose that $(\ul{x},y) \in \Delta_{\Omega,n}$ for some $n$.  Then $\ul{x} \in [\shortminus \ct,\ct]^2$ and
\[
\frac{\at}{\ct^2} < \frac{\at}{|x_1 x_2|} < y \leq \frac{\bt}{|x_1 x_2|} \leq \frac{\bt e^2}{\ct^2},
\]
and thus
\[
\Delta_{\Omega,n} \subset  \left[\shortminus \ct,\ct\right]^2 
 \times 
 \left(\frac{\at}{\ct^2},\frac{\bt e^2}{\ct^2}\right]
\]
independently of $n$.
\end{proof}







\section{Group perturbations and controlled sets}
\label{Sec:SL3pert}
The aim of this section is to establish 
regularity under small $\SL_3(\bR)$-perturbations for families sets of the as in \eqref{Def_DeltaTn}. Since the sets relevant to our problem degenerate
in some directions, this will need to introduce a new notion of regularity.

We begin with some notation.  If $E \subset \bR^2 \times \bR$,  let
\[
E^y = \{ \ul{x} \in \bR^2 \,  : \,  (\ul{x},y) \in E \} \qand E_{\ul{x}} = \{ y \in \bR \,  : \,  (\ul{x},y) \in E \},
\]
for $y \in \bR$ and $\ul{x} \in \bR^2$.  We refer to $E^y$ as the \emph{$y$-section} of 
$E$,  and to $E_{\ul{x}}$ as the \emph{$\ul{x}$-section} of $E$.  \\

The following definition captures a convenient 
form of regularity that we will use in the paper.

\begin{definition}[Controlled set]
\label{Def_Control}
Let $M > 1$ and $0 < \eps < \gamma<M$.  We say that a Borel set $E \subset \bR^2 \times \bR$ is \emph{$(\eps,\gamma,M)$-controlled} if either
\[
E \subset [-M,M]^2 \times (\gamma,M] \qand \sup_{y \in [\gamma,1]} \Vol_2(E^y)
\ll_M
\max\left(\eps,-\frac{\eps}{\gamma} \ln\left(\frac{\eps}{\gamma}\right)\right),
\]
or  if there is an interval $[\alpha,\beta]$ such that
\[
E \subset [-M,M]^2 \times [\alpha,\beta],  \quad \beta-\alpha \ll_M \eps,\qand  \alpha\geq \gamma/2
\]
with implicit constants that are independent of $\eps$ and $\gamma$.  If $E$
satisfies the first set of conditions,  we say that $E$ is \emph{type I} and if $E$
satisfies the second set of conditions,  we say that $E$ is \emph{type II}.  We refer
to the implicit constants above as the \emph{bounds for $E$}.
\end{definition}

Roughly speaking,  a controlled set is type I if all of its $y$-sections have uniformly small volumes and it is type II if all of its $\ul{x}$-sections have uniformly small volumes.

\begin{remark}
	\label{Rmk_smallergamma}
	Note that if $\eps < \gamma' < \gamma$, $\eps/\gamma'<1/e$,  and $E$ is an $(\eps,\gamma,M)$-controlled set,  then $E$ is also an $(\eps,\gamma',M)$-controlled set.
\end{remark}

Let $M>1$.
For the rest of this section,  let us fix the following real parameters:
\begin{equation}
\label{asspar0}
0 < a < b  \qand u_1^{-} < u_1^{+} \leq 1/2,  \enskip u_2^{-} < u_2^{+} \leq 1/2 \qand
0<\gamma < \delta \leq M.
\end{equation}
We define the set
\begin{equation}
\label{Def_Delta}
\Delta := \left\{ (\ul{x},y) \,  : \,  
\begin{array}{c}
a < |x_1 x_2| \cdot y \leq b \\[0.3cm]
u_i^{-} < |x_i| \leq u_i^{+}  \enskip \textrm{for $i=1,2$} \\[0.3cm]
\gamma \leq y <\delta
\end{array}
\right\} \subset \bR^2 \times \bR.
\end{equation}
We fix the $\max$-norm on $\bR^3$ and denote by $\|\cdot\|_{\textrm{op}}$
the operator norm on $\SL_3(\bR)$ induced by the $\max$-norm, in particular,
$$
\|g\cdot \ul{x}\|_\infty\le \|g\|_{\textrm{op}}\cdot \|\ul{x}\|_\infty,
\quad\hbox{for $g\in \SL_3(\bR)$ and $\ul{x}\in \bR^3$}
$$
For $\eps > 0$,  we denote by $V_\eps$ the symmetric open neighborhoods around the identity in 
$\SL_3(\bR)$:
\begin{equation}\label{eq:ve}
V_\eps := \left\{ g \in \SL_3(\bR) \,  : \,  \|g-\id\|_{\textrm{op}} < \eps\qand \|g^{-1}-\id\|_{\textrm{op}} < \eps \right\}.
\end{equation}
In other words,  
\begin{equation}
\label{Def_Veps}
\|g^{\pm 1}.(\ul{x},y) - (\ul{x},y)\|_\infty < \eps \cdot \|(\ul{x},y)\|_\infty,  \quad \textrm{for all $g \in V_\eps$ and $(\ul{x},y) \in \bR^2 \times \bR$}.
\end{equation}
The following lemma is the main result of this section.
\begin{lemma}
\label{Lemma_DeltaPert}
Let $\Delta$ be as in \eqref{Def_Delta} with the assumptions \eqref{asspar0}.
Then,  for every 
$$
0 < \eps < \min\left(1/(2M),\gamma/(2M),a/(M^2+1)\right),
$$
there 
are $(\eps,\gamma/2,M+1)$-controlled sets $E_s\subset \bR^2 \times \bR$, $s=1,\ldots,24$, such that
\[
\left(g^{-1}\Delta \setminus \Delta\right) \sqcup \left(\Delta \setminus g^{-1}\Delta\right) \subset \bigcup_{s} E_s\quad\hbox{for all $g \in V_\eps$. }
\]
 Furthermore,  the bounds in Definition \ref{Def_Control} for the sets $E_s$ are independent of the parameters $a,b,u_{1}^{\pm},u_{2}^{\pm},\delta$. 
\end{lemma}

\begin{proof}
Let us fix $\eps$ as above.  For $g \in V_\eps$ and $(\ul{x},y) \in \bR^2 \times \bR$,  we define $(\ul{x}(g),y(g))$ by
\begin{equation}\label{eq:action}
g.(\ul{x},y) := (\ul{x}(g),y(g)) = ({x}_1(g),x_2(g),y(g)).
\end{equation}
Then
\begin{equation}
\label{epsbnds}
\max(|x_1(g)-x_1|,|x_2(g)-x_2|,|y(g)-y|) < \eps \cdot \|(\ul{x},y)\|_\infty,
\end{equation}
for all $(\ul{x},y) \in \bR^2 \times \bR$. 
In particular, for $(\ul{x},y)\in \Delta$, 
$$
\max(|x_1(g)-x_1|,|x_2(g)-x_2|,|y(g)-y|) < M \eps.
$$
Hence, it follows that 
$$
u_i^{-}-M\eps < |x_i(g)| \leq u_i^{+}+M\eps  \enskip \textrm{for $i=1,2$} 
\qand
\gamma-M\eps \leq y(g) \leq\delta+M\eps.
$$
Also using that $\eps<1/(2M)$,
\begin{align*}
\big||x_1 x_2| \cdot y-|x_1(g) x_2(g)| \cdot y(g)\big|
\le& \big||x_1 x_2| \cdot y -|x_1(g) x_2| \cdot y \big|
+\big||x_1(g) x_2| \cdot y-|x_1(g) x_2(g)| \cdot y\big|\\
&+\big||x_1(g) x_2(g)| \cdot y-|x_1(g) x_2(g)| \cdot y(g)\big|
\le \eps M^2.
\end{align*}
Therefore, we also have that
$$
a-\eps M^2 < |x_1(g) x_2(g)| \cdot y(g) \leq b+\eps M^2.
$$
We conclude that for $g\in V_\eps$,
$$
g^{-1}\Delta\subset \Delta_\eps^+,
$$
where
$$
\Delta_\eps^+ := \left\{ (\ul{x},y) \,  : \,  
\begin{array}{c}
a -\eps M^2< |x_1 x_2| \cdot y \leq b +\eps M^2\\[0.3cm]
u_i^{-}(\eps)< |x_i| \leq u_i^{+}(\eps)  \enskip \textrm{for $i=1,2$} \\[0.3cm]
\gamma -\eps M \leq y <\delta+\eps M
\end{array}
\right\},
$$
where $u_i^{-}(\eps):=u_i^{-} -\eps M$ and $u_i^{+}(\eps):=u_i^{+} +\eps M$.
We note that 
\begin{equation}
\label{eq:M}
\Delta_\eps^+\subset (-1,1)^2\times [\gamma/2,M+1],
\end{equation}
because $\eps<\gamma/(2M)$ and $\eps<1/(2M)$.

Let
$$
E_1:=\{(\ul{x},y)\in \Delta_\eps^+: \, a-\eps M^2<|x_1 x_2| \cdot y \le a\}.
$$
For all $y>0$,
\[
E_{1}^y \subset \left(
\begin{tabular}{ll} 
$u_1^{+}(\eps)$ & 0 \\
0 & $u_2^{+}(\eps)$
\end{tabular}
\right) \cdot \left(\Xi\left( \frac{a}{y \cdot u_1^{+}(\eps) u_2^{+}(\eps)} \right) \setminus  
\Xi\left( \frac{a-\eps M^2}{y \cdot u_1^{+}(\eps) u_2^{+}(\eps)} \right) \right),
\]
where the set $\Xi(\cdot)$ is defined as in \eqref{Def_xigamma}.  
It follows from \eqref{eq:mean} that
\begin{align*}
&\Vol_2\left( 
\Xi\left( \frac{a}{y \cdot u_1^{+}(\eps) u_2^{+}(\eps)} \right) \setminus  
\Xi\left(\frac{a-\eps M^2}{y \cdot u_1^{+}(\eps) u_2^{+}(\eps)}\right)\right)\\
\leq & -\frac{4\eps M^2}{y \cdot u_1^{+}(\eps) u_2^{+}(\eps)} \cdot \ln^-\left (\frac{a-\eps M^2}{y \cdot u_1^{+}(\eps) u_2^{+}(\eps)}\right),
\end{align*}
where $\ln^-(z):=\ln\min(1,z)$.
Since $\eps < a/(M^2+1)$ and $u_1^+(\eps),u_2^+(\eps)< 1$,  we have $$
\frac{a-\eps M^2}{u_1^{+}(\eps)u_2^{+}(\eps)} > \eps,
$$
and
thus
\begin{equation}
\label{VolE11}
\Vol_2((E_{1})^y) \ll_M -\frac{\eps}{y} \cdot \ln^-\left(\frac{\eps}{y}\right) \ll -\frac{2\eps}{\gamma} \cdot \ln^-\left(\frac{2\eps}{\gamma}\right)
\end{equation}
for all $y\ge \gamma/2$. This verifies that the set $E_1$ is $(\eps,\gamma/2,M+1)$-controlled.  

Let
$$
E_2:=\{(\ul{x},y)\in \Delta_\eps^+: \, u_1^-(\eps) < x_1\le  u_1^-\}.
$$
Then for all $y$,
$$
\Vol_2((E_{1})^y)\le M\eps,
$$
so that this set is also $(\eps,\gamma/2,M+1)$-controlled.  
The set
$$
E_3:=\{(\ul{x},y)\in \Delta_\eps^+: \, \gamma-M\eps \le y <  \gamma\}.
$$
is also obviously $(\eps,\gamma/2,M+1)$-controlled of type II.

Finally, we observe that $\Delta_\eps^+\backslash\Delta=\Delta_\eps^+\cap\Delta^c$ 
is contained in a union of of subsets of $\Delta_\eps^+$, where each subset is defined
by the negation of one of the inequalities that appears in the definition of $\Delta$.
Furthermore, in the case of ${x}_i$, $i=1,2$, we get two sets of inequalities 
$u_i^{-}(\eps)< |x_i|\le u_i^-$ and $u_i^{+}< |x_i| \leq u_i^{+}(\eps)$
which we view as four sets of inequalities in terms of ${x}_i$ and consider the four corresponding subsets.
This way we obtain twelve sets $E_s$. Since these sets are defined similarly to either 
$E_1$, $E_2$, or $E_3$, we can analyze them as above. 
Therefore, we conclude that all $E_s$'s are $(\eps,\gamma/2,M+1)$-controlled. 
This verifies the claim of the lemma for the set
$g^{-1}\Delta\backslash \Delta\subset \Delta_\eps^+\backslash\Delta$. 

To handle the set $\Delta\backslash g^{-1}\Delta$,
we observe that for $g\in V_\eps$,
$$
g^{-1}\Delta\supset \Delta_\eps^-,
$$
where
$$
\Delta_\eps^- := \left\{ (\ul{x},y) \,  : \,  
\begin{array}{c}
a +\eps M^2< |x_1 x_2| \cdot y \leq b -\eps M^2\\[0.3cm]
u_i^{-} +\eps M< |x_i| \leq u_i^{+} -\eps M  \enskip \textrm{for $i=1,2$} \\[0.3cm]
\gamma +\eps M \leq y <\delta-\eps M
\end{array}
\right\},
$$
and 
$\Delta\backslash \Delta_\eps^-$  is contained in the union of 
twelve $(\eps,\gamma/2,M+1)$-controlled sets. This can be verified as above,
so that we omit the details.
\end{proof}





\section{Height function estimates}
 \label{Sec:Height}
 
In this section we define a height function on the space $\widetilde{\cL}_3$ consisting of all lattices in $\bR^3$ and
prove some technical level set estimates for this function that will be used later in our analysis of Siegel transforms.  

Let $\widetilde{\cL}_3$ denote the space of all lattices in $\bR^3$,  endowed with
the standard action of $\GL_3(\bR)$.  Given a bounded Borel function $f : \bR^3 \ra \bR$
such that the set $\{ f \neq 0\}$ is bounded,  we define the \emph{Siegel transform} 
$\widehat{f} : \widetilde{\cL_3} \ra \bR$ by
\[
\widehat{f}(\Lambda) := \sum_{\ul{\lambda} \in \Lambda \setminus \{0\}} f(\ul{\lambda}),  \quad \textrm{for $\Lambda \in \widetilde{\cL_3}$}.
\]
Let $\cL_3 \subset \widetilde{\cL}_3$ denote the subspace of all \emph{unimodular} lattices.  This subspace is preserved by the restricted $\SL_3(\bR)$-action,  and it is 
well-known that $\cL_3$ carries a unique $\SL_3(\bR)$-invariant Borel probability measure,  which we denote by $\mu$.  The following classical theorem of Siegel will play an important role in
our analysis.

\begin{theorem}[Siegel's theorem]
	\label{Thm_Siegel}
	Let $f$ be a bounded Borel function with compact support.  
	Then $\widehat{f} \in L^1(\mu)$ and
	\[
	\int_{\cL_3} \widehat{f} \,  d\mu = \int_{\bR^3} f(\ul{z}) \,  d\ul{z},
	\]
	where the volume measure on $\bR^3$ has been normalized so that $\Vol_3([0,1]^3) = 1$.
\end{theorem}

\subsection{A height function on $\tilde \cL_3$}
\label{Subsec:heightL3}
Let $\{\ul{e}_1,\ul{e}_2,\ul{e}_3\}$ denote the standard
(ordered) basis of $\bR^3$.  We extend the $\max$-norm with respect to this basis to the second exterior power $\bR^3 \wedge \bR^3$ as follows.  If $\ul{u},  \ul{v} \in \bR^3$ and 
\[
w = \ul{u} \wedge \ul{v} = w_{12} \, \ul{e}_1 \wedge \ul{e}_2 + w_{13} \, \ul{e}_1 \wedge \ul{e}_3 + w_{23} \, \ul{e}_2 \wedge \ul{e}_3,
\]
then $\|w\|_\infty = \max(|w_{12}|,|w_{13}|,|w_{23}|)$.  \\

Let $\Lambda$ be a (not necessarily unimodular) lattice in $\bR^3$.  We define
the functions
\begin{align*}
\label{Def_s1}
s_1(\Lambda) 
&:= 
\inf\left\{ \|\ul{\lambda}\|_\infty \,  : \,  \ul{\lambda} \in \Lambda \setminus \{0\} \right\}, \\[0.2cm]
s_2(\Lambda)
&:=
\inf\left\{ \|\ul{\lambda}_1 \wedge \ul{\lambda}_2\|_\infty \,  : \,  \ul{\lambda}_1 \wedge \ul{\lambda}_2 \neq 0,  \enskip \ul{\lambda}_1,\ul{\lambda}_2 \in \Lambda \setminus \{0\} \right\}, \\[0.2cm]
s_3(\Lambda)
&:=
\inf\left\{ |\alpha| \,  : \,  
\alpha \, \ul{e}_1 \wedge \ul{e}_2 \, \wedge \ul{e}_3 = \ul{\lambda}_1 \wedge \ul{\lambda}_2 \wedge \ul{\lambda}_3 \neq 0, \enskip \ul{\lambda}_1,\ul{\lambda}_2,  \ul{\lambda}_3 \in \Lambda \setminus \{0\} \right\},
\end{align*}
and the \emph{height function}
\begin{equation}
\label{Def_ht}
\height(\Lambda) := \min\big(s_1(\Lambda),s_2(\Lambda),s_3(\Lambda)\big)^{-1}.
\end{equation}
Since 
\[
\|g\cdot\ul{w}\|_\infty \leq \|g\|_{\textrm{op}} \cdot \|\ul{w}\|_\infty 
\qand \|g\cdot \ul{w}\|_\infty \geq \|g^{-1}\|^{-1}_{\textrm{op}} \cdot \|\ul{w}\|_\infty,  
\] 
for all $g \in \GL_3(\bR)$ and $\ul{w} \in \bR^3$ (where $\|\cdot\|_{\textrm{op}}$ denotes the operator norm with respect to the $\max$-norm on $\bR^3$), we have 
\begin{equation}
\label{upperbnd_ht}
\|g\|^{-1}_{\textrm{op}}  \cdot \height(\Lambda) \leq \height(g.\Lambda) \leq  \|g^{-1}\|_{\textrm{op}}\cdot \height(\Lambda),  
\end{equation}
for all $g \in \GL_3(\bR)$ and $\Lambda \in \widetilde{\cL}_3$.  Before we proceed to the main topic of this section,  we recall an important inequality due to Schmidt \cite[Lemma 2]{Schmidt68}.

\begin{lemma}
\label{Lemma_Schmidt}
For every bounded Borel function with compact support,
\[
|\widehat{f}(\Lambda)| \ll_{\supp(f)} \|f\|_\infty \cdot \height(\Lambda),  \quad \textrm{for all $\Lambda \in \widetilde{\cL}_3$},
\]
where the implicit constants only depend on $\supp(f)$.
\end{lemma}

\subsection{Main results}

We recall that
\[
a(t) 
=
\left(
\begin{matrix}
e^{t_1} & & \\
& e^{t_2} & \\
& & e^{-(t_1 + t_2)}
\end{matrix}
\right),  \quad \textrm{for $t = (t_1,t_2) \in \bR_{+}^2$}.
\]
Given $\ul{x} =({x}_1,x_2)\in  \bR^2$ and $r > 0$,  we define the lattice
\begin{equation}
\label{Def_Lambdaxr}
\Lambda_{\ul{x},r} := \left\{ \ul{p} + q(\iota(\ul{x}) + r \ul{e}_3) \,  : \,  \ul{p} \in \Span_{\bZ}(\ul{e}_1,\ul{e}_2),  \enskip q \in \bZ \right\},
\end{equation}
where $\iota(\ul{x}):=({x}_1,x_2,0)$.
We note that $\Lambda_{\ul{x},r}$ is a unimodular lattice if and only if $r = 1$.

Our first theorem in this section provides uniform upper bounds of the height function along $a(t)$-orbits of lattices of the form $\Lambda_{\ul{x},r}$. Here and later in the paper,
we use the notation
$$
\lfloor t \rfloor:=\min(t_1,t_2)\quad\hbox{for $t=(t_1,t_2)\in\bR^2_{+}$.}
$$

\begin{theorem}
\label{Thm_Upperbound}
For all $r>0$ and $t = (t_1,t_2) \in \bR^2_{+}$,  
\[
\sup\left\{ \height(a(t)\Lambda_{\ul{x},r}) \,  : \,  \ul{x} \in [0,1)^2 \right\} \leq \max\left(e^{t_1 + t_2}/r,e^{-\lfloor t \rfloor} \right).
\]
\end{theorem}

Our second theorem roughly tells us that the map 
$t \mapsto\height(a(t)\Lambda_{\ul{x},r})$ is not large on a big volume set.  

\begin{theorem}
\label{Thm_Levelsets}
For every $r>0$, $L > \max(1,1/r)$ and $t \in \bR^2_{+}$,  
\[
\Vol_2\left( \left\{ \ul{x} \in [0,1)^2 \,  : \,  \height(a(t)\Lambda_{\ul{x},r}) \geq L \right\}\right) 
\ll \max(r^{-1},r^{-2})L^{-3} + r^{-1}L^{-2} e^{-\lfloor t \rfloor} ,
\]
where the implicit constants are independent of $L, r$ and $t$.
\end{theorem}

The theorems above provide useful upper bounds on certain integrals involving the
height functions restricted to super-level sets.  To state these bounds,  we define the sets
\begin{equation}
\label{Def_Mtr}
M_{t,r}(\eta) := \left\{ \ul{x} \in [0,1)^2 \,  : \,  \height(a(t)\Lambda_{\ul{x},r}) \geq \eta \right\}
\end{equation}
for $\eta \geq 0$.  We then have the following corollary, which we prove at the end of this section:

\begin{corollary}
\label{Cor_Heightintegrals}
Let $\rho>0$ and $\theta : [0,\infty) \ra [0,\infty)$ be an increasing measurable function such that $u \mapsto \frac{\theta(u)}{u^2}$ is decreasing on $[1,\infty)$.  Then,  for every $t\in\bR^2_+$, $r>\rho$ and $\eta\ge e^{2}\rho^{-1}$,  we have
\[
\int_{M_{t,r}(\eta)} \theta(\height(a(t)\Lambda_{\ul{x},r})) \,  d\ul{x}
\ll \left( \max(r^{-1},r^{-2}) \eta^{-1} + r^{-1}e^{-\lfloor t \rfloor} \right) \cdot 
\int_{ e^{-2}\eta}^{e^{t_1 + t_2 + 1}\max(1,\rho^{-1})} \frac{\theta(u)}{u^3} \,  du,
\]
where the implicit constants are independent of $\rho, \eta,  r$ and $t$.
\end{corollary}

\subsection{Lemmas about heights}

\begin{lemma}[Uniform Lower Bound for $s_{1}$]
\label{Lemma_s1low}
For all $r>0$ and $t = (t_1,t_2) \in \bR^2_{+}$,
\[
\inf\left\{ s_1(a(t)\Lambda_{\ul{x},r}) \,  : \,  \ul{x} \in [0,1)^2 \right\} \geq 
\min\left(e^{\lfloor t \rfloor},  r \cdot e^{-(t_1 + t_2)}\right).
\]
\end{lemma}

\begin{proof}
Let $\ul{x} = ({x}_1,x_2) \in [0,1)^2$, $r > 0$, and $t = (t_1,t_2) \in \bR^2_{+}$. 
We note that for every $\ul{p} = (p_1,p_2) \in \bZ^2$ and $q \in \bZ$,  we have
\[
\|a(t)(\ul{p} + q(\iota(\ul{x})+r\ul{e}_3))\|_\infty = \max\left(e^{t_1} \cdot |p_1 + qx_1|,  e^{t_2} \cdot  |p_2 + qx_2|, r \cdot |q| \cdot e^{-(t_1+t_2)}\right).
\]
We assume that $\ul{p} + q(\iota(\ul{x})+r\ul{e}_3) \neq 0$.  If $q = 0$,  then $\ul{p} \neq 0$,  and thus
\[
\|a(t)(\ul{p} + q(\iota(\ul{x})+r\ul{e}_3))\|_\infty \geq e^{\lfloor t \rfloor}.
\]
If $q \neq 0$,  then
\[
\|a(t)(\ul{p} + q(\iota(\ul{x})+r\ul{e}_3))\|_\infty \geq r \cdot e^{-(t_1 + t_2)}.
\]
We get a uniform lower bound by taking the minimum of these two bounds.  
\end{proof}

\begin{lemma}[Volume Bound for $s_{1}$]
\label{Lemma_s1upp}
For every $0 < \eps < 1$, $r >0$ and $t \in \bR^2_{+}$,  we have
\[
\Vol_2\left(\left\{ \ul{x} \in [0,1)^2 \,  : \,  s_{1}(a(t)\Lambda_{\ul{x},r}) \leq \eps \right\}\right)
\ll \frac{\eps^3}{r},
\]
where the implicit constants are independent of $\eps,  r$and $t$.
\end{lemma}

\begin{proof}
Let $0 < \eps < 1$ and  $r >0$ and $t = (t_1,t_2) \in \bR^2_{+}$.  Pick $\ul{x}=({x}_1,x_2) \in [0,1)^2$ such that 
\[
s_{1}(a(t)\Lambda_{\ul{x},r}) \leq \eps.
\]
This means that we can find a non-zero vector $\ul{p} + q (\iota(\ul{x}) + r\ul{e}_3) \in \Lambda_{\ul{x},r}$
such that 
\[
\Lambda_{\ul{x}}(t) := e^{t_1}(p_1 + qx_1) \, \ul{e}_1 + e^{t_2}(p_2 + qx_2) \,  \ul{e}_2 + 
r e^{-(t_1 + t_2)}q \,  \ul{e}_3\in a(t)\Lambda_{\ul{x},r}
\]
satisfies $\|\Lambda_{\ul{x}}(t)\|_\infty \leq \eps$.  Writing this information coordinate-wise,  we get the inequalities
\begin{equation}
\label{pqineq}
|p_1 + qx_1| \leq \eps \cdot e^{-t_1}, 
\enskip 
|p_2 + qx_2| \leq \eps \cdot e^{-t_2},
\enskip
|q|  \leq \frac{\eps \cdot e^{t_1 + t_2}}{r}.
\end{equation}
Since $\eps < 1$ and $t_1,  t_2 \geq 0$,  there are no non-zero solutions for $p$ when $q = 0$.  

For $q \neq 0$,  we note that since $\ul{x} = ({x}_1,x_2) \in [0,1)^2$,  there are at most $O(q^2)$ choices for $\ul{p} = (p_1,p_2) \in \bZ^2$.  Furthermore,  for each such choice
of $\ul{p}$, 
\[
\Vol_2\left( \left\{ \ul{x} \in [0,1)^2 \,  : \,  |p_1 + qx_1| \leq \eps \cdot e^{-t_1}, 
\enskip 
|p_2 + qx_2| \leq \eps \cdot e^{-t_2} \right\}\right) \ll \frac{\eps^2 \cdot e^{-(t_1 + t_2)}}{q^2},
\]
where the implicit constants are independent of $\ul{p}$,  $\eps$ and $t$.  In particular, upon summing over all possible $\ul{p}$ and $q$,  we see that the volume of the set of $\ul{x} \in [0,1)^2$ for which the inequalities in \eqref{pqineq} are satisfied is bounded from above by $O( \eps^3/r)$,  which finishes the proof.
\end{proof}

\vspace{0.2cm}

Let us introduce some notation which will be used in the proofs below.  Given 
\[
r > 0,  \enskip \ul{x} \in [0,1)^2,  \enskip \ul{p}^{(1)},  \ul{p}^{(2)} \in \Span_{\bZ}(\ul{e}_1,\ul{e}_2),  \enskip q_1,q_2 \in \bZ,
\]
define $\ul{v}_1,  \ul{v}_2 \in \Lambda_{\ul{x},r}$ by
\begin{equation}
\label{Defv1v2}
\ul{v}_1 := \ul{p}^{(1)} + q_1(\iota(\ul{x}) + r\ul{e}_3) \qand \ul{v}_2 := \ul{p}^{(2)} + q_2(\iota(\ul{x}) + r\ul{e}_3).
\end{equation}
Furthermore,  given $t = (t_1,t_2) \in \bR^2$,  we define
\begin{align*}
\omega_{\ul{x}}(t) 
&:= 
a(t)\left(\left( \ul{p}^{(1)} + q_1(\iota(\ul{x}) + r\ul{e}_3) \right) \wedge \left( \ul{p}^{(2)} + q_2(\iota(\ul{x}) + r\ul{e}_3) \right)\right) \\[0.2cm]
&=
a(t)\left(\ul{p}^{(1)} \wedge \ul{p}^{(2)} + \left(q_2 \ul{p}^{(1)} - q_1 \ul{p}^{(2)}\right) \wedge \ul{x}  +
r \left(q_2 \ul{p}^{(1)} - q_1 \ul{p}^{(2)}\right) \wedge \ul{e}_3\right).
\end{align*}
Note that
\[
\omega_{\ul{x}}(t) = e^{t_1 + t_2 } \,  \omega_{1,2}(\ul{x}) \,  \ul{e}_1 \wedge \ul{e}_2 + r \,  e^{-t_2}w_1 \,  \ul{e}_1 \wedge \ul{e}_3 + r \,  e^{-t_1} \,  w_2 \,  \ul{e}_2 \wedge \ul{e}_3,
\]
where
\[
\omega_{1,2}(\ul{x}) = m(\ul{p}) + (w_1 x_2 - w_2 x_1) 
\qand \ul{w} = w_1 \ul{e}_1 + w_2 \ul{e}_2 = q_2 \ul{p}^{(1)} - q_1 \ul{p}^{(2)},
\]
and $m(\ul{p})$ is the unique integer such that $\ul{p}^{(1)} \wedge \ul{p}^{(2)} = m(\ul{p}) \,  \ul{e}_1 \wedge \ul{e}_2$.  In particular, 
\begin{equation}
\label{normomegaxt}
\|\omega_{\ul{x}}(t)\|_\infty = 
\max\left(e^{t_1 + t_2} \cdot |m(\ul{p}) + (w_1 x_2 - w_2 x_1)|,  e^{-t_2} \cdot r \cdot |w_1|,  e^{-t_1} \cdot r \cdot |w_2| \right).
\end{equation}
Note that if $\ul{v}_1$ and $\ul{v}_2$ are linearly independent,  then $\omega_{\ul{x}}(t) \neq 0$,  
and thus either $\ul{w} \neq 0$ or $m(\ul{p}) \neq 0$.  

\begin{lemma}[Uniform Lower Bound for $s_{2}$]
\label{Lemma_s2low}
For all $r>0$ and $t = (t_1,t_2) \in \bR^2_{+}$,     
\[
\inf\left\{ s_2(a(t)\Lambda_{\ul{x},r}) \,  : \,  \ul{x} \in [0,1)^2 \right\} \geq 
\min\left(r \cdot e^{-\lfloor t \rfloor},e^{t_1 + t_2}\right).
\]
\end{lemma}

\begin{proof}
Note that $s_2(a(t)\Lambda_{\ul{x},r})$ is the minimum of $\|\omega_{\ul{x}}(t)\|_\infty$,  when 
$\ul{v}_1$ and $\ul{v}_2$,  defined as in \eqref{Defv1v2}, vary over all linearly independent pairs of vectors in $\Lambda_{\ul{x},r}$.
Hence,  by \eqref{normomegaxt},  we need lower bounds on
\[
\max\left(e^{t_1 + t_2} \cdot |m(\ul{p}) + (w_1 x_2 - w_2 x_1)|,  e^{-t_2} \cdot r \cdot |w_1|,  e^{-t_1} \cdot r \cdot |w_2| \right),
\]
when either $\ul{w} \neq 0$ or $m(\ul{p}) \neq 0$.  If $\ul{w} \neq 0$,  then since $w_1$ and $w_2$
are integers,  we have
\[
\max\left(e^{t_1 + t_2} \cdot |m(\ul{p}) + (w_1 x_2 - w_2 x_1)|,  e^{-t_2} \cdot r \cdot |w_1|,  e^{-t_1} \cdot r \cdot |w_2| \right) \geq r \cdot e^{-\lfloor t \rfloor},
\]
and if $m(\ul{p}) \neq 0$ and $\ul{w}=0$,  then
\[
\max\left(e^{t_1 + t_2} \cdot |m(\ul{p}) + (w_1 x_2 - w_2 x_1)|,  e^{-t_2} \cdot r \cdot |w_1|,  e^{-t_1} \cdot r \cdot |w_2| \right) \geq e^{t_1 + t_2}.
\]
We get a uniform lower bound by taking the minimum of these two bounds.  
\end{proof}

\begin{lemma}[Volume Bound for $s_{2}$]
\label{Lemma_s2upp}
For every $0 < \eps < 1$, $r>0$ and $t \in \bR^2_{+}$,  we have
\[
\Vol_2\left(\left\{ \ul{x} \in [0,1)^2 \,  : \,  s_{2}(a(t)\Lambda_{\ul{x},r}) \leq \eps \right\}\right)
\ll \frac{\eps^3}{r^2} +\frac{\eps^2 \cdot e^{-\lfloor t \rfloor}}{r},
\]
where the implicit constants are independent of $\eps,  r$ and $t$.
\end{lemma}

\begin{proof}
Let $0 < \eps < 1$, $r>0$ and $t = (t_1,t_2) \in \bR^2_{+}$.  Pick $\ul{x} \in [0,1)^2$ such
that
\[
s_{2}(a(t)\Lambda_{\ul{x},r}) \leq \eps.
\]
This means that we can find two linearly independent vectors $\ul{v}_1$ and $\ul{v}_2$
as in \eqref{Defv1v2} such that $\|\omega_{\ul{x}}(t)\|_\infty \leq \eps$.  By \eqref{normomegaxt},  this results in the bounds
\begin{equation}
\label{wineq}
|m(\ul{p}) + (w_1 x_2 - w_2 x_1)| \leq \eps \cdot e^{-(t_1+t_2)},  
\enskip |w_1| \leq \frac{\eps \cdot e^{t_2}}{r},  
\enskip
|w_2| \leq \frac{\eps \cdot e^{t_1}}{r}.
\end{equation}
We make two observations: 
\vspace{0.1cm}
\begin{itemize}
\item If $q_1 = q_2 = 0$,  then $\ul{w} = 0$ and  $|m(\ul{p})| \leq  \eps \cdot e^{-(t_1+t_2)}$.  Since $\eps < 1$ and $t_1,  t_2 \geq 0$ and $m(\ul{p})$ is an integer,  
we must have $m(\ul{p}) = 0$.  This readily implies that $\ul{p}^{(1)} \wedge \ul{p}^{(2)} = 0$,  so 
$\ul{p}^{(1)}$ and $\ul{p}^{(2)}$ are linearly dependent,  and thus $\ul{v}_1$ and $\ul{v}_2$ are linearly dependent as well,  contrary to our assumption.  \vspace{0.1cm}
\item If $\ul{w} = 0$ and $(q_1,q_2) \neq (0,0)$,  then $q_2 \ul{v}_1 - q_1 \ul{v}_2 = 0$,  which contradicts our assumption that $\ul{v}_1$ and $\ul{v}_2$ are linearly independent.  
\end{itemize}
\vspace{0.1cm}
We can thus without loss of generality assume that $\ul{w} \neq 0$ and 
$(q_1,q_2) \neq (0,0)$.  \\

For a fixed $\ul{w} \neq 0$ and $m(\ul{p}) \in \bZ$,  we have
\[
\Vol_2\left(\left\{ \ul{x} \in [0,1)^2 \,  : \,  |m(\ul{p}) + (w_1 x_2 - w_2 x_1)| \leq \eps \cdot e^{-(t_1+t_2)} \right\}\right) \ll \frac{\eps \cdot e^{-(t_1 + t_2)}}{\max(|w_1|,|w_2|)},
\]
where the implicit constants are independent of $\ul{p}$ and $t$.  Furthermore, 
since $0<\eps<1$ and $t_1,t_2\ge 0$, it follows from \eqref{wineq} that
there are at most $O\big(\max(|w_1|,|w_2|)\big)$ choices for $m(\ul{p})$.  We also note from \eqref{wineq} that there 
are
\vspace{0.1cm}
\begin{itemize}
\item $O\left(\frac{\eps^2 \cdot e^{t_1+t_2}}{r^2}\right)$ choices for $\ul{w}$ with $w_1,  w_2 \neq 0$,  \vspace{0.1cm}
\item $O\left(\frac{\eps \cdot e^{t_2}}{r}\right)$ choices for $\ul{w}$ with 
$w_2 = 0$,  \vspace{0.1cm}
\item $O\left(\frac{\eps \cdot e^{t_1}}{r}\right)$ choices for $\ul{w}$ with 
$w_1 = 0$. 
\end{itemize}
\vspace{0.1cm}
Summing over all of these choices,  we get
\begin{align*}
\Vol_2\left(\left\{ \ul{x} \in [0,1)^2 \,  : \,  s_{2}(a(t)\Lambda_{\ul{x},r}) \leq \eps \right\}\right)
\ll \frac{\eps^3}{r^2} + \frac{\eps^2 \cdot e^{-t_1}}{r} + \frac{\eps^2 \cdot e^{-t_2}}{r},
\end{align*}
which implies the claim.
\end{proof}

\vspace{0.2cm}

Finally, we note the following elementary result for $s_3$,
whose proof is left to the reader.

\begin{lemma}
\label{Lemma_s3}
For all $t \in \bR^2_{+}$ and $r > 0$,  
\[
s_3(a(t)\Lambda_{\ul{x},r}) = r,  \quad \textrm{for all $\ul{x} \in [0,1)^2$}.
\]
\end{lemma}

\subsection{Proof of Theorem \ref{Thm_Upperbound}}

By Lemma \ref{Lemma_s1low},  Lemma \ref{Lemma_s2low} and Lemma \ref{Lemma_s3},
we have
\begin{align*}
\min(s_1(a(t)\Lambda_{\ul{x},r}),s_2(a(t)\Lambda_{\ul{x},r}),s_3(a(t)\Lambda_{\ul{x},r})) 
&\geq  \min\left(e^{\lfloor t \rfloor},  r \cdot e^{-(t_1 + t_2)} ,r \cdot e^{-\lfloor t \rfloor},e^{t_1 + t_2},r\right)\\[0.2cm]
&\geq \min\left(r \cdot e^{-(t_1 + t_2)},e^{\lfloor t \rfloor}\right),
\end{align*}
for all $t = (t_1,t_2) \in \bR^2_{+}$.  Hence, 
\[
\height(a(t)\Lambda_{\ul{x},r}) \leq \max\left(e^{t_1 + t_2}/r,e^{-\lfloor t \rfloor} \right),
\]
for all $\ul{x} \in [0,1)^2$.

\subsection{Proof of Theorem \ref{Thm_Levelsets}}

Let $L > 1$.  We first note that
\begin{align*}
\Vol_2\left( \left\{ \ul{x} \in [0,1)^2 \,  : \,  \height(a(t)\Lambda_{\ul{x},r}) \geq L \right\} \right)
&= 
\Vol_2\left( \bigcup_{i=1}^3 \left\{ \ul{x} \in [0,1)^2 \,  : \,  s_i(a(t)\Lambda_{\ul{x},r}) \leq \frac{1}{L}\right\} \right) \\[0.2cm]
&\leq 
\sum_{i=1}^3 
\Vol_2\left(\left\{ \ul{x} \in [0,1)^2 \,  : \,  s_i(a(t)\Lambda_{\ul{x},r}) \leq \frac{1}{L}\right\} \right).
\end{align*}
By Lemma \ref{Lemma_s1upp} and  Lemma \ref{Lemma_s2upp} (applied with $\eps = \frac{1}{L}$),  we have
\begin{align*}
&\Vol_2\left(\left\{ \ul{x} \in [0,1)^2 \,  : \,  s_1(a(t)\Lambda_{\ul{x},r}) \leq \frac{1}{L}\right\} \right)
\ll \frac{1}{r L^3},
\end{align*}
and 
\begin{align*}
&\Vol_2\left(\left\{ \ul{x} \in [0,1)^2 \,  : \,  s_2(a(t)\Lambda_{\ul{x},r}) \leq \frac{1}{L}\right\} \right)
\ll   \frac{1}{r^2 L^3} + \frac{e^{-\lfloor t \rfloor}}{rL^2}.
\end{align*}
Furthermore, by Lemma \ref{Lemma_s3}, the last set in the sum is empty if $L > \frac{1}{r}$. Combining these estimates, we obtain
the theorem.

\subsection{Proof of Corollary \ref{Cor_Heightintegrals}}

Let $r>\rho$ and $t \in \bR^2_{+}$. We introduce the sets
\[
B_{t,r}(i) = \left\{ \ul{x} \in [0,1)^2 \,  : \,  e^{i-1} < \height(a(t)\Lambda_{\ul{x},r}) \leq e^i \right\},  \quad \textrm{for $i \in \bZ$}.
\]
By Theorem \ref{Thm_Upperbound},  
\[
\height(a(t)\Lambda_{\ul{x},r}) \leq 
\max\left(\frac{e^{t_1 + t_2}}{r},e^{-\lfloor t \rfloor} \right) \leq \max(1,\rho^{-1})e^{t_1 + t_2},  \quad \textrm{for all $\ul{x} \in [0,1)^2$},
\]
and thus $B_{t,r}(i)$ is empty for $i \geq t_1 + t_2 +\ln\max(1,\rho^{-1})+ 1$. 
Furthermore,  by Theorem \ref{Thm_Levelsets},  applied with $L = e^{i-1}$ for
$i \geq  \ln(\rho^{-1})+1$,  we have
\[
\Vol_2(B_{t,r}(i)) \ll \max(r^{-1},r^{-2}) e^{-3i}+  r^{-1}e^{-\lfloor t\rfloor}e^{-2i} .
\]
Hence,  since $\theta$ is increasing, for $\eta\ge e^{2}\rho^{-1}$,
\begin{align*}
\int_{M_{t,r}(\eta)} \theta(\height(a(t)\Lambda_{\ul{x},r})) \,  d\ul{x}
\leq &
\sum_{i = \lfloor \ln(\eta) \rfloor}^{\lceil t_1 + t_2+\ln\max(1,\rho^{-1}) \rceil} \theta(e^i)
\cdot \Vol_2(B_{t,r}(i)) \\[0.2cm]
\ll &
\max(r^{-1},r^{-2})   \left( \sum_{i=\lfloor \ln(\eta) \rfloor}^{\lceil t_1 + t_2+\ln\max(1,\rho^{-1}) \rceil} 
\theta(e^i) \cdot e^{-3i} \right) \\[0.2cm]
&+
r^{-1} \cdot \left( \sum_{i=\lfloor \ln(\eta) \rfloor}^{\lceil t_1 + t_2 +\ln\max(1,\rho^{-1}) \rceil}
\theta(e^i) \cdot e^{-2i} \right) \cdot 
e^{-\lfloor t \rfloor}.
\end{align*}
By assumption,  $u \mapsto \frac{\theta(u)}{u^2}$ is decreasing for $u \geq 1$,  and thus
\[
\sum_{i=\lfloor \ln(\eta) \rfloor}^{\lceil t_1 + t_2+\ln\max(1,\rho^{-1}) \rceil} 
\theta(e^i) \cdot e^{-2i} \leq \int_{\lfloor \ln(\eta) \rfloor-1}^{t_1 + t_2+\ln\max(1,\rho^{-1}) + 1} \frac{\theta(e^u)}{e^{2u}} \,  du \leq \int_{ e^{-2}\eta}^{e^{t_1 + t_2 + 1}\max(1,\rho^{-1})} \frac{\theta(u)}{u^3} \,  du.
\]
Also, in the summation range,  we have
\[
\sum_{i=\lfloor \ln(\eta) \rfloor}^{\lceil t_1 + t_2+\ln\max(1,\rho^{-1}) \rceil} 
\theta(e^i) \cdot e^{-3i} \ll \eta^{-1} \cdot 
\int_{ e^{-2}\eta}^{e^{t_1 + t_2 + 1}\max(1,\rho^{-1})} \frac{\theta(u)}{u^3} \,  du.
\]
We conclude that
\[
\int_{M_{t,r}(\eta)} \theta(\height(a(t)\Lambda_{\ul{x},r})) \,  d\ul{x}
\ll \left( \max(r^{-1},r^{-2}) \eta^{-1} + r^{-1}e^{-\lfloor t \rfloor} \right) \cdot 
\int_{ e^{-2}\eta}^{e^{t_1 + t_2 + 1}\max(1,\rho^{-1})} \frac{\theta(u)}{u^3} \,  du,
\]
which finishes the proof.

\section{Correlations between the number of shifted lattice points in boxes}
\label{Sec_Cor}

If $B \subset \bR^2$ is a bounded Borel set,  we define the counting function
\begin{equation}
\label{Def_NB}
N_B(\ul{x}) = \left|\left(\bZ^2 + \ul{x}\right) \cap B\right|,  \quad \textrm{for $\ul{x} \in \bR^2$}.
\end{equation}
Note that $N_B$ is $\bZ^2$-periodic,  and thus completely determined by its values on
$[0,1)^2$.  \\

Our main result in this section reads as follows.
\begin{lemma}
	\label{Lemma_MeanDt}
	Let $M > 0$ and let $D^{(1)}$ and $D^{(2)}$ be Borel subsets of the square $[\shortminus M,M]^2$.  
	Let $t = (t_1,t_2) \in \bR^2_{+}$ with $t_1 \leq t_2$ and define
	\[
	D^{(i)}_t = 
	\left(
	\begin{matrix}
	e^{-t_1} & 0 \\
	0 & e^{-t_2}
	\end{matrix}
	\right) D^{(i)} 
	\quad \textrm{for $i=1,2$}.
	\]
	Then,  for all $(q_1,q_2) \in \bN^2$,  
	\[
	\int_{[0,1)^2} N_{D^{(1)}_t}(q_1\ul{x}) N_{D^{(2)}_t}(q_2\ul{x}) \,  d\ul{x} \leq F_t\left( \frac{\max(q_1,q_2)}{\gcd(q_1,q_2)} \right) \cdot \max\left(\Vol_2(D^{(1)}),\Vol_2(D^{(2)})\right),
	\]
	where $F_t$ is defined as in \eqref{Def_Ft}.
\end{lemma}

We will derive Lemma \ref{Lemma_MeanDt} from the following general result.

\begin{lemma}
	\label{Lemma_MeanCount}
	Let $B_1$ and $B_2$ be bounded Borel sets in $\bR^2$ and let $q_1$ and $q_2$
	be relatively prime positive integers.  Then, 
	\[
	\int_{[0,1)^2} N_{B_1}(q_1 \ul{x}) N_{B_2}(q_2 \ul{x}) \,  d\ul{x} \leq 
	\left| \bZ^2 \cap \left(q_2 B_1 - q_1 B_2\right) \right| \cdot \min\left( \frac{\Vol_2(B_1)}{q_1^2},\frac{\Vol_2(B_2)}{q_2^2} \right).
	\]
\end{lemma}

\begin{proof}
	Note that $\bZ^2 \times \bZ^2 = \bigsqcup_{\ul{k} \in \bZ^2} E_{\ul{k}}(q_1,q_2)$,  where
	\[
	E_{\ul{k}}(q_1,q_2) := \big\{ (\ul{p}_1,\ul{p}_2) \in \bZ^2 \times \bZ^2 \,  : \,  q_2 \ul{p}_1 - q_1 \ul{p}_2 = \ul{k} \big\},
	\quad \textrm{for $\ul{k} \in \bZ^2$}.
	\]
	Hence,  
	\vspace{0.1cm}
	\begin{align*}
	\int_{[0,1)^2} N_{B_1}(q_1\ul{x}) \,  N_{B_2}(q_2\ul{x}) \,  d\ul{x}
	&= 
	\sum_{\ul{p}_1,  \ul{p}_2} \int_{[0,1)^2} 
	\chi_{B_1}(\ul{p}_1 + q_1 \ul{x}) \,  \chi_{B_2}(\ul{p}_2 + q_2 \ul{x}) \,  d\ul{x} \\[0.2cm]
	&=
	\sum_{\ul{k}} C_{\ul{k}}(q_1,q_2),
	\end{align*}
	where $C_{\ul{k}}(q_1,q_2) := \sum_{(\ul{p}_1,\ul{p}_2)  \in E_{\ul{k}}(q_1,q_2)} \int_{[0,1)^2} 
	\chi_{B_1}(\ul{p}_1 + q_1 \ul{x}) \,  \chi_{B_2}(\ul{p}_2 + q_2 \ul{x}) \,  d\ul{x}$. \\
	
	Fix $\ul{k} \in \bZ^2$ and $(\ul{p}_1',\ul{p}_2') \in E_{\ul{k}}(q_1,q_2)$.  Then,  
	\begin{equation}
	\label{Ek}
	E_{\ul{k}}(q_1,q_2) = \left\{ \left(\ul{p}_1' + q_1\ul{l},\ul{p}_2' + q_2 \ul{l}\right) \, : \,  \ul{l} \in \bZ^2 \right\}.
	\end{equation}
	Indeed,  if $(\ul{p}_1,\ul{p}_2)$ is any point in $E_{\ul{k}}(q_1,q_2)$,  then $q_2(\ul{p}_1-\ul{p}_1') = q_1(\ul{p}_2-\ul{p}_2')$.  Since $q_1$ and $q_2$ are relatively prime integers,  we must have $\ul{p}_1 - \ul{p}_2' = q_1 \ul{l}$ and $\ul{p}_2 - \ul{p}_2' = q_2 \ul{l}$ for some (unique) $\ul{l} \in \bZ^2$,  thus proving \eqref{Ek}.  \\
	
	Hence,  for a fixed choice of $(\ul{p}_1',\ul{p}_2') \in E_{\ul{k}}(q_1,q_2)$,  the identity \eqref{Ek} allows us to rewrite the term $C_{\ul{k}}(q_1,q_2)$ as follows:  
	\begin{align*}
	C_{\ul{k}}(q_1,q_2) &= 
	\sum_{\ul{l}} \int_{[0,1)^2} \chi_{B_1}(\ul{p}_1' + q_1(\ul{l} + \ul{x}))
	\chi_{B_2}(\ul{p}_2' + q_2(\ul{l} + \ul{x})) \,  d\ul{x}  \\[0.1cm]
	&=
	\int_{\bR^2} \chi_{B_1}(\ul{p}_1' + q_1\ul{x}) \chi_{B_2}(\ul{p}_2' + q_2 \ul{x}) \,  d\ul{x} \\[0.2cm]
	&=
	\int_{\bR^2} \chi_{B_1}(\ul{p}_1' - q_1\ul{p}_2'/q_2 + q_1 \ul{x}) \chi_{B_2}(q_2\ul{x}) \,  d\ul{x}
	\\[0.2cm]
	&= \int_{\bR^2} \chi_{B_1}(\ul{k}/q_2 + q_1 \ul{x}) \chi_{B_2}(q_2\ul{x}) \,  d\ul{x}, 
	\end{align*}
	where, in the last step, we have used the fact that $q_2 \ul{p}_1' - q_1 \ul{p}_2' = \ul{k}$.  We conclude
	that
	\begin{align*}
	C_{\ul{k}}(q_1,q_2) &=
	\Vol_2\left( \left(\frac{1}{q_1}B_1 - \frac{\ul{k}}{q_1q_2} \right) \cap \frac{1}{q_2}B_2\right) \\[0.2cm]
	&= 
	\left( \frac{1}{q_1q_2} \right)^2 \cdot \Vol_2\left( \left(q_2 B_1 - \ul{k} \right) \cap q_1 B_2 \right).
	\end{align*}
	In particular,  
	\begin{align*}
	C_{\ul{k}}(q_1,q_2)
	&\leq \left( \frac{1}{q_1q_2} \right)^2 \cdot \min\left( \Vol_2\left(q_2 B_1\right),  \Vol_2\left(q_1 B_2\right)\right) \\[0.2cm]
	&= \min\left(\frac{\Vol_2(B_1)}{q_1^2}, \frac{\Vol_2(B_2)}{q_2^2} \right) \quad \textrm{for all $\ul{k} \in \bZ^2$},  
	\end{align*}
	and
	\[
	C_{\ul{k}}(q_1,q_2) = 0,  \quad \textrm{for all $\ul{k} \notin q_2 B_1 - q_1 B_2$}.
	\]
	Hence,  
	\[
	\sum_{\ul{k} \in \bZ^2} C_{\ul{k}}(q_1,q_2)
	\leq \left| \bZ^2 \cap \left(q_2 B_1 - q_1 B_2\right)\right| \cdot
	\min\left(\frac{\Vol_2(B_1)}{q_1^2}, \frac{\Vol_2(B_2)}{q_2^2} \right),
	\]
	which finishes the proof. 
\end{proof}

\begin{proof}[Proof of Lemma \ref{Lemma_MeanDt}]
Fix $(q_1,q_2) \in \bN^2$ and write $q_1 = s q'_1$ and $q_2 = s q_2'$,  where
$s = \gcd(q_1,q_2)$ and $q_1'$ and $q_2'$ are relatively prime.  Since multiplication
by positive integers on the torus $\bR^2/\bZ^2$ preserves the Lebesgue measure,  we have
\[
\int_{[0,1)^2} N_{D^{(1)}_t}(q_1\ul{x}) N_{D^{(2)}_t}(q_2\ul{x}) \,  d\ul{x}
= 
\int_{[0,1)^2} N_{D^{(1)}_t}(q'_1\ul{x}) N_{D^{(2)}_t}(q'_2\ul{x}) \,  d\ul{x}.
\]
If we apply Lemma \ref{Lemma_MeanCount} to the right-hand side with 
\[
B_1 = D^{(1)}_t \qand B_2 = D^{(2)}_t,
\]
and note that $q'_2 D^{(1)}_t - q'_1 D^{(2)}_t \subset [\shortminus M_1,M_1] \times [\shortminus M_2,M_2]$,  where 
\[
M_i = 2 \max(q_1',q_2') \cdot M \cdot e^{-t_i} \quad \textrm{for $i = 1,2$}.
\]
we get
\vspace{0.1cm}
\begin{align*}
\int_{[0,1)^2} N_{D^{(1)}_t}(q'_1\ul{x}) N_{D^{(2)}_t}(q'_2\ul{x}) \,  d\ul{x}
&\leq
\left| \bZ^2 \cap \left(q'_2 D^{(1)}_t - q'_1 D^{(2)}_t\right) \right| \cdot 
\min\left( \frac{\Vol_2\left(D^{(1)}_t\right)}{(q'_1)^2},\frac{\Vol_2\left(D^{(2)}_t\right)}{(q'_2)^2} \right) \\[0.2cm]
&\leq 
\frac{G(M_1,M_2)}{\max(q'_1,q_2')^2} \cdot e^{-(t_1 + t_2)} \cdot \max\left(\Vol_2\left(D^{(1)}\right),\Vol_2\left(D^{(2)}\right)\right),
\end{align*}
where $G$ is defined as in \eqref{Def_G}.  Hence, 
\[
\int_{[0,1)^2} N_{D^{(1)}_t}(q'_1\ul{x}) N_{D^{(2)}_t}(q'_2\ul{x}) \,  d\ul{x} \leq F_t(\max(q_1',q_2')) \cdot
\max\left(\Vol_2(D^{(1)}),\Vol_2(D^{(2)})\right),
\]
where $F_t$ is defined as in \eqref{Def_Ft}.  This finishes the proof.
\end{proof}


\section{Mean counting within controlled sets}
\label{Sec:MeanCnt}

In this section we prove $L^2$-bounds for Siegel transforms of indicator functions of controlled sets.  These bounds will be useful later in Section \ref{Sec:Smooth} when we
analyze smooth approximations of counting functions.  

\begin{lemma}
	\label{Lemma_L2Siegel}
	Let $M>1$ and $0 < \eps <3\eps < \gamma<1$. 
	We suppose that
	$E \subset \bR^2 \times \bR$ is an $(\eps,\gamma,M)$-controlled set.
	Then,  for all $t = (t_1,t_2) \in \bR^2_{+}$ such that
	\[
	t_1 + t_2 >\max(1,-\ln\left({\gamma}/{2}\right)),
	\]
	we have
	\[
	\int_{[0,1)^2} \widehat{\chi}_E(a(t)\Lambda_{\ul{x}})^2 \,  d\ul{x}
	\ll_M e^{-(t_1 + t_2)} + \max\left(\eps,-\frac{\eps}{\gamma} \ln\left(\frac{\eps}{\gamma}\right)\right) \cdot \max(1,(t_1 + t_2))^2,
	\]
	where the implicit constants depend only on $M$.
\end{lemma}

\begin{proof}
Let $E \subset \bR^2 \times \bR$ be a bounded Borel set,  let $E(t) = a(t)^{-1}E$,  and note that
\begin{align*}
\int_{[0,1)^2} \widehat{\chi}_E(a(t)\Lambda_{\ul{x}})^2 \,  d\ul{x}
&\leq \sum_{q_1,q_2 \in \bZ} \, \sum_{\ul{p}_1,\ul{p}_2 \in \bZ^2} 
\int_{[0,1)^2} \chi_{E(t)^{q_1}}(\ul{p}_1 + q_1\ul{x}) \chi_{E(t)^{q_2}}(\ul{p}_2 + q_2\ul{x}) \,  d\ul{x} \\[0.2cm]
&=
\sum_{q_1,q_2 \in \bZ} \int_{[0,1)^2}N_{D_{q_1}}(q_1 \ul{x}) N_{D_{q_2}}(q_2\ul{x}) \, d\ul{x},
\end{align*}
where $N_{\bullet}$ is defined as in \eqref{Def_NB} and $D_q = E(t)^{q}$ for $q \in \bZ$.  We define
\[
J := \{ y \in \bR \,  : \,  E^y \neq \emptyset \},
\]
and note that for every $q \in \bZ$, 
\[
E(t)^q = \left(
\begin{matrix}
e^{-t_1} & \\
& e^{-t_2}
\end{matrix}
\right)E^{q(t)},
\]
where $q(t) = e^{-(t_1 + t_2)}q$.  In particular, 
\[
E(t)^q \neq \emptyset \iff q \in J_t := e^{t_1 + t_2}J.
\]
Hence,
\[
\int_{[0,1)^2} \widehat{\chi}_E(a(t)\Lambda_{\ul{x}})^2 \,  d\ul{x}
\ll 
\sum_{q_1,q_2 \in J_t} 
\int_{[0,1)^2}N_{D^{(1)}_t}(q_1\ul{x}) N_{D^{(2)}_t}(q_2\ul{x}) \, d\ul{x},
\]
where $D^{(i)} = E^{q_i(t)}$ for $i =1,2$,  and $D^{(i)}_t$ is defined as in Lemma \ref{Lemma_MeanDt}.  The same lemma now tells us that
\begin{align*}
\int_{[0,1)^2} \widehat{\chi}_E(a(t)\Lambda_{\ul{x}})^2 \,  d\ul{x}
&\ll 
\sum_{q_1,q_2 \in J_t} F_t\left(\frac{\max(q_1,q_2)}{\gcd(q_1,q_2)}\right) \cdot \max\left(\Vol_{2}\left(E^{q_1(t)}\right),\Vol_2\left(E^{q_2(t)}\right)\right) \\[0.2cm]
&\ll \left(\sum_{q_1,q_2 \in J_t} F_t\left(\frac{\max(q_1,q_2)}{\gcd(q_1,q_2)}\right)\right) \cdot \sup_{y \in J} \Vol_2(E^y),
\end{align*}
where $F_t$ is defined as in \eqref{Def_Ft}.  \\

The arguments up to this point have not made use of any special properties of $E$.  
In what follows,  we will fix $0 < \eps < 3\eps < \gamma$ and assume that 
$E$ is an $(\eps,\gamma,M)$-controlled set (see Definition \ref{Def_Control}).  
The analysis will depend on whether $E$ is type I or type II.\\

Let us first assume that $E$ is type I.  Then,
\[
J \subset (\gamma,M] \qand \sup_{y \in J} \Vol_2(E^y) \ll_M 
\max\left(\eps,-\frac{\eps}{\gamma} \ln\left(\frac{\eps}{\gamma}\right)\right).
\]
Furthermore,  by Lemma \ref{Lemma_AuxilliarySum} (with $\alpha = \gamma$ and $\beta = M$),  
\[
\sum_{q_1,q_2 \in J_t} F_t\left(\frac{\max(q_1,q_2)}{\gcd(q_1,q_2)}\right) \ll_M
e^{-(t_1 + t_2)} + \max(1,-\ln(\gamma)) \cdot \max(1,t_1 + t_2),
\]
provided that
\begin{equation}
\label{CondType1}
t_1 + t_2 >\max(1, -\ln(\gamma)).
\end{equation}
We conclude that if the conditions \eqref{CondType1} hold,  then
\begin{equation}
\label{EstType1}
\int_{[0,1)^2} \widehat{\chi}_E(a(t)\Lambda_{\ul{x}})^2 \,  d\ul{x}
\ll 
(t_1 + t_2)^2  \cdot 
\max\left(\eps,-\frac{\eps}{\gamma} \ln\left(\frac{\eps}{\gamma}\right)\right).
\end{equation}

\vspace{0.2cm}

Let us now assume that $E$ is type II.  Then,
\[
E \subset [-M,M]^2 \times [\alpha,\beta],
\]
where $\frac{\gamma}{2} \leq \alpha$,  and $\beta - \alpha \ll \eps$.  In particular, 
\[
J \subset [\alpha,\beta] \qand \sup_{y \in J} \Vol_2(E^y) \ll M^2.
\]
By Lemma \ref{Lemma_AuxilliarySum},
\begin{align*}
\sum_{q_1,q_2 \in J_t} F_t\left(\frac{\max(q_1,q_2)}{\gcd(q_1,q_2)}\right) 
&\ll_M
e^{-(t_1+t_2)} + (\beta-\alpha) \cdot \max\left(1,\ln\left(\frac{\beta}{\alpha}\right)\right) \cdot \max(1,t_1 + t_2) \\[0.2cm]
&\ll_M
e^{-(t_1+t_2)} + \eps \cdot \max\left(1,\ln\left(\frac{\beta}{\alpha}\right)\right) \cdot 
\max(1,(t_1 + t_2)),
\end{align*}
provided that 
\begin{equation}
\label{CondTypeII}
t_1 + t_2 > -\ln(\alpha).
\end{equation}
Note that 
\[
\max\left(1,\ln\left(\frac{\beta}{\alpha}\right)\right) \leq 
\max\left(1,\frac{\beta-\alpha}{\alpha}\right).
\]
If $\beta \leq 2\alpha$,  the right-hand side is bounded from above by $1$.  Otherwise ,  the right-hand side is bounded from above by
\[
\frac{\beta-\alpha}{\alpha} \ll \frac{\eps}{\alpha} \ll \frac{\eps}{\gamma},
\]
since $\frac{\gamma}{2} \leq \alpha$.  Hence,
\begin{equation}
\label{EstType2}
\int_{[0,1)^2} \widehat{\chi}_E(a(t)\Lambda_{\ul{x}})^2 \,  d\ul{x}
\ll 
e^{-(t_1 + t_2)} + \max\left(\eps,\frac{\eps^2}{\gamma}\right) \cdot \max(1,(t_1 + t_2)),
\end{equation}
provided that \eqref{CondTypeII} hold.

\vspace{0.2cm}

Now we combine \eqref{EstType1} and \eqref{EstType2}.
Since $3\eps < \gamma$,  we have
\[
\frac{\eps^2}{\gamma} \leq -\frac{\eps}{\gamma} \ln\left(\frac{\eps}{\gamma}\right).
\]
Hence,  if we combine \eqref{EstType1} and \eqref{EstType2},  we get the uniform estimate (independent of whether $E$ is type I or type II):
\[
\int_{[0,1)^2} \widehat{\chi}_E(a(t)\Lambda_{\ul{x}})^2 \,  d\ul{x}
\ll e^{-(t_1 + t_2)} + \max\left(\eps,-\frac{\eps}{\gamma} \ln\left(\frac{\eps}{\gamma}\right)\right) \cdot \max(1,(t_1 + t_2))^2,
\]
provided that \eqref{CondType1} and \eqref{CondTypeII} both hold.   Note that since
$\frac{\gamma}{2} \leq \alpha$,  the second conditions in \eqref{CondType1} and \eqref{CondTypeII} are both satisfied if 
\[
t_1 + t_2 > \max(1,-\ln\left({\gamma}/{2}\right)).
\]
This finishes the proof.
\end{proof}


\section{Smooth approximations}
\label{Sec:Smooth}

Let $\cL_3$ denote the space of unimodular lattices in $\bR^3$.  We can identify 
$\cL_3$ with the homogeneous space $\SL_3(\bR)/\SL_3(\bZ)$ via the map $g \SL_3(\bZ) \mapsto g\bZ^3$ .  Fix a basis $\{Y_1,\ldots,Y_8\}$ of the Lie algebra 
$\mathfrak{sl}_3(\bR)$.  We adopt the following slight abuse of notation: for every 
$i = 1,\ldots,8$,  let $D_i$ both denote the differential operator 
\[
D_i \rho(g) = \frac{d}{dt} \rho(\exp(tY_i)g)|_{t=0}
\]
on $C_b^\infty(\SL_3(\bR))$ and the differential operator
\[
D_i \varphi = \frac{d}{dt} \rho(\exp(tY_i)\Lambda)|_{t=0}
\]
on $C_b^\infty(\cL_3)$ (which we identify with the space of bounded and smooth right $\Gamma$-invariant functions on the group $\SL_3(\bR)$ via the map above).  Differential operators can clearly be composed,  and any composition of the $D_1,\ldots,D_8$
can be rewritten as a linear combination of compositions of the form 
$D_m := D_1^{m_1} \circ \cdots \circ D_8^{m_8}$ for some vector 
$m = (m_1,\ldots,m_8) \in \bN_o^8$ (with the convention that $D_0 = \id$). 
We define the norms
\[
\|\rho\|_{C_b^s(\SL_3(\bR)} := \max\{ \|D_m \rho\|_\infty \,  : \,  m_1 + \ldots + m_8 \leq s\},  \quad \textrm{for $\rho \in C_b^\infty(\SL_3(\bR))$}
\]
and
\[
\|\varphi\|_{C^s(\cL_3)} := \max\{ \|D_m \varphi\|_\infty \,  : \,  m_1 + \ldots + m_8 \leq s\},  \quad \textrm{for $\varphi \in C_b^\infty(\cL_3)$}.
\]
Fix a right-invariant Riemannian metric on $\SL_3(\bR)$,  and denote by $\Lip$ the corresponding Lipschitz semi-norm on $\cL_3$ (viewed as the right quotient space $\SL_3(\bR)/\SL_3(\bZ)$).  We define
\begin{equation}
\label{Def_Ss}
\cN_s(\varphi) := \max\big(\|\varphi\|_{C_b^s(\cL_3)},\Lip(\varphi)\big),  \quad \textrm{for $\varphi \in C^\infty(\cL_3)$}.
\end{equation}

For $\eps > 0$,  let $V_\eps$ denote 
the symmetric open neighborhoods around the identity in $\SL_3(\bR)$
defined in \eqref{eq:ve}.
For the rest of this paper,  we fix a non-negative smooth function $\rho_\eps$ on $\SL_3(\bR)$ whose support is
contained in $V_\eps$ and has integral one with respect to the Haar measure on $\SL_3(\bR)$.  We leave it to the reader to verify that $\rho_\eps$ can be chosen so that for every integer $s \geq 1$,  
there is an integer $\sigma_s > 0$ such that
\begin{equation}
\label{rhoeps}
\|\rho_\eps\|_{C_b^s(\SL_3(\bR))} \ll \eps^{-\sigma_s},  \quad \textrm{for all $\eps\in (0,1)$}
\end{equation}
and where the implicit constants do not depend on $\eps$.  \\

By Lemma \cite[Lemma 4.11]{BGCLTDA}, for every $L > 1$ there exists   
a smooth function $\eta_L : \cL_3 \ra [0,1]$ such that
\begin{equation}
\label{etaL}
\{ \height \leq L/2 \} \subset \{ \eta_L = 1 \} \subset \supp(\eta_L) \subset 
\{ \height \leq 2L \},
\end{equation}
where $\height$ is the height function on $\cL_3$ defined in \eqref{Def_ht},  with
the property that for every $s \geq 1$ and $m \in \bN_o^8$ such that 
$m_1 + \ldots + m_8 \leq s$,  we have
\begin{equation}
\label{DmetaL}
D_m \eta_L \ll_s 1,  \quad \textrm{for all $L \geq 1$},
\end{equation}
where the implicit constants only depend on $s$,  but not on $L$.  \\

If $F$ is a locally bounded Borel function on $\bR^3$ and $\rho \in C_c^\infty(G)$,  we denote by $\rho * F$ the (action) convolution
\[
(\rho * F)(u) := \int_{G} \rho(g) \,  F(g^{-1}u) \,  dm(g),  \quad \textrm{for $u \in \bR^3$},
\]
where $m$ is a (fixed) Haar measure on $\SL_3(\bR)$.  We note that
\begin{equation}
D_m(\rho * F) := (D_m \rho) * F,  \quad \textrm{for every $m \in \bN_o^8$}.
\end{equation}
Finally,  if $\varphi$ is a locally bounded function on $\cL_3$,  we write 
\[
(\rho * \varphi)(u) = \int_{G} \rho(g) \,  \varphi(g^{-1}\Lambda) \,  dm(g),  \quad \textrm{for $\Lambda \in \cL_3$}.
\]
Our first main result in this section now reads as follows.
\begin{lemma}
\label{Lemma_Smooth}
Let $B$ be a bounded subset of $\bR^3$ and let $s \geq 1$ be an integer.  Then,  for every bounded Borel function $f : \bR^3 \ra \bR$ such that $\{f\neq 0\} \subset B$ 
and for every $L > 1$ and $\eps\in (0,1)$, the Siegel transform $\hat{f}$ satisfies
\[
\cN_s((\rho_\eps * \widehat{f}) \cdot \eta_L) \ll_{B,s} \eps^{-\sigma_s} \cdot  L
\cdot \|f\|_\infty,
\]
where the implicit constants only depend on $B$ and $s$.
\end{lemma}

\subsection{Proof of Lemma \ref{Lemma_Smooth}}

We first establish pointwise upper bounds on $D_m (\rho_\eps *\widehat{f})$,  for
an arbitrary multi-index $m$.

\begin{lemma}
\label{Lemma_Cs}
Let $B$ be a bounded subset of $\bR^3$ and let $s \geq 1$ be an integer.  Then,  for every $0 < \eps < 1$, $L > 1$ and for every bounded Borel function $f : \bR^3 \ra \bR$ such that $\{f\neq 0\} \subset B$ and for every multi-index $m = (m_1,\ldots,m_8) \in \bN_o^8$ such that $m_1 + \ldots + m_8 \leq s$,  we have
\[
|D_m ((\rho_\eps *\widehat{f})\cdot \eta_L)| \ll_{B,s} \eps^{-\sigma_s} \cdot \|f\|_\infty \cdot  L,  \quad \textrm{for all $\Lambda \in \cL_3$},
\]
where the implicit constants only depend on $B$ and $s$.
\end{lemma}

\begin{proof}
Note that
\[
D_m(\rho_\eps * \widehat{f}) = (D_m \rho_\eps) * \widehat{f} = \reallywidehat{(D_m \rho_\eps) * f}.
\]
Hence,  by Lemma \ref{Lemma_Schmidt},  
\[
|(D_m(\rho_\eps * \widehat{f})(\Lambda)| \ll_B \|D_m(\rho_\eps) * f\|_\infty \cdot \height(\Lambda) \ll \|\rho_\eps\|_{C^s(\SL_3(\bR))}  \cdot \|f\|_\infty \cdot \height(\Lambda),
\]
for every $\Lambda \in \cL_3$ and $m = (m_1,\ldots,m_8)$ such that $m_1 + \ldots m_8 \leq s$,  where the implicit constants only depend on the
support of of $D_m(\rho_\eps) * f$,  which is contained in $V_\eps. B$ (and thus
in a ball of radius $2R$,  where $R$ is the smallest radius of a ball enclosing $B$).
By \eqref{rhoeps},  $\|\rho_\eps\|_{C_b^s(\SL_3(\bR))} \ll \eps^{-\sigma_s}$,  and we 
obtain the bound
\[
|D_m (\rho_\eps *\widehat{f})(\Lambda)| \ll_{B,s} \eps^{-\sigma_s} \cdot \|f\|_\infty \cdot  \hbox{ht}(\Lambda),  \quad \textrm{for all $\Lambda \in \cL_3$},
\]
Now the lemma follows from \eqref{DmetaL}.
\end{proof}

Let us now discuss Lipschitz semi-norms.  By definition,  if $\Lambda \in \cL_3$ and 
$\dist$ is a right-invariant distance on $\SL_3(\bR)$,  we can induce a (non-invariant) distance 
on $\cL_3$ by
\[
\dist(\Lambda,h.\Lambda) = \inf \left\{ \dist(\id,h\gamma) \,  : \,  \gamma \in \Stab_{\SL_3(\bR)}(\Lambda)  \right\},
\quad \textrm{for $h \in \SL_3(\bR)$}.
\]
The corresponding Lipschitz semi-norms on $C_b^\infty(\SL_3(\bR))$ and $C_b^\infty(\cL_3)$ are thus given by
\[
\Lip_{\SL_3(\bR)}(\rho) = \sup\left\{ \frac{|\rho(g) - \rho(hg)|}{\dist(\id,h)} \, : \,  g,h \in \SL_3(\bR),  \enskip h \neq \id \right\},
\]
for $\rho \in C_b^\infty(\SL_3(\bR))$ and
\[
\Lip_{\cL_3}(\varphi) = \sup\left\{ \frac{|\varphi(\Lambda) - \varphi(h.\Lambda)|}{\dist(\Lambda,h.\Lambda)} \, : \,  \Lambda \in \cL_3,  \enskip h \notin 
\Stab_{\SL_3(\bR)}(\Lambda) \right\},
\]
for $\varphi \in C_b^\infty(\cL_3)$.  

One checks that there is a constant $\theta > 0$ such that
\begin{equation}
\label{Lipeps}
\Lip_{\SL_3(\bR)}(\rho_\eps) \ll \eps^{-\theta},  
\end{equation}
with implicit constants that are independent of $\eps$.  Upon possibly enlarging $\sigma_s$,  we can (and will) always assume that $\theta < \sigma_s$.  We also leave 
to the reader to show that the function $\eta_{L}$ can be chosen so that
\begin{equation}
\label{etaLip}
\left|\eta_L(\Lambda)-\eta_L(h.\Lambda)\right| \ll \dist(\Lambda,h.\Lambda),
\end{equation}
with implicit constants that are independent of $L$.

\begin{lemma}
\label{Lemma_Lip}
Let $B$ be a bounded subset of $\bR^3$.  Then,  for every $0 < \eps < 1$ and for every bounded Borel function $f : \bR^3 \ra \bR$ such that $\{f\neq 0\} \subset B$ and $L \geq 1$,
\[
\Lip_{\cL_3}((\rho_\eps * \widehat{f}) \cdot \eta_L)  \ll_B \eps^{-\theta} \cdot L,
\]
where the implicit constants only depend on $B$.
\end{lemma}

\begin{proof}
We use in the proof that 
$$
\|\rho_\eps\|_{C_b^s(\SL_3(\bR))}\ll \eps^{-\theta}\qand \Lip_{\SL_3(\bR)}(\rho_\eps) \ll \eps^{-\theta}.
$$
For every $h \in G$ and $\Lambda \in \cL_3$,  we have
\begin{align*}
\left| (\rho_\eps * \widehat{f})(\Lambda) - (\rho_\eps * \widehat{f})(h.\Lambda)\right|
&\leq 
\int_{(V_\eps \cup h^{-1}V_\eps)} 
\left|\rho_\eps(g) - \rho_\eps(hg)\right| \cdot \widehat{f}(g^{-1}.\Lambda)
\,  dg \\[0.2cm]
&\leq \Lip_{\SL_3(\bR)}(\rho_\eps) \cdot \dist(\id,h) \cdot 
\int_{(V_\eps \cup h^{-1}V_\eps)} |\widehat{f}(g^{-1}.\Lambda)| \,  dm(g) \\[0.2cm]
&\ll \eps^{-\theta} \cdot \dist(\id,h) \cdot 
\int_{(V_\eps \cup h^{-1}V_\eps)} |\widehat{f}(g^{-1}.\Lambda)| \,  dm(g),
\end{align*}
where we in the last estimate have used \eqref{Lipeps}.  Note that the same computation goes through with $h_\gamma := h\gamma$ for $\gamma \in \Stab_{\SL_3(\bR)}(\Lambda)$,  and thus
\begin{equation}
\label{eq:new1}
\left| (\rho_\eps * \widehat{f})(\Lambda) - (\rho_\eps * \widehat{f})(h.\Lambda)\right|
\leq  \eps^{-\theta} \cdot \dist(\Lambda,h.\Lambda) \cdot 
\int_{(V_\eps \cup h_\gamma^{-1}V_\eps)} |\widehat{f}(g^{-1}.\Lambda)| \,  dm(g).
\end{equation}
We also note that 
\begin{align}
((\rho_\eps * \widehat{f}) \cdot \eta_L)(\Lambda) - ((\rho_\eps * \widehat{f}) \cdot \eta_L)(h.\Lambda)
=&
\left((\rho_\eps * \widehat{f})(\Lambda) -(\rho_\eps * \widehat{f})(h.\Lambda)\right) 
\cdot \eta_L(\Lambda) \nonumber \\[0.2cm]
&+ (\rho_\eps * \widehat{f})(h.\Lambda) \cdot \left(\eta_L(\Lambda)-\eta_L(h.\Lambda)\right).  \label{LipDiff}
\end{align}
We can without loss of generality assume that at least one of the points $\Lambda$
and $h.\Lambda$ belong to $\supp(\eta_L)$; otherwise the Lipschitz condition is trivially satisfied.  The rest of our analysis is now divided into two cases. 

\subsubsection*{\textbf{Case I:} \textit{$\Lambda,  h.\Lambda  \in \supp(\eta_L)$}, so that $\hbox{\rm ht}(\Lambda),\hbox{\rm ht}(h\Lambda)\le 2L$.}

By Lemma \ref{Lemma_Schmidt} and \eqref{upperbnd_ht},  we have
\[
|\widehat{f}(g^{-1}h.\Lambda)| \ll_{B} \|f\|_\infty \cdot \height(g^{-1}h.\Lambda) 
\leq \|f\|_\infty \cdot \|g\|_{\textrm{op}} \cdot \height(h.\Lambda) ,  \quad \textrm{for all $g,h \in \SL_3(\bR)$},
\]
and thus,
\begin{align}
\int_{(V_\eps \cup h_\gamma^{-1}V_\eps)} |\widehat{f}(g^{-1}.\Lambda)| \,  dg
&\leq \int_{V_\eps} |\widehat{f}(g^{-1}.\Lambda)| \,  dg + 
\int_{V_\eps} |\widehat{f}(g^{-1}h.\Lambda)| \,  dg\nonumber \\[0.2cm]
&\ll_B (1+\eps) \cdot \|f\|_\infty \cdot \height(\Lambda) + (1+\eps) 
\cdot \height(h.\Lambda)\nonumber \\[0.2cm]
&\ll \|f\|_\infty \cdot \max(\height(\Lambda),\height(h.\Lambda)),\label{eq:new2}
\end{align}
for all $g \in V_\eps$ and for all $h \in \SL_3(\bR)$,  where the implicit constants only depend on $B$ (and not on $\eps \in (0,1)$).  Hence, we deduce from (\ref{eq:new1}) and (\ref{eq:new2}) that the first term on the right-hand side in \eqref{LipDiff} is bounded above in absolute value by
\begin{equation}
\label{Lip1}
\ll \|f\|_\infty \cdot \max(\height(\Lambda),\height(h.\Lambda)) \cdot \eta_L(\Lambda)  \cdot \dist(\Lambda,h.\Lambda) \cdot \eps^{-\theta},
\end{equation}
while the second term is bounded above in absolute value by
\begin{equation}
\label{Lip2}
\ll \|f\|_\infty \cdot \height(h.\Lambda) \cdot  \left|\eta_L(\Lambda)-\eta_L(h.\Lambda)\right| \ll \|f\|_\infty \cdot \height(h.\Lambda) \cdot \eps^{-\theta} \cdot \dist(\Lambda,h.\Lambda),
\end{equation}
where we in the last inequality have used \eqref{etaLip}.  By our assumption,  $\height(\Lambda) \leq 2L$ and $\height(h.\Lambda) \leq 2L$,  so we conclude that
\[
\left| (\rho_\eps * \widehat{f})(\Lambda) - (\rho_\eps * \widehat{f})(h.\Lambda)\right|
\ll \|f\|_\infty \cdot \eps^{-\theta} \cdot L \cdot \dist(\Lambda,h.\Lambda).
\]

\subsubsection*{\textbf{Case II:} $\Lambda \in \supp(\eta_L)$ \textit{but $h.\Lambda \notin \supp(\eta_L)$ or the opposite}.}

We split this case into two sub-cases.  Let us first assume that $\dist(\Lambda,h.\Lambda) \geq 1$.  Then,  by the same analysis as above, 
\begin{align*}
|(\rho_\eps * \widehat{f})(\Lambda) \eta_L(\Lambda)| 
&\ll \eps^{-\theta} \cdot \int_{V_\eps} |\widehat{f}(g^{-1}.\Lambda)| \,  dm(g) \\[0.2cm]
&\ll_B \eps^{-\theta} \cdot \|f\|_\infty \cdot \height(\Lambda) \leq \eps^{-\theta} \cdot \|f\|_\infty \cdot L \\[0.2cm]
&\ll \eps^{-\theta} \cdot \|f\|_\infty \cdot L \cdot \dist(\Lambda,h.\Lambda).
\end{align*}
If $\dist(\Lambda,h.\Lambda) \leq 1$,  we can choose $h_\gamma$ such that $\dist(\id,h_\gamma) \leq 1$,  and hence there is a compact set $K$,  which is independent of $\eps$ (and $h$ as long as 
$\dist(\Lambda,h.\Lambda) \leq 1$),  with the property that
\[
V_\eps \cup h_\gamma^{-1}V_\eps \subset K.
\]
Hence,  in this case,  by \eqref{Lipeps}, Lemma \ref{Lemma_Schmidt} and \eqref{upperbnd_ht},  
\[
\left| (\rho_\eps * \widehat{f})(\Lambda) - (\rho_\eps * \widehat{f})(h.\Lambda)\right|
\ll_B \eps^{-\theta} \cdot \dist(\Lambda,h.\Lambda) \cdot \|f\|_\infty \cdot  
\left( \int_{K} \|g^{-1}\|_{\textrm{op}}  \,  dm(g) \right) \cdot \height(\Lambda).
\]
Similarly, when $\dist(\Lambda,h.\Lambda) \leq 1$, 
\[
|(\rho_\eps * \widehat{f})(h.\Lambda) | \ll_K \eps^{-\theta} \cdot \|f\|_\infty \cdot \height(\Lambda).
\]
We conclude that in the difference \eqref{LipDiff} both terms are bounded above in absolute value by
\[
\ll \|f\|_\infty \cdot \eps^{-\theta} \cdot L \cdot \dist(\Lambda,h.\Lambda),
\]
where we for the second term have used \eqref{etaLip}.  In both sub-cases,  we see that
\[
\left| (\rho_\eps * \widehat{f})(\Lambda) \eta_L(\Lambda) - (\rho_\eps * \widehat{f})(h.\Lambda) \eta_L(h.\Lambda) \right|
\ll \|f\|_\infty \cdot \eps^{-\theta} \cdot L \cdot \dist(\Lambda,h.\Lambda).
\]
The opposite case is handled in the same way by interchanging $\Lambda$ and $h.\Lambda$,  and we are done.
\end{proof}

The proof of Lemma \ref{Lemma_Smooth}  follows upon combining Lemma \ref{Lemma_Cs} and Lemma \ref{Lemma_Lip}.

\subsection{Smooth approximations of counting functions}

In this subsection we discuss smooth approximation of the counting functions that come out of our tessellation scheme in Section \ref{Sec:Tesselate}.  We begin by recalling the notation. We have
\begin{equation}
\label{Def_alphabetaT}
\alpha_\Omega =  \ln\left(\frac{T_0\ct^2}{\bt e^2} \right), \quad \beta_\Omega = 
\ln\left(\frac{T\ct^2}{\at}\right), 
\end{equation}
and
\[
\cF_\Omega = \left\{ n \in \bN_o^2 \,  : \,  \alpha_\Omega
\leq n_1 + n_2 < \beta_\Omega \right\}.
\]
We also assume that 
\begin{equation}
\label{eq:ccc}
\zeta\cdot \bt\le  \ct^2\qand \ct<1/2
\end{equation}
for some $\zeta>0$. Without loss of generality, $\zeta<1$.

For $n \in \cF_\Omega$ and a \emph{bounded} measurable function $h : [0,\infty) \ra \bR$,  we define
\[
h_{\Omega,n}(\ul{u},y) := h\left(\frac{e^{n_1 + n_2} \cdot y}{T}\right) \,  \chi_{\Delta_{\Omega,n}}(\ul{u},y),
\quad \textrm{for $(\ul{u},y) \in \bR^2 \times \bR$},
\]
where $\Delta_{\Omega,n}$ is defined in \eqref{Def_DeltaTn}.
By Lemma \ref{Lemma_tesselate},
\begin{equation}\label{eq:d}
\Delta_{\Omega,n} 
\subset
\left[\shortminus \ct,\ct\right]^2
\times 
\left(\frac{\at}{\ct^2},\frac{\bt e^2}{\ct^2}\right]\subset [-1/2,1/2]^2\times (0,e^2\zeta^{-1}].
\end{equation}
Let
\[
\phi_{\Omega,n} := \widehat{h}_{\Omega,n} - \int_{\cL_3} \widehat{h}_{\Omega,n}  d\mu,
\]
where $\mu$ is the unique $\SL_3(\bR)$-invariant probability measure on $\cL_3$.  \\

Let $\eps_T$ be a decreasing function,  which converges to zero as
$T \ra \infty$,  and let $\rho_{\eps_T}$ be as in the previous subsection (with $\eps = \eps_T$).  
Let $L_T$ be an increasing function,  which tends to infinity with $T$,  and  
define
\[
f_{\Omega,n}(\ul{u},y) := (\rho_{\eps_T} * h_{\Omega,n})(\ul{u},y),  \quad \textrm{for $(\ul{u},y) \in \bR^2 \times \bR$},
\]
and
\begin{equation}
\label{def_varphiTn}
\varphi_{\Omega,n} := \widehat{f}_{\Omega,n} \cdot \eta_{L_T} - \int_{\cL_3} \widehat{f}_{\Omega,n} \cdot \eta_{L_T} \,  d\mu.
\end{equation}
It follows from \eqref{eq:d} that 
\begin{equation}\label{eq:f}
\{ f_{\Omega,n} \neq 0 \} \subset [-1,1]^2 \times [-1,e^2\zeta^{-1}+1]
\end{equation}
for all $n$ and for all sufficiently large $T$,  so that by Theorem \ref{Thm_Siegel},
\begin{equation}
\label{eq:new3}
\int_{\cL_3} \widehat{f}_{\Omega,n} \cdot \eta_L \,  d\mu \leq 
\int_{\cL_3} \widehat{f}_{\Omega,n}  \,  d\mu = \int_{\bR^2 \times \bR} f_{\Omega,n}(\ul{u},y) \,  d\ul{u} dy\ll_\zeta 1,
\end{equation}
with implicit constants that are independent of $T$,  $\eps_T$ and $n$. Then, (\ref{def_varphiTn}), (\ref{eq:f}), (\ref{eq:new3}), and    
Lemma \ref{Lemma_Smooth} immediately imply the following result.
\begin{lemma}
\label{Lemma_Approx0}
For every $s \geq 1$ and for all sufficiently large $T$,  
\[
\sup_{n \in \cF_\Omega} \cN_s(\varphi_{\Omega,n}) \ll_s \eps_T^{-\sigma_s} \cdot L_T,
\]
where the implicit constants only depend on $\|h\|_\infty$
and $s$,  and not on $T$.
\end{lemma}

Let us now assume that $h$ is a bounded Lipschitz continuous function on $[0,\infty)$.
The rest of this section is devoted to the proof of the following lemma which roughly 
asserts that the smooth approximations of $\widehat{h}_{\Omega,n}$ above are good in the
$L^2(\nu)$-sense along $a$-orbits, where $\nu$ is the unique $\bR^2$-invariant measure on the torus $\cY_3 := \{ \Lambda_{\ul{x}} \,  : \,  \ul{x} \in [0,1)^2 \} \subset \cL_3$.

\begin{lemma}
\label{Lemma_Approx1}
Suppose that 
$$
\eps_T < \at \zeta^2/100\qand L_T \geq 2e^{2}\zeta^{-1}.
$$
Then, for all $n \in \bN^2$ such that $n_1+n_2> \max(1,-\ln\left({\at}/{2}\right))$, we have
\begin{align*}
\big\|(\phi_{\Omega,n}  - \varphi_{\Omega,n}) \circ a(n) \big\|_{L^2(\nu)} 
\ll_{h,\zeta}&
\frac{\eps_T}{\at} \cdot \max(1,n_1 + n_2)^{1/2} +
e^{-\frac{(n_1 + n_2)}{2}}\\[0.2cm]
& +  \max\Big(\eps_T,-\frac{\eps_T}{\at} \ln\Big(\frac{\eps_T}{\at} \Big)\Big)^{1/2} \cdot  \max(1,(n_1 + n_2)) \\[0.2cm]
&+
\left(L_T^{-1/2} + e^{-\frac{\lfloor n \rfloor}{2}} \right)\cdot \max(1,n_1 + n_2)^{1/2}.
\end{align*}
\end{lemma}

\begin{remark}
The implicit function in the lemma depends only on $\|h\|_\infty$
and $\sup_{y\neq y'} \frac{|h(y)-h(y')|}{|y-y'|}$.
\end{remark}

\begin{proof}
We observe that 
\begin{equation}
\label{split0}
\big\|(\phi_{\Omega,n}  - \varphi_{\Omega,n}) \circ a(n) \big\|_{L^2(\nu)} 
\le \big\| (\widehat{h}_{\Omega,n}  - \widehat{f}_{\Omega,n} \cdot \eta_{L_T}) \circ a(n) \big\|_{L^2(\nu)}+\big\| \widehat{h}_{\Omega,n}  - \widehat{f}_{\Omega,n} \cdot \eta_{L_T} \big\|_{L^1(\mu)},
\end{equation}
and
\begin{equation}
\label{split1}
(\widehat{h}_{\Omega,n}  - \widehat{f}_{\Omega,n} \cdot \eta_{L_T}) \circ a(n)
= (\widehat{h}_{\Omega,n} - \widehat{f}_{\Omega,n}) \circ a(n) + \left(\widehat{f}_{\Omega,n} \cdot (1-\eta_{L_T})\right) \circ a(n).
\end{equation}
We estimate each of the the above terms separately.\\

First, we proceed with the estimate of $\big\|(\widehat{h}_{\Omega,n}  - \widehat{f}_{\Omega,n}) \circ a(n)\big\|_{L^2(\nu)}$. 
We recall that 
$$
h_{\Omega,n}(\ul{x},y)=h\left(\frac{e^{n_1 + n_2} \cdot y}{T}\right)
\cdot \chi_{\Delta_{\Omega,n}}(\ul{x},y),\quad (\ul{x},y)\in \bR^2\times \bR.
$$
For $g \in V_{\eps_T}$,  we write $g.(\ul{x},y) := (\ul{x}(g),y(g))$ as in \eqref{eq:action}.  Then
$$
h_{\Omega,n}(g.(\ul{x},y))=h\left(\frac{e^{n_1 + n_2} \cdot y(g)}{T}\right)
\cdot \chi_{\Delta_{\Omega,n}}(\ul{x}(g),y(g)),\quad (\ul{x},y)\in \bR^2\times \bR.
$$
It follows from \eqref{eq:d} that
\[
\max(\|\ul{x}-\ul{x}(g)\|_\infty,|y-y(g)|) \ll_\zeta \eps_T,  \quad \textrm{for all $g \in V_{\eps_T}$ }
\]
provided that either $(\ul{x},y) \in \Delta_{\Omega,n}$ or $(\ul{x}(g),y(g)) \in \Delta_{\Omega,n}$,
which we assume from now on.

We observe that $|h_{\Omega,n}-h_{\Omega,n}\circ g|$ can be bounded by
\[
\left| h\left(\frac{e^{n_1 + n_2} \cdot y}{T}\right) - h\left(\frac{e^{n_1 + n_2} \cdot y(g)}{T}\right)\right| \cdot \chi_{\Delta_{\Omega,n}}(\ul{x},y) + \left|h\left(\frac{e^{n_1 + n_2} y(g)}{T}\right)\right| \cdot 
\chi_{E_{\Omega,n}(g)},
\]
where 
\[
E_{\Omega,n}(g) := \left(g^{-1}\Delta_{\Omega,n} \setminus \Delta_{\Omega,n}\right) \cup 
\left(\Delta_{\Omega,n} \setminus g^{-1}\Delta_{\Omega,n}\right).
\]
Recall that $\Delta_{\Omega,n}= \emptyset$ unless $n\in \cF_\Omega$
so that we may assume that $n\in \cF_\Omega$,
thus 
$n_1 + n_2 \leq \ln\left(\frac{T\ct^2}{\at}\right)$.
Since $h$ is Lipschitz continuous, for $n \in \cF_\Omega$,  we have
\begin{align*}
\left| h\left(\frac{e^{n_1 + n_2} \cdot y}{T}\right) - h\left(\frac{e^{n_1 + n_2} \cdot y(g)}{T}\right)\right| 
&\leq \frac{e^{n_1+n_2}}{T} \cdot |y-y(g)| \cdot \Lip(h), \nonumber \\
&\ll_\zeta \frac{\ct^2 \cdot \eps_T}{\at} \cdot \Lip(h)\le 
\frac{\eps_T}{\at} \cdot \Lip(h).
\end{align*}
We note that the set $\Delta_{\Omega,n}$ is of the form \eqref{Def_Delta},  with
\[
a = \at,  \enskip b = \bt,  \enskip u_{i}^{-} = e^{-1}\ct ,  \enskip u_{i}^{+} = \ct,\hbox{ for $i=1,2$,}
\]
 and with $\gamma = \gamma_{\Omega,n}=T_0e^{-(n_1+n_2)}$ and $\delta = \delta_{\Omega,n}=Te^{-(n_1+n_2)}$, which, 
by \eqref{eq:d} satisfy the bounds
\[
2\at<\frac{\at}{\ct^2} < \gamma_{\Omega,n} \qand \delta_{\Omega,n} < \frac{\bt e^2}{\ct^2}\le e^2\zeta^{-1}
\]
Hence,  by Lemma \ref{Lemma_DeltaPert} (see also Remark \ref{Rmk_smallergamma}) and the fact that $\varepsilon_{T}\leq a_{T}\zeta^{2}/100$,  we can find  $(\eps_T,\at,e^2\zeta^{-1}+1)$-controlled sets $E_{\Omega,n}^{(s)}$,
$s=1,\ldots,24$, such that
\[
{E_{\Omega,n}(g)} \subset  \bigcup_s {E_{\Omega,n}^{(s)}},  \quad \textrm{for all $g \in V_{\eps_T}$}.
\]
Therefore, we conculde that for all $g\in V_\eps$ and $n\in\cF_\Omega$,
\[
|h_{\Omega,n}(\ul{x},y) - h_{\Omega,n}(g.(\ul{x},y))| \ll_\zeta \left(\frac{ \eps_T}{\at} \cdot \chi_{\Delta_{\Omega,n}}(\ul{x},y) + \sum_s \chi_{E_{\Omega,n}^{(s)}}(\ul{x},y)\right) \cdot \Lip(h),
\]
provided that either $(\ul{x},y) \in \Delta_{\Omega,n}$ or $(\ul{x}(g),y(g)) \in \Delta_{\Omega,n}$, In fact, this estimate holds for all $(\ul{x},y)\in \bR^2\times \bR$
since it holds trivially in the complementary set.
Since $f_{\Omega,n}=\rho_{\eps_T}*h_{\Omega,n}$, we obtain that 
\[
|h_{\Omega,n} - f_{\Omega,n}| \ll_\zeta \left(\frac{\eps_T}{\at} \cdot \chi_{\Delta_{\Omega,n}} + \sum_s \chi_{E_{\Omega,n}^{(s)}}\right) \cdot \Lip(h).
\]
for all $n\in\bN^2_o$. We also have the corresponding bounds for the Siegel transforms:
\begin{equation}\label{eq:fff}
|\widehat h_{\Omega,n} - \widehat f_{\Omega,n}| \ll_\zeta \left(\frac{\eps_T}{\at} \cdot \widehat \chi_{\Delta_{\Omega,n}} + \sum_s \widehat \chi_{E_{\Omega,n}^{(s)}}\right) \cdot \Lip(h).
\end{equation}
Therefore,
\begin{align}
\big\|(\widehat h_{\Omega,n} - \widehat f_{\Omega,n})\circ a(n)\big\|_{L^2(\nu)} 
\ll_\zeta &
\Big( \frac{\eps_T}{\at} \cdot 
\big\|\widehat{\chi}_{\Delta_{\Omega,n}} \circ a(n)\big\|_{L^2(\nu)} \nonumber \\
&+ \sum _s \big\|\widehat{\chi}_{E_{\Omega,n}^{(s)}} \circ a(n)\big\|_{L^2(\nu)}\Big) \cdot \Lip(h).\label{eq:l}
\end{align}

It follows from \eqref{eq:d} and Lemma \ref{Lemma_Schmidt} that
\[
\big\|\widehat{\chi}_{\Delta_{\Omega,n}} \circ a(n)\big\|_{L^2(\nu)} \ll_\zeta \left( \int_{[0,1)^2} \height(a(n)\Lambda_{\ul{x}})^2 \,  d\ul{x} \right)^{1/2},
\]
where the implicit constants are independent of $\Omega$ and $n$.  Furthermore, 
Corollary \ref{Cor_Heightintegrals},  applied with $\theta(u) = u^2$ and   $\eta = e^2\zeta^{-1}$,  tells us
that
\[
\int_{[0,1)^2} \height(a(n)\Lambda_{\ul{x}})^2 \,  d\ul{x} \ll_\zeta \max(1,n_1+n_2),
\]
We conclude that
\begin{equation}
\label{first}
\big\|\widehat{\chi}_{\Delta_{\Omega,n}} \circ a(n)\big\|_{L^2(\nu)}
\ll \max(1,n_1 + n_2)^{1/2},
\end{equation}
for all $(n_1,n_2) \in \bN^2$.   

Next we estimate $\|\widehat{\chi}_{E_{\Omega,n}^{(s)}} \circ a(n)\|_{L^2(\nu)}$.  Since the sets $E_{\Omega,n}^{(s)}$ are $(\eps_T,\at,e^2\zeta^{-1}+1)$-controlled, 
it follows from Lemma \ref{Lemma_L2Siegel} that
\begin{align}
 \big\|\widehat{\chi}_{E_{\Omega,n}^{(s)}} \circ a(n)\big\|_{L^2(\nu)} &\ll_\zeta 
\left( e^{-(n_1 + n_2)} + \max\left(\eps_T,-\frac{\eps_T}{\at} \ln\left(\frac{\eps_T}{\at}\right)\right) \cdot \max(1,n_1 + n_2)^2 \right)^{1/2}\nonumber \\
&\ll e^{-\frac{(n_1 + n_2)}{2}} + 
\max\Big(\eps_T,-\frac{\eps_T}{\at} \ln\Big(\frac{\eps_T}{\at} \Big)\Big)^{1/2} \max(1,n_1 + n_2),\label{second}
\end{align}
where the implicit constants are independent of $n$ and $T$,  provided that
\begin{equation}\label{nrange}
n_1 + n_2 > \max(1,-\ln\left(\at/2\right)).
\end{equation}
If we now combine \eqref{first} and \eqref{second},  we get from \eqref{eq:l}
\begin{align}
\big\|(\widehat h_{\Omega,n} - \widehat f_{\Omega,n})\circ a(n)\big\|_{L^2(\nu)} 
\ll_{\zeta,h}& \frac{\eps_T}{\at} \cdot \max(1,n_1 + n_2)^{1/2} +
e^{-\frac{(n_1 + n_2)}{2}}\nonumber\\[0.2cm]
&+  \max\Big(\eps_T,-\frac{\eps_T}{\at} \ln\Big(\frac{\eps_T}{\at} \Big)\Big)^{1/2} \cdot  \max(1,(n_1 + n_2)),\label{s1}
\end{align}
for all $n\in\bN_o^2$ satisfying \eqref{nrange}.  
This provides an estimate of the first term in \eqref{split1}.\\

Now we estimate the term $\left\|\left(\widehat{f}_{\Omega,n} \cdot (1-\eta_{L_T})\right) \circ a(n) \right\|_{L^2(\nu)}$. 
It follows from \eqref{eq:f} and Lemma \ref{Lemma_Schmidt} that 
\begin{equation}
\label{eq:o1}
f_{\Omega,n}\ll_{h,\zeta} \hbox{ht}.
\end{equation}
Furthermore, 
\begin{equation}
\label{eq:o2}
\supp(1-\eta_{L_T}) \subset \left\{ \height \geq \frac{L_T}{2} \right\}.
\end{equation}
Thus, 
\[
\left\|\left(\widehat{f}_{\Omega,n} \cdot (1-\eta_{L_T})\right) \circ a(n) \right\|_{L^2(\nu)} \ll  \left(\int_{M_{T,n}} (\height(a(n)\Lambda_{\ul{x}}))^2 \,  d\ul{x} \right)^{1/2},
\]
where
\[
M_{T,n} := \left\{ \ul{x} \in [0,1)^2 \,  : \,  \height(a(n)\Lambda_{\ul{x}}) \geq \frac{L_T}{2}\right\}.
\]
Note that $M_{T,n} = M_{n,1}(L_T/2)$,  where the latter set is defined as in \eqref{Def_Mtr}.
Hence,  by Corollary \ref{Cor_Heightintegrals},  applied with $\theta(u) = u^2$,  we get
\begin{align*}
\int_{M_{T,n}} (\height(a(n)\Lambda_{\ul{x}}))^2 \,  d\ul{x} &\ll_\zeta \left(L_T^{-1} + e^{-\lfloor n \rfloor} \right) \cdot \int_{L_T e^{-2}/2}^{e^{n_1 + n_2 + 1}\max(1,\zeta^{-1})} \frac{du}{u} \\
&\ll_\zeta 
\left(L_T^{-1} + e^{-\lfloor n \rfloor} \right) \cdot \max(1,n_1 + n_2),
\end{align*}
for all $n \in \bN_o^2$,  provided that $L_T \geq 2e^{2}\zeta^{-1}$.
Hence, we conclude that 
\begin{equation}\label{s2}
\left\|\left(\widehat{f}_{\Omega,n} \cdot (1-\eta_{L_T})\right) \circ a(n) \right\|_{L^2(\nu)}
\ll_{h,\zeta} \left(L_T^{-1/2} + e^{-\lfloor n \rfloor/2} \right) \cdot \max(1,n_1 + n_2)^{1/2},
\end{equation}
This proves \eqref{split1}
and hence provides an estimate of the first term in \eqref{split0}. \\

Finally, we estimate the second term in \eqref{split0}.
We use that 
\begin{equation}\label{eq:last}
\big\|\widehat{h}_{\Omega,n} - \widehat{f}_{\Omega,n} \cdot \eta_{L_T}\big\|_{L^1(\mu)}
\le  \big\|\widehat{h}_{\Omega,n} - \widehat{f}_{\Omega,n}\big\|_{L^1(\mu)} + \big\|\widehat{f}_{\Omega,n} \cdot (1 - \eta_{L_T})\big\|_{L^1(\mu)}.
\end{equation}
By \eqref{eq:fff} and Theorem \ref{Thm_Siegel},
\begin{align*}
\big\|\widehat{h}_{\Omega,n} - \widehat{f}_{\Omega,n} \big\|_{L^1(\mu)}
&\ll_\zeta
\left(\frac{\eps_T}{\at} \cdot \int_{ \cL_3} \widehat{\chi}_{\Delta_{\Omega,n}} \,d\mu + \sum_s
\int_{ \cL_3}\widehat{\chi}_{E_{\Omega,n}^{(s)}} \,d\mu\right) \cdot \Lip(h)\\
&\ll_h
\frac{\eps_T}{\at} \cdot \hbox{vol}_3(\Delta_{\Omega,n})+\sum_s \hbox{vol}_3(E_{\Omega,n}^{(s)}).
\end{align*}
It follows from the definition of $\Delta_{\Omega,n}$ that 
$$
\hbox{vol}_3(\Delta_{\Omega,n})\ll \bt\ll_\zeta 1,
$$
and since the sets $E_{\Omega,n}^{(s)}$ are $(\eps_T,\at,e^2\zeta^{-1}+1)$-controlled,
$$
\hbox{vol}_3(E_{\Omega,n}^{(s)})\ll_\zeta \max\Big(\eps_T,-\frac{\eps_T}{\at} \ln\Big(\frac{\eps_T}{\at} \Big)\Big).
$$
Hence,
\begin{align*}
\big\|\widehat{h}_{\Omega,n} - \widehat{f}_{\Omega,n} \big\|_{L^1(\mu)}
&\ll_{\zeta,h}
\frac{\eps_T}{\at} +\max\Big(\eps_T,-\frac{\eps_T}{\at} \ln\Big(\frac{\eps_T}{\at} \Big)\Big).
\end{align*}

To estimate the second term in \eqref{eq:last}, we use \eqref{eq:o1}
and \eqref{eq:o2} and obtain
\begin{align*}
\int_{\cL_3} \widehat{f}_{\Omega,n} \cdot (1 - \eta_{L_T}) \,  d\mu
&\ll \int_{\{\height \geq \frac{L_T}{2}\}} \height \,  d\mu \leq \|\height\|_{L^2(\mu)} \cdot \mu(\{ \height \geq L_T/2\})^{1/2} \\[0.2cm]
&\ll L_T^{-1/2} \cdot \|\height\|_{L^2(\mu)} \cdot \|\height\|_{L^1(\mu)}^{1/2},
\end{align*}
where, in the second and third inequalities, we have used the Cauchy-Schwarz inequality and Markov's inequality respectively.  By \cite[Lemma 3.10]{EsMaMo},  $\height \in L^2(\mu)$ and thus
\[
\int_{\cL_3} \widehat{f}_{\Omega,n} \cdot (1 - \eta_{L_T}) \,  d\mu \ll L_T^{-1/2}.
\]
Therefore, we conclude that
\begin{equation}
\label{s3}
\big\|\widehat{h}_{\Omega,n} - \widehat{f}_{\Omega,n} \cdot \eta_{L_T}\big\|_{L^1(\mu)}
\ll_{\zeta,h}
\frac{\eps_T}{\at} +\max\Big(\eps_T,-\frac{\eps_T}{\at} \ln\Big(\frac{\eps_T}{\at} \Big)\Big)+L_T^{-1/2}.
\end{equation}
This gives the estimate for the second term in \eqref{split0}.
Finally, combining \eqref{s1}, \eqref{s2}, and \eqref{s3}, we deduce the lemma.
\end{proof}

\section{Proof of Theorem \ref{Thm_Main}}
\label{Sec:Main}

\subsection{Quantitative equidistribution of order two}
We recall some notation from Section \ref{Sec:Smooth}.  For an integer $s \geq 1$,  
we defined the norms $\cN_s$ on the space $C^\infty_b(\cL_3)$ of bounded
smooth functions on $\cL_3$ by
\[
\cN_s(\varphi) := \max\big(\|\varphi\|_{C^s(\cL_3)},\Lip(\varphi)\big),  \quad \textrm{for $\varphi \in C^\infty(\cL_3)$}.
\]
We denote by $C^\infty_c(\cL_3)$ the subspace of \emph{compactly supported} smooth functions on $\cL_3$,  and we recall that $\mu$ denotes the unique $\SL_3(\bR)$-invariant probability measure on $\cL_3$.  The following result is due to Kleinbock and Margulis \cite{KM2} (see also \cite{KM1} for the special case of one-parameter subgroups).

\begin{theorem}[Kleinbock--Margulis]
	There exist $\delta_o > 0$ and $s_o \geq 1$ such that for every $\varphi \in \bR \cdot 1 + C^\infty_c(\cL_3)$,
	\[
	 \int_{[0,1)^2} \varphi(a(t)\Lambda_{\ul{x}}) \,  d\ul{x} =\int_{\cL_3} \varphi\,d\mu+O\Big( e^{-\delta_o \lfloor t \rfloor} \,  \cN_{s_o}(\varphi)\Big)\,
	\]
	for all $t\in \bR^2_{+}$,  where the implicit constants are independent of $t$ and $\varphi$.
\end{theorem}

In \cite{BGEME},  the first and third author extended this result to products of two (and more) functions.  For products of two functions,  the exact statement is as follows:

\begin{theorem}
	\label{Thm_BG}
	There exist $\delta_o > 0$ and $s_o \geq 1$ such that for all $\varphi,  \psi \in \bR \cdot 1 + C^\infty_c(\cL_3)$,
	\[
	\int_{[0,1)^2} \varphi(a(s)\Lambda_{\ul{x}}) \psi(a(t)\Lambda_{\ul{x}}) \,  d\ul{x} =
	\int_{\cL_3} \varphi\,d\mu \int_{\cL_3} \psi\,d\mu +O\Big(
	e^{-\delta \min(\lfloor s \rfloor,\lfloor t \rfloor,\|s-t\|_\infty)} \,  \cN_{s_o}(\varphi) \cN_{s_o}(\psi)\Big),
	\]
	for all $s,t\in \bR^2_{+}$,  where the implicit constants are independent of $s,t$.
\end{theorem}

An analogous formula for correlations but with different error term was also proven by Shi \cite{Shi}.

\subsection{Notation}
We recall that the parameters $\at,  \bt$ and $\ct$ are positive real numbers satisfying
\begin{equation}\label{c}
\at < \bt \qand \ct < \frac{1}{2} \qand \ct^2>\zeta\cdot \bt
\end{equation}
for some $\zeta>0$. Without loss of generality, we may assume that $\zeta<1$.
Our goal is to analyze the sets
\[
Q_\Omega(\ul{x}) = \left\{ q \in [T_0,T) \cap \bN \,  : \,  
\exists \,  \ul{p} = (p_1,p_2) \in \bZ^2: \enskip  \,  
\begin{array}{c}
\at < |p_1 + qx_1| |p_2 + qx_2| \cdot q \leq \bt \\[0.2cm]
\max(|p_1 + qx_1|,|p_2+qx_2|) \leq \ct 
\end{array}
\right\}
\]
defined for $\ul{x} = ({x}_1,x_2) \in \bR^2$
and the sets
\[
\Omega = \left\{ (\ul{u},y) \in \bR^2 \times \bR \,  : \,  \at < |u_1 u_2| y \leq \bt,  \enskip \|\ul{u}\|_\infty \leq \ct,  \enskip T_0 \leq y < T \right\}.
\]
We note that for the lattices $\Lambda_{\ul{x}} \subset \bR^2 \times \bR$ defined by
\[
\Lambda_{\ul{x}} = \left\{ (\ul{p} + q\ul{x},q) \,  : \,  \ul{p} \in \bZ^2,  \enskip q \in \bZ \right\},
\]
we have
\begin{equation}
\label{qiff}
q \in Q_\Omega(\ul{x}) \iff \exists \,  \ul{p} \in \bZ^2 \enskip \textrm{such that} \enskip (\ul{p} + q\ul{x},q) \in \Lambda_{\ul{x}} \cap \Omega.
\end{equation}
Furthermore,  since $\ct < \frac{1}{2}$,  there can be at most one $\ul{p} \in \bZ^2$
for which the condition on the right-hand side holds.  \\

According to Lemma \ref{Lemma_tesselate}, we have the decomposition
\begin{equation}
\label{OmegaTtess}
\Omega =  \bigsqcup_{n \in \mathcal{F}_\Omega} a(n)^{-1}\Delta_{\Omega,n},
\end{equation}
where the sets
\begin{equation}
\label{DeltaIncl}
\Delta_{\Omega,n} \subset [-\ct,\ct]^2 \times \left(\frac{\at}{\ct^2},\frac{\bt e^2}{\ct^2} \right]
\end{equation}
are defined in \eqref{Def_DeltaTn}, and 
$\cF_\Omega$ is defined in \eqref{Def_FT}--\eqref{Def_AB}.

Let us now fix a compactly supported 
Lipschitz function $h : \bR \ra \bR$.  By
\eqref{qiff} and \eqref{OmegaTtess},  
\begin{align*}
\sum_{q \in Q_\Omega(\ul{x})} h\left(\frac{q}{T}\right)
&= \sum_{(\ul{p},q) \in \bZ^2 \times \bZ} h\left(\frac{q}{T}\right) \cdot \chi_{\Omega}(\ul{p} + q\ul{x},q) =
\sum_{(\ul{\lambda}_1,\lambda_2) \in \Lambda_{\ul{x}}} h\left(\frac{\lambda_2}{T}\right) 
\cdot \chi_{\Omega}(\ul{\lambda}_1,\lambda_2) \\[0.2cm]
&= 
\sum_{n \in \cF_\Omega} \sum_{\ul{\lambda} \in \Lambda_{\ul{x}}} 
h\left(\frac{\lambda_2}{T}\right) \cdot \chi_{\Delta_{\Omega,n}}(a(n)\ul{\lambda}) =
\sum_{n \in \cF_\Omega} \widehat{h}_{\Omega,n}(a(n)\Lambda_{\ul{x}}),
\end{align*}
for every $\ul{x} \in \bR^2$,  where
\[
h_{\Omega,n}(\ul{u},y): = h\left(\frac{e^{n_1 + n_2}y}{T}\right) \cdot \chi_{\Delta_{\Omega,n}}(\ul{u},y), \quad (\ul{u},y)\in\bR^2\times \bR.
\]
By Siegel's Theorem (Theorem \ref{Thm_Siegel}),  we have 
\begin{align*}
\int_{\cL_3}  \widehat{h}_{\Omega,n}(a(n)\Lambda) \,  d\mu(\Lambda)
&= \int_{\bR} \int_{\bR^2} h_{\Omega,n}(a(n)(\ul{u},y)) \,  d\ul{u} \,  dy \\[0.2cm]
&= \int_{\bR} h\left(\frac{y}{T}\right) \cdot 
\left( \int_{\bR^2} \chi_{a(n)^{-1}\Delta_{\Omega,n}}(\ul{u},y) \,  d\ul{u} \right) \,  dy,
\end{align*}
and thus,  by \eqref{OmegaTtess}, 
\begin{align*}
\sum_{n \in \cF_\Omega} \int_{\cL_3} \widehat{h}_{\Omega,n}(a(n)\Lambda) \,  d\mu(\Lambda)
&= \int_{\bR}
h\left(\frac{y}{T}\right) \cdot 
\left( \int_{\bR^2} \sum_{n \in \cF_\Omega }\chi_{a(n)^{-1}\Delta_{\Omega,n}}(\ul{u},y) \,  d\ul{u} \right) \,  dy
\\[0.2cm]
&= 
\int_{\bR}
h\left(\frac{y}{T}\right) \cdot \left(\int_{\bR^2}\chi_{\Omega}(\ul{u},y) \,  d\ul{u}\right) dy = \cM_\Omega(h),
\end{align*}
where $\cM_\Omega(h)$ is defined in \eqref{eq:MT}.
Therefore, we conclude that
\begin{equation}
\label{Rewrite1}
\sum_{q \in Q_\Omega(\ul{x})} h\left(\frac{q}{T}\right) - \cM_\Omega(h)
= \sum_{n \in \cF_\Omega} \phi_{\Omega,n}(a(n)\Lambda_{\ul{x}}),
\end{equation}
where
\[
\phi_{\Omega,n}(\Lambda) := \widehat{h}_{\Omega,n} - \int_{\cL_3} \widehat{h}_{\Omega,n} \,  d\mu.
\]

\subsection{ Partial moments}

To simplify notation, we write
\begin{equation}\label{STabIS}
D_\Omega(\Lambda) :=
 \sum_{n \in \cF_\Omega} \phi_{\Omega,n}(a(n)\Lambda)
\end{equation}
for a lattice $\Lambda$.
Our goal is to analyze convex moments of the form
\[
\cC_\Omega := \int_{\cL_3} \theta_\kappa(D_\Omega) \,  d\nu,
\]
where the function $\theta_\kappa$ is defined in \eqref{eq:theta00} 
and the measure $\nu$ on $\cL_3$ is given by
\[
\int_{\cL_3} \varphi \,  d\nu = \int_{[0,1)^2} \varphi(\Lambda_{\ul{x}}) \,  d\ul{x},  \quad\hbox{for $\varphi \in C_c(\cL_3)$.  }
\]

The following properties of the function $\theta_\kappa$ can be readily verified: 
\begin{itemize}
	\item[\textsc{(P1)}] $\theta_\kappa$ is a convex function,  increasing on $[0,\infty)$, 
	and $\theta_\kappa(t) \leq t^2$ for all $t$. \vspace{0.2cm}
	
	\item[\textsc{(P2)}] $\theta_\kappa(ct) \leq c^2 \cdot \theta_\kappa(t)$ for all $c \geq 1$ and $t$.  \vspace{0.1cm}
	
	\item[\textsc{(P3)}] $\theta_\kappa^{-1}(u) \ll_\kappa u^{1/2} \left(\ln^+  u \right)^{(1+\kappa)/2}$,  for all $u\ge 0$, where $\ln^+ u:=\ln\max(e,u)$.
\end{itemize}

It will be convenient to decompose the sum \eqref{STabIS} further:
for $\xi_T < \beta_\Omega$,  we introduce the sets
\[
\cG_\Omega^{+}(\xi_T) := \left\{ n \in \cF_\Omega \,  : \,  \lfloor n \rfloor \geq \xi_T \right\} \qand
\cG_\Omega^{-}(\xi_T) := \left\{ n \in \cF_\Omega \,  : \,  \lfloor n \rfloor < \xi_T \right\}.
\]
Then $\cF_\Omega = \cG_{\Omega}^{+}(\xi_T) \sqcup \cG_{\Omega}^{-}(\xi_T)$ for all $T$. 
Let us fix $\eps_T \in (0,1)$ and $L_T >1$,  and let $\rho_{\eps_T}$ and $\eta_{L_T}$ be defined as in Section \ref{Sec:Smooth}.  
We introduce the functions
\[
f_{\Omega,n} := \rho_{\eps_T} * h_{\Omega,n}
\]
on $\bR^3$,
and consider 
\[
\varphi_{\Omega,n} : = \widehat{f}_{\Omega,n} \cdot \eta_{L_T} - \int_{\cL_3} \widehat{f}_{\Omega,n} \cdot \eta_{L_T} \,  d\mu.
\]
on $\cL_3$ that provide a $C_c^\infty$-approximation for the functions 
$\phi_{\Omega,n}$. We now write:
\begin{align}
D_\Omega 
&=
\sum_{n \in \cG_\Omega^{-}(\xi_T)} \phi_{\Omega,n} \circ a(n) + \sum_{n \in \cG_{\Omega}^{+}(\xi_T)}
\left( \phi_{\Omega,n} - \varphi_{\Omega,n} \right) \circ a(n) 
+
\sum_{n \in \cG_{\Omega}^{+}(\xi_T)} \varphi_{\Omega,n} \circ a(n) 
\nonumber \\[0.2cm]
&=: D_\Omega^{(1)} + D_\Omega^{(2)} + D_\Omega^{(3)}.  \label{ST1ST2ST3}
\end{align}
Since $\theta_\kappa$ is convex and satisfies condition \textsc{(P2)},  we have
\[
\cC_\Omega \leq 3 \cdot \left( \cC_\Omega^{(1)} + \cC_\Omega^{(2)} + \cC_\Omega^{(3)} \right),
\] 
where
\[
\cC_\Omega^{(k)} := \int_{\cL_3} \theta_\kappa\left(D_\Omega^{(k)}\right) \,  d\nu,
\quad \textrm{for $k=1,2,3$}.
\]
In what follows,  we will estimate these partial moments separately.

\subsection{An upper bound on $\cC_\Omega^{(1)}$}

Since $\theta_\kappa$ is convex and satisfies \textsc{(P2)},  we have
\begin{align}
\cC_\Omega^{(1)}
&=
\int_{\mathcal{L}_{3}}\theta_\kappa\left(\sum_{n \in \cG^{-}_\Omega(\xi_T)} \phi_{\Omega,n} \circ a(n) \,  d\nu \right) \nonumber \\[0.2cm]
&\leq |\cG^{-}_\Omega(\xi_T)|^2 \cdot 
\int_{\mathcal{L}_{3}}\theta_\kappa\left( \frac{1}{|\cG^{-}_\Omega(\xi_T)|} \sum_{n \in \cG^{-}_\Omega(\xi_T)} \phi_{\Omega,n} \circ a(n) \,  d\nu \right) \nonumber \\[0.3cm]
&\leq
|\cG^{-}_\Omega(\xi_T)| \cdot \sum_{n \in \cG^{-}_\Omega(\xi_T)} 
\int_{\cL_3} \theta_\kappa\big(\phi_{\Omega,n} \circ a(n)\big) \,  d\nu. \label{C1bnd}
\end{align}
Furthermore,  
\begin{align}
\int_{\cL_3} \theta_\kappa\big(\phi_{\Omega,n} \circ a(n)\big) \,  d\nu
&\leq 2 \max\left(\int_{\cL_3} \theta_\kappa\big(\widehat{h}_{\Omega,n} \circ a(n)\big) \,  d\nu,  
\theta_\kappa\left(\int_{\cL_3} \widehat{h}_{\Omega,n} \,  d\mu\right) \right) 
\label{upperbndthetavarphi}
\end{align}

By Siegel's Theorem (Theorem \ref{Thm_Siegel}) and \eqref{DeltaIncl},  
\begin{equation}\label{eq:h1}
\int_{\cL_3} \widehat{h}_{\Omega,n} \,  d\mu = \int_{\bR^2 \times \bR} h_{\Omega,n}(\ul{u},y) \,  d\ul{u} dy
\leq \|h\|_\infty \Vol_3(\Delta_{\Omega,n}) \ll_h \bt \ll 1.
\end{equation}
We stress that the implicit constants only depend on $\|h\|_\infty$ (which is assumed to be fixed throughout the proof).  \\

Now, we observe that
\[
\big|\widehat{h}_{\Omega,n}(a(n)\Lambda_{\ul{x}}) \big|\leq \|h\|_\infty \cdot \widehat{\chi}_{\Delta_{\Omega,n}}(a(n)\Lambda_{\ul{x}}),  \quad \textrm{for all $\ul{x} \in \bR^2$}.
\]
Furthermore,  by \eqref{DeltaIncl},
\[
\chi_{\Delta_{\Omega,n}}(\ul{u},y) \leq \chi_{[-1,1]^2 \times [0,1]}(\ul{u},r_\Omega), \quad
\textrm{where $r_\Omega := \frac{\ct^2}{\bt e^2}$},
\]
so we conclude that
\[
\big|\widehat{h}_{\Omega,n}(a(n)\Lambda_{\ul{x}})\big| \leq \|h\|_\infty \cdot \widehat{\chi}_{[\shortminus 1,1]^2 \times [0,1]}(a(n)\Lambda_{\ul{x},r_\Omega}),
\]
where the lattice $\Lambda_{\ul{x},r_\Omega}$ is defined as in \eqref{Def_Lambdaxr}.  Thus, by Lemma \ref{Lemma_Schmidt}, 
\begin{equation}\label{eq:h2}
\big|\widehat{h}_{\Omega,n}(a(n)\Lambda_{\ul{x}})\big| \ll_h \height(a(n)\Lambda_{\ul{x},r_\Omega}).
\end{equation}
We note that $r_\Omega \geq e^{-2}\zeta$ by \eqref{c}.

Let us now use these estimates to bound $\cC_\Omega^{(1)}$.  We introduce the set
\[
E_{T,n} := \left\{ \ul{x} \in [0,1)^2 \,  : \,  \height(a(n)\Lambda_{\ul{x},r_\Omega}) \geq e^4\zeta^{-1} \right\}.
\]
By Corollary \ref{Cor_Heightintegrals},  applied with $\rho=e^{-2}\zeta$ and $\eta = e^4\zeta^{-1}$,  we have
for all $n \in \bN_o^2$,
\[
\int_{E_{T,n}} \theta_\kappa\big(\height(a(n)\Lambda_{\ul{x},r_\Omega})\big) \,  d\ul{x} 
\ll_{\zeta} r_\Omega^{-1} \cdot \int_{e^2\zeta^{-1}}^{\infty} \frac{\theta_\kappa(u)}{u^3} \, du \ll_{\kappa,\zeta} r_\Omega^{-1},
\]
where the implicit constants are independent of $T$ and $n$. Here we used that the function $u \mapsto \frac{\theta_\kappa(u)}{u^3}$ is integrable on $[1,\infty)$.  Hence, from \eqref{eq:h2}, since $\theta_{\kappa}$
is increasing and satisfies (P2), we conclude that
\begin{equation}
\label{intETn}
\int_{E_{T,n}} \theta_\kappa\big(\widehat{h}_{\Omega,n}(a(n)\Lambda_{\ul{x}})\big) \,  d\ul{x}
\ll_h \int_{E_{T,n}} \theta_\kappa\big(\height(a(n)\Lambda_{\ul{x},r_\Omega})\big) \,  d\ul{x} \ll_{\kappa,\zeta} r_\Omega^{-1},
\end{equation}
for all $n$.  

It remains to estimate the integral over $E_{T,n}^c$.
To do so,  note that
\[
\int_{E_{T,n}^c} \theta_\kappa\big(\widehat{h}_{\Omega,n}(a(n)\Lambda_{\ul{x}})\big) \,  d\ul{x}
\ll_h \int_{E_{T,n}^c \cap P_{\Omega,n}} \theta_\kappa\big(\height(a(n)\Lambda_{\ul{x},r_\Omega}) \big)\,  d\ul{x} 
\ll_{\kappa,\zeta} \Vol_2(P_{\Omega,n}),
\]
where 
\[
P_{\Omega,n} :
= \left\{ \ul{x} \in [0,1)^2 \,  : \,  \widehat{h}_{\Omega,n}(a(n)\Lambda_{\ul{x}}) > 0 \right\}.
\]
Using that $\ct < \frac{1}{2}$,  and the upper bounds, which follow from \eqref{DeltaIncl}, we obtain that
\[
\big|\widehat{h}_{\Omega,n}(a(n)\Lambda_{\ul{x}})\big| \leq \|h\|_\infty 
\cdot 
\widehat{\chi}_{\left[\shortminus \frac{1}{2},\frac{1}{2}\right]^2 \times \left[0,\frac{1}{2}\right]}\big(a(n)\Lambda_{\ul{x},{r_\Omega}/{2}}\big),
\]
we see that
\begin{align*}
P_{\Omega,n} 
&\subseteq \left\{ \ul{x} \in [0,1)^2 \,  : \,  a(n)\Lambda_{\ul{x},{r_\Omega}/{2}} \cap 
\left(\left[\shortminus \frac{1}{2},\frac{1}{2}\right]^2 \times 
\left[0,\frac{1}{2}\right]\right) \neq \{0\} \right\} \\[0.3cm]
&\subseteq \left\{ \ul{x} \in [0,1)^2 \,  : \,  s_1\big(a(n)\Lambda_{\ul{x},r_\Omega/2}\big) \leq \frac{1}{2} \right\},
\end{align*}
where $s_1$ is defined as in Subsection \ref{Subsec:heightL3}.  Hence,  by Lemma \ref{Lemma_s1upp},  applied with $\eps = \frac{1}{2}$ and $r = r_\Omega/2$,  we have
\[
\Vol_2(P_{\Omega,n}) \ll r_\Omega^{-1},  \quad \textrm{for all $n$}.
\]
Therefore, 
$$
\int_{E_{T,n}^c} \theta_\kappa\big(\widehat{h}_{\Omega,n}(a(n)\Lambda_{\ul{x}})\big) \,  d\ul{x} \ll_{h,\kappa,\zeta} r_\Omega^{-1}.
$$
If we combine this estimate with 
\eqref{eq:h1} and \eqref{intETn},  we deduce from \eqref{upperbndthetavarphi} that
\begin{equation}
\int_{\cL_3} \theta_\kappa\big(\varphi_{\Omega,n} \circ a(n)\big) \,  d\nu \ll_{h,\kappa,\zeta}
r_\Omega^{-1}\ll \frac{\bt}{\ct^2}  \label{Est1_varphiTn}
\end{equation}
for all $n \in \bN_o^2$. Hence,  by \eqref{C1bnd},
\begin{align*}
\cC_\Omega^{(1)} 
\leq 
|\cG^{-}_\Omega(\xi_T)| \cdot \sum_{n \in \cG^{-}_\Omega(\xi_T)} 
\int_{\cL_3} \theta_\kappa\big(\varphi_{\Omega,n} \circ a(n)\big) \,  d\nu 
\ll_{h,\kappa,\zeta} |\cG_{\Omega}^{-}(\xi_T)|^2 \cdot \frac{\bt}{\ct^2}.
\end{align*}
Since 
\[
|\cG_\Omega^{-}(\xi_T)| \leq 2\xi_T \cdot (\beta_\Omega-\alpha_\Omega)
\qand
(\beta_\Omega - \alpha_\Omega)^2 \leq |\cF_\Omega|,
\]
we conclude that
\[
|\cG_\Omega^{-}(\xi_T)|^2 \ll \xi_T^2 \cdot |\cF_\Omega|,
\]
and
\begin{equation}
\label{C1FinalBnd}
\cC_\Omega^{(1)} \ll_{h,\kappa,\zeta} \frac{\bt \cdot \xi^2_T}{\ct^2} \cdot |\cF_\Omega|.
\end{equation}

\subsection{An upper bound on $\cC_\Omega^{(2)}$}

Since  $\theta_\kappa(u) \leq u^2$ for all $u$,  we have
\begin{align*}
\int_{\cL_3} \theta_\kappa\big(D_\Omega^{(2)}\big) \,  d\nu
&\leq \int_{\cL_3} \left( \sum_{n \in \cG_\Omega^{+}(\xi_T)} (\phi_{\Omega,n} - \varphi_{\Omega,n}) \circ a(n)  \right)^2 \,  d\nu \\[0.2cm]
&\leq \left(
\sum_{n \in \cG_T^{+}}
\big\|(\phi_{\Omega,n} - \varphi_{\Omega,n}) \circ a(n)\big\|_{L^2(\nu)}\right)^2.
\end{align*}
Let us assume that
\begin{equation}
\label{Ass_xiTepsT2}
\xi_T > \max(1,-\ln\left({\at}/{2}\right)).
\end{equation}
Then  for all $n \in \cG^{+}_\Omega(\xi_T)$,  we have
$n_1+n_2>\max(1,-\ln\left({\at}/{2}\right))$.
In particular, if we additionally assume that 
\begin{equation}
\label{Ass_xiTepsT_2}
\eps_T<\at\zeta^2/100  \qand L_T\ge 2e^2\zeta^{-1},
\end{equation}
then
the conditions of Lemma \ref{Lemma_Approx1} are satisfied,  from which we conclude that
\begin{align*}
\big\|(\phi_{\Omega,n}  - \varphi_{\Omega,n}) \circ a(n) \big\|_{L^2(\nu)} 
\ll_{h,\zeta}&
\frac{ \eps_T}{\at} \cdot \max(1,n_1 + n_2)^{1/2} \\[0.2cm]
&+
e^{-(n_1 + n_2)/2} +  \max\Big(\eps_T, -\frac{\eps_T}{\at} \ln\Big(\frac{\eps_T}{\at} \Big)\Big)^{1/2} \cdot  \max(1,n_1 + n_2) \\[0.2cm]
&+
\left(L_T^{-1/2} + e^{-\lfloor n \rfloor/2} \right)\cdot \max(1,n_1 + n_2)^{1/2}
\end{align*}
for all $n \in \cG^{+}_\Omega(\xi_T)$.
We note that
\begin{align*}
&\sum_{n \in \cG_\Omega^{+}(\xi_T)} \max(1,n_1 + n_2)^{1/2}
\leq \beta_\Omega^{1/2} \cdot |\cG_{\Omega}^{+}(\xi_T)|, \\[0.2cm]
&\sum_{n \in \cG_\Omega^{+}(\xi_T)} \max(1,n_1 + n_2)
\leq \beta_\Omega \cdot |\cG_{\Omega}^{+}(\xi)|, \\[0.2cm]
&\sum_{n \in \cG_\Omega^+(\xi_T)} e^{-(n_1 + n_2)/2} \le 2 e^{-\xi_T/2} \cdot |\cG_\Omega^{+}(\xi_T)|, \\[0.2cm]
&\sum_{n \in \cG_\Omega^{+}(\xi_T)} e^{-\lfloor n \rfloor/2} \max(1,n_1 + n_2)^{1/2}
\le 2 \beta_\Omega^{1/2} \cdot e^{-\xi_T/2} \cdot |\cG_{\Omega}^{+}(\xi_T)|,
\end{align*}
and thus
\vspace{0.1cm}
\begin{align}
\label{C2FinalBnd}
\int_{\cL_3} \theta_\kappa\big(D_\Omega^{(2)}\big) \,  d\nu 
&\leq \left(\sum_{n \in \cG_\Omega^{+}(\xi_T)}
\big\|(\phi_{\Omega,n} - \varphi_{\Omega,n}) \circ a(n)\big\|_{L^2(\nu)}\right)^2
 \nonumber \\[0.2cm]
&\ll_{h,\zeta} \gamma \cdot |\cG_\Omega^{+}(\xi_T)|^2\leq \gamma\cdot\beta_\Omega^2 \cdot |\cF_\Omega|,
\end{align}
  where
\[
\gamma := 
\beta_\Omega \cdot \max\left( \left(\frac{\eps_T}{\at}\right)^2,  \beta_\Omega \cdot \max \left(\eps_T,-\frac{\eps_T}{\at} \ln\Big(\frac{\eps_T}{\at}\Big)\right),  
L_T^{-1},  e^{-\xi_T} \right).
\]
\subsection{An upper bound on $\cC_\Omega^{(3)}$}

{
We have
\begin{align}
\int_{\cL_3} \left(D_\Omega^{(3)}\right)^2 \,  d\nu 
= & \int_{\cL_3} \sum_{m,n \in \cG_\Omega^{+}(\xi_T)}\varphi_{\Omega,m} \circ a(m) \cdot \varphi_{\Omega,n} \circ a(n) \,  d\nu \nonumber \\
=& \sum_{\substack{m,n \in \cG_\Omega^{+}(\xi_T) \\[0.1cm] \|m-n\| < \xi_T}} 
\int_{\cL_3} \varphi_{\Omega,m} \circ a(m) \cdot \varphi_{\Omega,n} \circ a(n) \,  d\nu \nonumber \\[0.2cm]
&+ \sum_{\substack{m,n \in \cG_\Omega^{+}(\xi_T) \\[0.1cm] \|m-n\| \geq \xi_T}} 
\int_{\cL_3} \varphi_{\Omega,m} \circ a(m) \cdot \varphi_{\Omega,n} \circ a(n) \,  d\nu.  \label{C3est}
\end{align}}
We estimate the two sums on the right-hand side separately.  \\

{
By Lemma \ref{Lemma_Schmidt} and \eqref{DeltaIncl},
$|\varphi_{\Omega,n}|\ll_{h,\zeta} L_T$ for all $n$.
Then, by the Cauchy-Schwarz inequality, we deduce  
\begin{align*}
& \int_{\cL_3}\left|\varphi_{\Omega,m} \circ a(m) \cdot \varphi_{\Omega,n} \circ a(n) \right|\,  d\nu \\
& \ll_{h,\zeta} (\ln L_T)^{1+\kappa}\int_{\cL_3}\ln\left(e+|\varphi_{\Omega,m}\circ a(m)|\right)^{-(1+\kappa)/2}\cdot\left|\varphi_{\Omega,m} \circ a(m)\right| \\
&\qquad\qquad\qquad\qquad\qquad\cdot\ln\left(e+|\varphi_{\Omega,n}\circ a(n)|\right)^{-(1+\kappa)/2}\cdot \left|\varphi_{\Omega,n} \circ a(n) \right|\,  d\nu \\
& \leq (\ln L_T)^{1+\kappa}\left(\int_{\cL_3}\theta_{\kappa}(\varphi_{\Omega,m}\circ a(m))\,d\nu\cdot \int_{\cL_3}\theta_{\kappa}(\varphi_{\Omega,n}\circ a(n))\, d\nu\right)^{1/2}.
\end{align*}}
{
Hence, by  (\ref{Est1_varphiTn}), we have that
$$\int_{\cL_3}\left|\varphi_{\Omega,m} \circ a(m) \cdot \varphi_{\Omega,n} \circ a(n)\right|\,  d\nu\ll_{h,\zeta}\frac{b{\cdot(\ln L_T)^{1+\kappa}}}{c^2}.$$
}
{so that
\begin{equation}
\label{eq:newnew}
\sum_{\substack{m,n \in \cG_\Omega^{+}(\xi_T) \\[0.1cm] \|m-n\| < \xi_T}} 
\int_{\cL_3} \left|\varphi_{\Omega,m} \circ a(m) \cdot \varphi_{\Omega,n} \circ a(n)\right| \,  d\nu\ll_{h,\zeta} \frac{b{\cdot (\ln L_T)^{1+\kappa}}}{c^2}\cdot|\cM_\Omega(\xi_T)|,
\end{equation}}
where
\[
\cM_\Omega(\xi_T) :=  \left\{ (m,n) \in \cG_\Omega^{+}(\xi_T) \,  : \,  \|m-n\| < \xi_T \right\} .
\]
For every $m \in \cG_\Omega^{+}(\xi_T)$,  there are $O(\xi_T^2)$ elements $n$ in $\cG_\Omega^{+}(\xi_T)$ such that $\|m-n\| < \xi_T$,  and thus $|\cM_\Omega(\xi_T)| \ll \xi_T^2 \cdot |\cG_\Omega^{+}(\xi_T)|$,  where the implicit constants are independent of $\Omega$.  We thus see that
{\begin{equation}
\label{C3Est1}
\sum_{\substack{m,n \in \cG_\Omega^{+}(\xi_T) \\[0.1cm] \|m-n\| < \xi_T}} 
\int_{\cL_3} \varphi_{\Omega,m} \circ a(m) \cdot \varphi_{\Omega,n} \circ a(n) \,  d\nu \ll_{h,\zeta}
\frac{b{\cdot (\ln L_T)^{1+\kappa}}}{c^2}\cdot \xi_T^2 \cdot |\cF_\Omega|.
\end{equation}}

\vspace{0.2cm}

Let us now turn to the second sum in \eqref{C3est}.  By construction,  the function $\varphi_{\Omega,n}$
belongs to $\bR \cdot 1 + C^\infty_c(\cL_3)$ and satisfies $\int_{\cL_3} \varphi_{\Omega,n}  \,  d\mu = 0$.  Hence,  by Theorem \ref{Thm_BG},  there exist
$s \geq 1$ and $\delta > 0$ such that
\[
\left| \int_{\cL_3} \varphi_{\Omega,m} \circ a(m) \cdot \varphi_{\Omega,n} \circ a(n) \,  d\nu \right|
\ll 
e^{-\delta \min(\lfloor m \rfloor, \lfloor n \rfloor,\|m-n\|)} \,  \cN_s(\varphi_{\Omega,m}) \cdot \cN_s(\varphi_{\Omega,n}),
\]
where the implicit constants are independent of $m,n,T$.  Furthermore,  by Lemma
\ref{Lemma_Smooth},  there exists a constant $\sigma_s > 0$ such that
\[
 \cN_s(\varphi_{\Omega,n}) \ll_h \eps_T^{-\sigma_s} \cdot L_T,  \quad \textrm{for all $n$},
\]
where the implicit constants are independent of $\Omega$ and $n$.  We conclude that for 
all $m,n \in \cG_\Omega^{+}(\xi_T)$ such that $\|m-n\| \geq \xi_T$,  we have
\[
\left| \int_{\cL_3} \varphi_{\Omega,m} \circ a(m) \cdot \varphi_{\Omega,n} \circ a(n) \,  d\nu \right|
\ll_h e^{-\delta \xi_T} \cdot \eps_T^{-2\sigma_s} \cdot L_T^2,
\]
and thus
\begin{equation}
\label{C3Est2}
\sum_{\substack{m,n \in \cG_\Omega^{+}(\xi_T) \\[0.1cm] \|m-n\| \geq \xi_T}} 
\int_{\cL_3} \varphi_{\Omega,m} \circ a(m) \cdot \varphi_{\Omega,n} \circ a(n) \,  d\nu
\ll_h
e^{-\delta \xi_T} \cdot \eps_T^{-2\sigma_s} \cdot L_T^2 \cdot |\cG_\Omega^{+}(\xi_T)|^2.
\end{equation}
If we now plug our estimates \eqref{C3Est1} and \eqref{C3Est2} into \eqref{C3est},  we conclude that
\begin{align*}
\int_{\cL_3} \theta_\kappa\big(D_\Omega^{(3)}\big) \,  d\nu 
\ll_{h,\zeta} &\frac{b\cdot(\ln L_T)^{1+\kappa}\cdot\xi_T^2}{c^2} \cdot |\cF_\Omega| 
+ e^{-\delta \xi_T} \cdot \eps_T^{-2\sigma_s} \cdot L_T^2 \cdot |\cG_\Omega^{+}(\xi_T)|^2.
\end{align*}
Since $|\cG_\Omega^{+}(\xi_T)| \ll \beta_\Omega^2$,  where the implicit constants are independent of $\Omega$,  we have
\begin{align}
\label{C3FinalBnd}
\int_{\cL_3} \theta_\kappa\big(D_\Omega^{(3)}\big) \,  d\nu 
\ll_{h,\zeta} \left(\frac{b\cdot (\ln L_T)^{1+\kappa}\cdot \xi_T^2}{c^2} +
+ e^{-\delta \xi_T} \cdot \eps_T^{-2\sigma_s} \cdot L_T^2 \cdot \beta_\Omega^2 \right) \cdot |\cF_\Omega|.
\end{align}

\subsection{Putting it all together}

Let us now summarize what we have so far:
\[
\cC_\Omega \leq 3 \cdot \left(\cC^{(1)}_\Omega + 
\cC^{(2)}_\Omega + \cC^{(3)}_\Omega \right),
\]
and by \eqref{C1FinalBnd},  \eqref{C2FinalBnd},  \eqref{C3FinalBnd},
\begin{align*}
\cC_\Omega^{(1)} 
&\ll_{h,\kappa,\zeta}
\frac{\bt \cdot \xi^2_T}{\ct^2} \cdot |\cF_\Omega|, \\[0.2cm]
\cC_\Omega^{(2)}
&\ll_{h,\zeta} 
\beta_\Omega^3 \cdot \max\left( \left(\frac{\eps_T}{\at}\right)^2,  \beta_\Omega \cdot \max\left(\eps_T,-\frac{\eps_T}{\at} \ln\Big(\frac{\eps_T}{\at}\Big)\right),  
{L_T}^{-1},  e^{-\xi_T} \right) \cdot |\cF_\Omega|, \\[0.2cm]
\cC_\Omega^{(3)} 
&{\ll_h  \Big( \frac{b\cdot(\ln L_T)^{1+\kappa}\cdot \xi_T^2}{c^2} + e^{-\delta \xi_T} \cdot \eps_T^{-2\sigma_s} \cdot L_T^2 \cdot \beta_\Omega^2 \Big) \cdot |\cF_\Omega|},
\end{align*}
provided that the parameters have been chosen so that
\begin{align}
\label{conds}
\xi_T \geq \max(1,-\ln\left({\at}/{2}\right)),\quad 
\eps_T<\at\zeta^2/100,  \quad L_T\ge e^2\zeta^{-1},\quad \ct^2>\zeta\cdot \bt.
\end{align}
We further assume that
\[
\at \geq (\ln T)^{-\theta},  \quad \textrm{for some $\theta > 0$}.
\]
Then, in particular, $\beta_\Omega= \ln\left(\frac{T\ct^2}{\at}\right)\ll_{\theta} \ln T.$
In what follows, we will choose the parameters $\xi_T, \eps_T$, and $L_T$, so that
\begin{equation}
\label{CbndFinal}
\cC_\Omega \ll_{h,\kappa,\zeta} \left(\frac{\bt\cdot(\ln L_T)^{1+\kappa} \cdot \xi_T^2}{\ct^2} + O_{\theta,\rho}\big((\ln T)^{-\rho}\big)\right) \cdot |\cF_\Omega|
\end{equation}
for all $\rho>0$.  \\

To prove \eqref{CbndFinal},  let 
\[
\xi_T = \xi_o \cdot \ln(\ln T), \quad L_T = (\ln T)^{3 + \eta}, \quad \eps_T = (\ln T)^{-\gamma}
\]
for some \emph{positive} constants $\xi_o,  \eta$ and $\gamma$ that will be chosen later.
With these choices,  we see that
\begin{align*}
\beta^3_\Omega \cdot  \left(\frac{ \eps_T}{\at}\right)^2 
&\ll_{\theta} (\ln T)^{3 - 2(\gamma-\theta)}, \\[0.2cm]
\beta_\Omega^4 \cdot \max \left(\eps_T,-\frac{\eps_T}{\at} \ln\Big(\frac{\eps_T}{\at}\Big)\right)
&\ll_{\theta} (\ln T)^{4-(\gamma-\theta)} \cdot \ln(\ln T), \\[0.4cm]
\beta_\Omega^3\cdot \max\left({L_T}^{-1},  e^{-\xi_T} \right)
&\ll_{\theta} \max\left((\ln T)^{-\eta},(\ln T)^{3-\xi_o} \right),
\\[0.4cm]
e^{-\delta \xi_T} \cdot \eps_T^{-2\sigma_s} \cdot L_T^2 \cdot \beta_\Omega^2
&\ll_{\theta} (\ln T)^{-\delta \xi_o + 2\sigma_s \gamma + 2(3+\eta) + 2},
\end{align*}
for all sufficiently large $T$.   We think of $\theta$ as fixed throughout the argument.  Hence,  if we pick $\eta,  \gamma$ sufficiently large
and next $\xi_o$ sufficiently large,  
then all of the right-hand sides above are $O_{\theta,\rho}((\ln T)^{-\rho})$ for any $\rho>0$.
Furthermore, one verifies that all the conditions in \eqref{conds} are 
satisfied. Hence, \eqref{CbndFinal} follows.  Since
$$
\frac{b}{c^2}\gg a \ge (\ln T)^{-\theta},
$$
we obtain that 
\begin{equation}\label{eq:conc}
\cC_\Omega \ll_{h,\kappa,\zeta,\theta} \frac{\bt \cdot (\ln L_T)^{1+\kappa}\cdot \xi_T^2}{\ct^2} \cdot |\cF_\Omega|.
\end{equation}
Furthermore,
$$
\xi_{T}\ll \ln\ln T\qand |\cF_\Omega|\ll \beta_\Omega^2-\alpha_\Omega^2,
$$
where
$$
\alpha_\Omega=\ln \left(\frac{T_0 c^2}{e^2b} \right) \qand 
\beta_\Omega=\ln \left(\frac{T c^2}{a} \right).
$$
Since, according to our assumptions,
$$
\zeta\le \frac{c^2}{b}<\frac{c^2}{a}\le (\ln T)^{\theta},
$$
we conclude that
\begin{align*}
|\cF_\Omega|\ll
 \left(\ln T+ \ln  \left(\frac{ c^2}{a} \right)\right)^2-\left(\ln T_0 +\ln \left(\frac{c^2}{e^2b} \right)\right)^2 
 \ll_\zeta \cL([T_0,T)) + \cM([T_0,T)),
 \end{align*}
 where
$$
\cL([T_0,T)):=(\ln T)^2- (\ln T_0)^2\quad\hbox{and}\quad
\cM([T_0,T)):=\ln T\ln\ln T.
$$
Hence, it follows from \eqref{Rewrite1} and \eqref{eq:conc} that 
\begin{align}
\int_{[0,1)^2} \theta_\kappa\Big(
S_\Omega h(\ul{x})-\cM_\Omega(h) \Big)\, d\ul{x}&=\int_{\mathcal{L}_3} \theta_\kappa\left(D_\Omega\right)\, d\nu \nonumber \\
& \ll_{h,\kappa,\zeta,\theta} c^{-2}\cdot b\cdot (\ln\ln T)^{3+\kappa}\cdot \Big( \cL([T_0,T)) + \cM([T_0,T))\Big). \label{eq:conc2}
\end{align}
In particular, in the case when $T_0=1$, we obtain Theorem \ref{Thm_Main}.

\section{Proof of Theorem \ref{Thm_Cnt}}
\label{Sec:Main Thm}

We work with the sets
$$
\Omega=\Omega^{(a,b],c}_{[T_0,T)} := \left\{ (u,y) \in \bR^2 \times [T_0,T) \,  : \,  \begin{array}{c}
\max(|u_1|,|u_2|) \leq c \\[0.2cm]
a < |u_1u_2| \cdot y \leq b
\end{array}
 \right\}
$$
that depend on parameters $0<a<b<1$, $0<c<1/2$, $1\le T_0<T$.
For a lattice $\Lambda$, we consider the discrepancy function
$$
D^{(a,b],c}_{[T_0,T)}(\Lambda):=\left|\Lambda\cap \Omega^{(a,b],c}_{[T_0,T)}\right|-\Vol_3\left(\Omega^{(a,b],c}_{[T_0,T)}\right).
$$
Let us take in \eqref{eq:conc2} a Lipschitz function $h$ such that $h=1$ on $[0,1]$.  Then 
$$
S_\Omega h(\ul{x})= \left|\Lambda_{\ul{x}}\cap \Omega\right|\qand
\cM_\Omega (h) =\Vol_3(\Omega),
$$
so that it follows that
\begin{equation}\label{AsS_OmegahetaK}
\int_{\mathcal{L}_3} \theta_\kappa\left(\left|D^{(a,b],c}_{[T_0,T)}\right|\right)\, d\nu
 \ll_{\kappa,\zeta,\theta} c^{-2}\cdot b\cdot (\ln\ln T)^{3+\kappa}\cdot \Big( \cL([T_0,T)) + \cM([T_0,T))\Big).
\end{equation}
We use this estimate for a Borel--Cantelli argument below.

\medskip

For $s \in \bN$,  let us denote by $\cI_s$
the collection of intervals 
\begin{equation}
\label{Def_Iij}
I_{i,j}: = \left[e^{2^i j}, e^{2^{i}(1+j)}\right).
\end{equation}
with $(i,j) \in \bN_o^2$ satisfying $2^i (1+j) < 2^s$.
The following lemma is essentially \cite[Lemma 1]{Schmidt60},  although this lemma takes place in a slightly different setting.  We provide a proof of our version 
for completeness. 

\begin{lemma}
	\label{Lemma_Cover0N}
	For every $1 \leq N < 2^s$,  there exists a 
	subset $\cH_N \subset \cI_s$ such that
	\[
	|\cH_N| \leq s \qand [1,e^N) = \bigsqcup_{I \in \cH_N} I.
	\]
\end{lemma}

\begin{proof}
	We write 
	\[
	N = \sum_{k=1}^p 2^{n_k},
	\]
	where $0 \leq n_1 < n_2 < \ldots < n_p < s$ are integers.  In particular,  $p \leq s$.
	Let
	\[
	N_0 = 0 \qand N_m = \sum_{k=p - m + 1}^{p} 2^{n_k},  \quad \textrm{for $m = 1,\ldots,p$},
	\]
	so that $N_p = N$ and
	\[
	[1,e^N) = \bigsqcup_{m=1}^{p} [e^{N_{m-1}},e^{N_m}).
	\]
	We claim that each interval $[e^{N_{k}},e^{N_{k+1}})$ for $k=0,\ldots,p-1$ is of the form $I_{i,j}$ for some index pair $(i,j) \in \cG_s$.  To see this,  first note that $[1,e^{N_1}) = I_{i_1,j_1}$,  where $i_1 = n_p$ and $j_1 = 0$.  More generally,  we see that
	\[
	N_{m+1} = 2^{n_{p-m}} + \ldots + 2^{n_{p}} 
	= 
	2^{n_{p-m}}\left(1 + 2^{n_{p-m+1}-n_{p-m}} + \ldots + 2^{n_p - n_{p-m}} \right).
	\]
	Hence,  if we define 
	\[
	i_m = n_{p-m} \qand j_{m} = 2^{n_{p-m+1}-n_{p-m}} + \ldots + 2^{n_p - n_{p-m}}
	\] 
	for $1 \leq m \leq p-1$,  then $N_m = 2^{i_m} j_m$ and $N_{m+1} = 2^{i_m}(1 + j_m)$,
	and thus $[e^{N_{m}},e^{N_{m+1}}) = I_{i_m,j_m}$.  In particular,  
	\[
	[1,e^N) = \bigsqcup_{m=1}^p I_{i_m,j_m}.
	\]
	so that we establish the lemma with $\cH_N := \{ I_{i_m,j_m} \,  : \,  1 \leq m \leq p \}$.  
\end{proof}

\begin{lemma}
	\label{Corollary_SumGsFIij}
	For every $s\in\bN$,  
	\[
|\cI_s|\ll 2^s,\quad\quad	\sum_{I \in \cI_s} \cL(I) \ll s \cdot 2^{2s},
 \quad\quad \sum_{I \in \cI_s} \cM(I) \ll s \cdot 2^{2s}.
	\]
\end{lemma}

\begin{proof}
We have
\[	\cL(I_{i,j}) = (2^i j + 2^i)^2-(2^i j)^2\ll 2^{2i} \cdot \max(1,j), 
	\]
and
$$
\cM(I_{i,j}) \le s\cdot 2^i(j+1),
$$
so that
	\[
	\sum_{I \in \cI_s} \cL(I) \ll
	\sum_{i=0}^s 2^{2i} \left( \sum_{j=0}^{2^{s-i} -1}  \max(1,j) \right) \ll \sum_{i=0}^s 2^{2i} \,  (2^{s-i})^{2} \ll s \cdot 2^{2s},
	\]
 and
 $$
 \sum_{I \in \cI_s} \cM(I) \ll
	s\cdot \sum_{i=0}^s 2^{i} \left( \sum_{j=0}^{2^{s-i} -1} (j+1) \right) \ll
 s\cdot \sum_{i=0}^s 2^{i} \,  (2^{s-i})^{2} \ll s \cdot 2^{2s},
 $$
 as required.
\end{proof}

We note that later in the argument below, we choose the parameters $N$ and $s$
so that $e^{N-1}\le T < e^N$ and $2^{s-1}\le N < 2^s$,
so that   $\ln T\ll 2^s\ll \ln T$.

\medskip

We also use a similar argument to decompose the intervals $(a,b]$.
For $M\in \bN$ and $v\in\bN$, we write 
$$
a_M:=M^{-\sigma}\qand b_v:=2^{-v}\quad\hbox{with $\sigma>0.$ }
$$
Let $t\in\bN$. For $(i,j) \in \bN_o\times \bN$ satisfying $2^i (1+j) < 2^t$, we set
\begin{equation}
\label{Def_Iij_1}
J_{i,j}: = \left( (2^{i}(1+j))^{-\sigma},(2^i j)^{-\sigma}\right],
\end{equation}
and for $i=1,\ldots,t-1$,
\begin{equation}
\label{Def_Iij_1_1}
J_{i,0}: = \left( 2^{-\sigma i},1\right].
\end{equation}
Additionally, we subdivide each of those intervals into smaller intervals
using the points $b_v$ with $v=1,\ldots, t-1$. 
Let us denote the collection of all intervals that we obtain in this way by $\cJ_t$.
Since each of the original
intervals $J_{i,j}$ is subdivided into at most $t$ subintervals, the following lemma
follows immediately from Lemma \ref{Lemma_Cover0N}.

\begin{lemma}
	\label{Lemma_Cover0N_1}
	For every $M\in \bN$ with $2 \leq M < 2^t$ and $v\in \bN$ with $M^{-\sigma}<2^{-v}$, there exists a 
	subset $\cG_{M,v} \subset \cJ_t$ such that
	\[
	|\cG_{M,v}| \leq t^2 \qand (a_M,b_v] = \bigsqcup_{J \in \cG_{M,v}} J.
	\]
\end{lemma}

Additionally, we note that 

\begin{lemma}
	\label{Corollary_SumGsFIij_1}
Let $\rho\in (0,1)$ and $\sigma\ge \rho^{-1}$.
Then for every $t\in\bN$,  
	\[
|\cJ_t|\ll t\cdot 2^t\quad\hbox{and}\quad	
\sum_{(a,b]\in\cJ_t} b^{\rho}\ll t.
\]
\end{lemma}

\begin{proof}
 We note that the number of the intervals $J_{i,j}$ is $O(2^t)$.
 Since each of them is subdivided into at most $t$ subintervals, the first 
 estimate follows.
Regarding the second estimate, we observe that the sum over intervals
obtained from the subdivision of the interval $(u_1,u_2]=J_{i,j}$ is $O(u_2^{1/2})$,
so that
\begin{align*}
\sum_{(a,b]\in\cJ_t} b^{\rho}&\ll \sum_{i=0}^t \left(\sum_{j=1}^{2^{t-i} -1} (2^i  j)^{-\sigma\rho}+1\right)\ll \sum_{i=0}^t 2^{-i\sigma\rho}\cdot  2^{(t-i)(1-\sigma \rho)} +t
\ll t,
\end{align*}
which proves the lemma.
\end{proof}

We recall our assumption that $a\ge (\ln T)^{-\theta}$.
Later in the proof, we pick an integer $M$ such that $a_M\le a< a_{M-1}$
and $M<2^t$,
so that we may take the parameter $t$ satisfying $(2^t-1)^\sigma\le (\ln T)^{\theta}$,
in particular, 
$t\ll \ln\ln T\ll s$.

\medskip

To interpolate the parameter $c$, we use 
$$
c_w:=2^{-w} \quad \hbox{with $w\in \bN.$}
$$
According to our assumptions, $c\ge \zeta^{1/2} (\ln T)^{-\theta/2}$,
so that it is sufficient to consider $w\le  r$ with $r=O(\ln\ln T)$,
in particular, $r\ll s$.

\medskip

With these prerequisites, we are ready to set up a Borel--Cantelli argument.
Let us fix $\kappa,  \eps > 0$, $\rho\in (0,1)$, $\sigma\ge \rho^{-1}$,
and an integer $s \geq 2$.  
Let $t(s),r(s)\in \bN$ such that $t(s)\ll s$ and $r(s)\ll s$.
We set
	\begin{equation}
	\label{Def_Yeps}
	X_{s} :
	= 
	\left\{ \Lambda\in\cL_3 \,  : \,  \sum_{w\le r(s)} \sum_{I \in \cI_s}\sum_{(a,b]\in\cJ_{t(s)}} c_w^2 \cdot b^{-(1-\rho)}\cdot \theta_\kappa\left(\left|D^{(a,b], c_w}_{I}(\Lambda)\right|\right) \geq  s^{7+\kappa+\eps} \cdot  2^{2s}
 \right\}
	\end{equation}
	Then from Chebyshev's inequality and  \eqref{AsS_OmegahetaK},  
 	\begin{align*}
	\nu(X_s)
	&\leq \frac{1}{s^{7+\eps}\cdot 2^{2s}
 } \sum_{w\le r(s)} \sum_{I \in \cI_s}\sum_{(a,b]\in\cJ_{t(s)}}
	  c_w^2\cdot b^{-(1-\rho)} \int_{\cL_3} \theta_\kappa\left(\left|D^{(a,b], c_w}_{I}\right|\right) \,  d\nu\\
   &\ll_{\kappa,\zeta,\theta}
	\frac{s^{4+\kappa}}{s^{7+\kappa+\eps} \cdot 2^{2s}
 } 
 \left(\sum_{I \in \cI_s} \big(\cL(I)+\cM(I)\big)\right)
 \left(\sum_{(a,b]\in\cJ_{t(s)}} b^{\rho}\right).
	\end{align*}
Hence, it follows from Lemmas \ref{Corollary_SumGsFIij} and \ref{Corollary_SumGsFIij_1} that
\begin{align}\label{eq:summ}   
\nu(X_s) \ll s^{-(1+\eps)}.
\end{align}
	Let us take an integer $1 \leq N < 2^s$, $1 \leq M < 2^{t}$, and $v\in \bN$ such that $1/M<2^{-v}$. We use the partitions of the intervals
    $[1,e^N)$ and $(a_M,b_v]$ provided by Lemmas \ref{Lemma_Cover0N} and
    \ref{Lemma_Cover0N_1}.
    Then
	\[
	\left|D^{(a_M,b_v], c_w}_{[1,e^N)}\right| = \left| \sum_{I\in \cH_N} \sum_{J \in \cG_{M,v}} D^{J, c_w}_{I} \right| \leq 
	\frac{1}{|\cH_N||\cG_{M,v}|} \sum_{I\in \cH_N} \sum_{J \in \cG_{M,v}} |\cH_N||\cG_{M,v}| \cdot \left|D^{J, c_w}_{I}\right|.
	\]
	Since $\theta_\kappa$ is increasing, convex and satisfies the property \textsc{(P2)},  we have
	\begin{align*}
	\theta_\kappa\left(\left|D^{(a_M,b_v], c_w}_{[1,e^N)}\right|\right) 
	&\leq 
	\theta_\kappa\left(\frac{1}{|\cH_N||\cG_{M,v}|} \sum_{I\in \cH_N}\sum_{J \in \cG_{M,v}} |\cH_N||\cG_{M,v}| \cdot \left|D^{J, c_w}_{I}\right|\right) \\[0.2cm]
	&\leq 
	\frac{1}{|\cH_N||\cG_{M,v}|} \sum_{I\in \cH_N}\sum_{J\in \cG_{M,v}} \theta_\kappa\left(|\cH_N||\cG_{M,v}| \cdot \left|D^{J, c_w}_{I}\right| \right) \\[0.3cm]
	&\leq 
	|\cH_N||\cG_{M,v}| \cdot \sum_{I\in \cH_N}\sum_{J \in \cG_{M,v}} \theta_\kappa\left(\left|D^{J, c_w}_{I}\right|\right) 
	\ll
	s^3 \cdot \sum_{I\in \cH_N}\sum_{J \in \cG_{M,v}} \theta_\kappa\left(\left|D^{J, c_w}_{I}\right|\right),
	\end{align*}
	where we have used the property \textsc{(P2)} with $c = |\cH_N| |\cG_{M,v}|\le s\cdot t(s)^2\ll s^3$ in the last third inequality. This also implies that
    $$
    b_v^{-(1-\rho)}\cdot \theta_\kappa\left(\left|D^{(a_M,b_v], c_w}_{[1,e^N)}\right|\right) \ll 
    s^3 \cdot \sum_{I\in \cH_N}\sum_{(a,b] \in \cG_{M,v}} b^{-(1-\rho)}\cdot\theta_\kappa\left(\left|D^{J, c_w}_{I}\right|\right).
    $$
	Therefore, for $\Lambda \notin X_s$, 
	\[
	\theta_\kappa\left(\left|D^{(a_M,b_v], c_w}_{[1,e^N)}(\Lambda)\right|\right) \ll c_w^{-2}\cdot b_v^{1-\rho}\cdot s^{10+\kappa+\eps} \cdot 2^{2s}
 .
	\]
	Then  by the property \textsc{(P3)},
	\begin{align}\label{eq:NNN}
	\left|D^{(a_M,b_v], c_w}_{[1,e^N)}(\Lambda)\right| 
	&\ll_\kappa c_w^{-1}\cdot b_v^{(1-\rho)/2} \cdot s^{(10+\kappa+\eps)/2} \cdot 2^{s}
 \cdot \left(\ln^+\left(c_w^{-2}\cdot b^{1-\rho} \cdot s^{10+\kappa+\eps} \cdot 2^{s}
 \right) \right)^{(1+\kappa)/2} \nonumber \\[0.2cm]
	&\ll c_w^{-1}\cdot b_v^{(1-\rho)/2}\cdot  \left( \ln^+ \left(c_w^{-2}\cdot b_v^{1-\rho}\right) \right)^{(1+\kappa)/2} s^{11/2+(2\kappa+\eps)/2} \cdot \left( \ln s \right)^{(1+\kappa)/2} \cdot 2^{s}
 \nonumber \\
    &\ll_\zeta c_w^{-1}\cdot b_v^{(1-\rho)/2}\cdot  s^{6+(3\kappa+\eps)/{2}} \cdot \left( \ln s \right)^{1+\kappa} \cdot 2^{s}
. 
	\end{align}
	for all $s \geq 2$ and $\Lambda \notin X_s$,  where we  have used that
	\[
	\ln^+(x_1 \cdot x_2) \leq 2\cdot \ln^+ x_1 \cdot \ln^+ x_2 ,  \quad \textrm{for all $x_1,x_2\ge 0$},
	\]
    and our assumption on parameters $v$ and $w$.
    Since $\eps,\kappa > 0 $ are arbitrary, 
    the last estimate can be restated as:
    for all $\eps > 0$, $s\ge  2$, and $\Lambda\notin X_s$,
    \begin{align}\label{eq:NNN_1}
	\left|D^{(a_M,b_v], c_w}_{[1,e^N)}(\Lambda)\right| 
    &\ll_{\zeta,\eps} c_w^{-1}\cdot b_v^{(1-\rho)/2}\cdot  s^{6+\eps} \cdot 2^{s}.
    \end{align}
Since by \eqref{eq:summ}
\[
\sum_{s \geq 2} \nu(X_s)  < \infty,
\]
it follows from Borel-Cantelli's Lemma that there exists a conull Borel set $\Psi \subset [0,1)^2$ and a measurable map $s_o : \Psi \ra \bN$ such that for all $\ul{x} \in \Psi$ and  
$s \geq s_o(\ul{x})$, we have $\Lambda_{\ul{x}} \notin X_s$ and the estimate \eqref{eq:NNN_1} holds:
\begin{equation}\label{eq:almost}
\left|D^{(a_M,b_v], c_w}_{[1,e^N)}(\Lambda_{\ul{x}})\right| \ll_{\zeta,\eps}
 c_w^{-1}\cdot b_v^{1-\rho} \cdot s^{6+\eps}   \cdot 2^{s}
 ,
\end{equation}
for all integers $1 \leq N < 2^s$, $1 \leq M < 2^{t(s)}$, $v$ such that $1/M<2^{-v}$, and $w\le r(s)$.

\medskip

Now for general $T\ge 1$,
we denote by $N_T$ the positive integer such that
$$
e^{N_T-1} \leq T  < e^{N_T}.
$$
We apply the above estimate with 
$$
s=\lfloor\log_2(N_T+1)\rfloor+1,\quad t=\lfloor\log_2\big((\ln T)^{\theta/\sigma}+1\big)\rfloor+1,\quad r=\lfloor\log_2\big(\zeta^{-1/2}(\ln T)^{\theta/2}\big)\rfloor+1,
$$
so that 
$$
2^s\ll N_T,\quad 2^t\ll N_T^{\theta/\sigma},\quad r\ll s.
$$
Let  
\[
T_o(\ul{x}) := \min\big\{ T \geq 1 \,  : \,  \lfloor\log_2(N_T+1)\rfloor \geq s_o(\ul{x})  \big\}.
\]
For $a\in ((\ln T)^{-\theta}, 1)$, we pick an integer $M<2^{t}$
such that 
$$
a_M\le a< a_{M-1}.
$$
Note that then 
$$
a_{M-1}-a_M\ll M^{-1-\sigma}\ll a^{1+\sigma^{-1}}.
$$
For  the parameter $c\in (0,1/2)$ satisfying $c\ge \zeta^{1/2} (\ln T)^{-\theta/2}$ we choose $w\le r$  such that 
$$
c_{w}\le c <c_{w-1}.
$$
We observe that 
$$
 \Omega^{(a_{M-1},b_v], c_{w}}_{[1,e^{N_T-1})}\subset  
 \Omega^{(a,b_v], c}_{[1,T)}\subset
 \Omega^{(a_M,b_v], c_{w-1}}_{[1,e^{N_T})}.
$$
In particular,
$$
\left|\Lambda\cap \Omega^{(a,b_v], c}_{[1,T)}\right|\le 
\left|\Lambda\cap \Omega^{(a_M,b_v], c_{w-1}}_{[1,e^{N_T})}\right|,
$$
so that 
$$
D^{(a,b_v], c}_{[1,T)}\le D^{(a_M,b_v], c_{w-1}}_{[1,e^{N_T})}+\Vol_3\left(\Omega^{(a_{M},b_v], c_{w-1}}_{[1,e^{N_T})}\right)-\Vol_3\left(\Omega^{(a_{M-1},b_v], c_{w}}_{[1,e^{N_T-1})}\right).
$$
According to the volume estimates from Lemma \ref{Lemma_VolOmegaTy}, 
\begin{align*}
\Vol_3\left(\Omega^{(a_{M},b_v], c_{w-1}}_{[1,e^{N_T})}\right)&=
2 N_T^2 (b_v-a_{M})+O_{\zeta,\theta}\Big(
\ln T (\ln\ln T) (b_v-a_{M})+1\Big),
\end{align*}
and
\begin{align*}
\Vol_3\left(\Omega^{(a_{M-1},b_v], c_{w}}_{[1,e^{N_T-1})}\right)   &=
2 (N_T-1)^2 (b_v-a_{M-1})+O_{\zeta,\theta}\Big(
\ln T (\ln\ln T) (b_v-a_{M-1})+1\Big).
\end{align*}
Then
\begin{align*}
\Vol_3\left(\Omega^{(a_{M},b_v], c_{w-1}}_{[1,e^{N_T})}\right)-\Vol_3\left(\Omega^{(a_{M-1},b_v], c_{w}}_{[1,e^{N_T-1})}\right)
&\ll_{\zeta,\theta} N_T^2 (a_{M-1}-a_M)+N_T b_v+ 
\ln T (\ln\ln T) b_v+1\\
&\ll (\ln T)^2 a^{1+\sigma^{-1}} +
\ln T (\ln\ln T) b_v+1.
\end{align*}
Applying \eqref{eq:almost}, we obtain that for $\ul{x}\in \Psi$ and $T\ge T_o(\ul{x})$,
\begin{align*}
\left|D^{(a_M,b_v], c_w}_{[1,e^{N_T})}(\Lambda_{\ul{x}})\right| \ll_{\zeta,\eps} c_w^{-1}\cdot b_v^{(1-\rho)/2}
  \cdot (\ln N_T)^{6+\eps} 
  \cdot \ln T
. 
\end{align*}
Combining those estimates, we conclude that  
\begin{align*}
D^{(a,b_v], c}_{[1,T)}(\Lambda_{\ul{x}}) \ll_{\zeta, \theta,\eps}
  (\ln T)^2 a^{1+\sigma^{-1}}+c^{-1}\cdot b_v^{(1-\rho)/2} \cdot (\ln \ln T)^{6+\eps} \cdot 
  \ln T+1
\end{align*}
for sufficiently large $T$.
The lower bound on $D^{(a,b_v], c}_{[1,T)}$ is proved similarly.
Ultimately, we conclude that
\begin{align*}
\left|\Lambda_{\ul{x}}\cap \Omega^{(a,b_v],c}_{[1,T)}\right|
=&\Vol_3\left(\Omega^{(a,b_v],c}_{[1,T)}\right)\\
&+O_{\zeta, \theta, \eps}\left(
(\ln T)^2 a^{1+\sigma^{-1}}+c^{-1}\cdot b_v^{(1-\rho)/2} \cdot (\ln \ln T)^{6+\eps} \cdot 
\ln T+1
\right)
\end{align*}
for sufficiently large $T$.

\medskip

Finally, for $b\in (a,1)$, we pick $v\le t$ such that $b_v\le b< b_{v-1}$.
Since 
$$
\left|\Lambda_{\ul{x}}\cap \Omega^{(a,b],c}_{[1,T)}\right|
=\left|\Lambda_{\ul{x}}\cap \Omega^{(a,b_{v-1}],c}_{[1,T)}\right|-
\left|\Lambda_{\ul{x}}\cap \Omega^{(b,b_{v-1}],c}_{[1,T)}\right|,
$$
and 
$$
\Vol_3\left(\Omega^{(a,b],c}_{[1,T)}\right)=
\Vol_3\left(\Omega^{(a,b_{v-1}],c}_{[1,T)}\right)-
\Vol_3\left(\Omega^{(b,b_{v-1}],c}_{[1,T)}\right),
$$
we deduce that
\begin{align*}
\left|\Lambda_{\ul{x}}\cap \Omega^{(a,b],c}_{[1,T)}\right|
=&\Vol_3\left(\Omega^{(a,b],c}_{[1,T)}\right)\\
&+O_{\zeta, \theta, \eps}\left(
(\ln T)^2 \cdot b^{1+\sigma^{-1}}+c^{-1}\cdot b^{(1-\rho)/2} \cdot (\ln \ln T)^{6+\eps}   \cdot 
\ln T+1
\right)
\end{align*}
for sufficiently large $T$.
We recall that $\sigma\ge \rho^{-1}$, so that we get the best estimate  
when $\sigma= \rho^{-1}$:
\begin{align*}
\left|\Lambda_{\ul{x}}\cap \Omega^{(a,b],c}_{[1,T)}\right|
=&\Vol_3\left(\Omega^{(a,b],c}_{[1,T)}\right)\\
&+O_{\zeta,\theta,\eps}\left(
(\ln T)^2\cdot  b^{1+\rho}+c^{-1}\cdot b^{(1-\rho)/2} \cdot (\ln \ln T)^{6+\eps}  \cdot 
\ln T+1
\right).
\end{align*}
Let us suppose that 
$b\le (c\cdot\ln T)^{-2/(3\rho+1)}.$
Then one checks by a direct computation that the second summand in the error term dominates the first summand, so that we get 
\begin{align*}
\left|\Lambda_{\ul{x}}\cap \Omega^{(a,b],c}_{[1,T)}\right|
=\Vol_3\left(\Omega^{(a,b],c}_{[1,T)}\right)
+O_{\zeta,\theta, \eps}\left(
c^{-1}\cdot b^{(1-\rho)/2} \cdot (\ln \ln T)^{6+\eps}  \cdot 
\ln T+1
\right)
\end{align*}
for $T\ge T_o(\ul{x})$.
This provides a non-trivial estimate when
$b\ge (c\cdot\ln T)^{-2/(\rho+1)}.$
On the other hand, when
$b\ge (c\cdot\ln T)^{-2/(3\rho+1)},$
we get the bound:
\begin{align*}
\left|\Lambda_{\ul{x}}\cap \Omega^{(a,b],c}_{[1,T)}\right|
=\Vol_3\left(\Omega^{(a,b],c}_{[1,T)}\right)
+O_{\zeta,\theta, \eps}\left(
 b^{1+\rho}\cdot(\ln \ln T)^{6+\eps}\cdot (\ln T)^2+1\right)
\end{align*}
for  $T\ge T_o(\ul{x})$.
These estimates hold for all $T$ with 
explicit constant depending $\ul{x}$.
This gives Theorem \ref{Thm_Cnt} by choosing $\eta=2/(\rho+1)$.

\section{Proof of Theorem \ref{Cor_Cnt}}
\label{Sec:Cor}

To prove the corollary, we investigate existence of points 
in lattices 
$$
\Lambda_{\ul{x}} = \{ (\ul{p} + q\ul{x},q) \in \bR^2 \times \bR \,  : \,  (\ul{p},q) \in \bZ^2 \times \bZ \}
$$
contained in very thin hyperbolic strips. For $T \geq 1$,  let $a_{T}$ 
be a positive function of $T$.  We consider the domains
\begin{equation}
\label{Def_Upsilon}
\Upsilon_T := \left\{ (\ul{u},y) \in \bR^2 \times \bR \,  : \, |u_1u_2| \cdot y \leq a_{T},  \enskip \max(|u_1|,|u_2|) \leq \frac{1}{2},  \enskip 1 \leq y < T \right\}.
\end{equation}

\noindent Our main result in this section reads as follows. 
\begin{lemma}
	\label{Lemma_Absence}
	Suppose that $a_{T}$ is non-increasing and 
$a_{T}=o((\ln T)^{-2})$ as $T\to\infty$.  
Then there is a conull set $Z \subset [0,1)^2$ and a measurable function $T : Z \ra [1,\infty)$ such that for every $\ul{x} \in Z$,  
	\[
	\Lambda_{\ul{x}} \cap \Upsilon_T = \emptyset,  \quad \textrm{for all $T \geq T(\ul{x})$}.
	\]
\end{lemma}

\begin{proof}
By our assumption on $\at$,  the family $(\Upsilon_T)$ is decreasing. Hence,  if we can show that for almost every $\ul{x} \in [0,1)^2$,  there is $T(\ul{x}) \geq 1$
such that $\Lambda_{\ul{x}} \cap \Upsilon_{T(\ul{x})} = \emptyset$,  the lemma is established.  Let us consider the counting function
\[
N_T(\ul{x}) := \left|\Lambda_{\ul{x}} \cap \Upsilon_T\right|,  \quad \textrm{for $\ul{x} \in [0,1)^2$,}
\]
and the sets 
$$
\Upsilon_T(q) := \left\{ \ul{x} \in \bR^2 \,  : \,  |x_1x_2| \leq \frac{\at}{q},  \enskip \max(|x_1|,|x_2|) \leq \frac{1}{2} \right\}.
$$

Since the map $\ul{x} \mapsto q\ul{x}$ preserves the Haar measure on the $\bR^2/\bZ^2$,  we
note that upon unwrapping the the definition of the counting function $N_T$,  
\begin{align*}
\int_{[0,1)^2} N_T(\ul{x}) \,  d\ul{x} 
&= \int_{[0,1)^2} \left( \sum_{q=1}^T |(\bZ^{2} + q\ul{x}) \cap \Upsilon_T(q)| \right) \,  d\ul{x} \\[0.2cm]
&= \int_{[0,1)^2} \left(\sum_{q=1}^T\sum_{p\in\bZ^2} \chi_{\Upsilon_T(q)}(p+\ul{x})\right) \,  d\ul{x}
=\sum_{q=1}^T \Vol_2(\Upsilon_T(q)).
\end{align*}
Furthermore,  
\[
 \Lambda_{\ul{x}} \cap \Upsilon_T \neq \emptyset\quad\Longleftrightarrow\quad N_T(\ul{x}) \geq 1,
\]
so that 
\begin{align*}
\Vol_2\left(\left\{ \ul{x} \in [0,1)^2 \,  : \,  \Lambda_{\ul{x}} \cap \Upsilon_T \neq \emptyset \right\}\right)
\leq 
\int_{[0,1)^2} N_T(\ul{x}) \,  d\ul{x}
= \sum_{q=1}^T \Vol_2(\Upsilon_T(q)).
\end{align*}
It follows from \eqref{VolXigamma}
that for sufficiently large $T$,
\begin{align*}
\Vol_2(\Upsilon_T(q)) &=\frac{1}{4} \Vol_2\big(\Xi({4a_{T}}/{q})\big)={4a_{T}}/{q}\cdot(1-\ln({4a_{T}}/{q}))\\
 &=\frac{4}{q}\cdot a_{T}\left(1-\ln\left(4a_{T}\right)\right) + \frac{4 \ln(q)}{q}\cdot a_{T}.
\end{align*}
Hence,
\[
\sum_{q=1}^T \Vol_2(\Upsilon_T(q)) \ll 
(\ln T)\cdot a_{T}\left(1-\ln\left(4a_{T}\right)\right) + (\ln T)^2\cdot a_{T}.
\]
By our assumption,  the right-hand side tends to zero as $T \ra \infty$,  and thus we can find an increasing sequence $(T_k)$ such that
\[
\sum_{k=1}^\infty \Vol_2\left(\left\{ \ul{x} \in [0,1)^2 \,  : \,  \Lambda_{\ul{x}} \cap \Upsilon_{T_k} \neq \emptyset\right\}\right) < \infty.
\]
By Borel-Cantelli's Lemma,  there exists a conull subset $Z \subset [0,1)^2$ such that for every $\ul{x} \in Z$,  there is an index $k(\ul{x})$ such that 
$\Lambda_{\ul{x}} \cap \Upsilon_{T_{k(\ul{x})}} \neq \emptyset$.
This finishes the proof. 
\end{proof}

\begin{proof}[Proof of Theorem \ref{Cor_Cnt}]
Take a non-increasing function $a_{T}$ such that  
$a_{T}=o((\ln T)^{-2})$ and consider the sets $\Omega=\Omega_{[1,T)}^{(a_T,b],1/2}$.
Then
\[
L(\ul{x};\bt)\cap [1,T) = L(\ul{x};a_{T})\cap [1,T) \bigsqcup Q_\Omega(\ul{x}).
\]
By Lemma \ref{Lemma_Absence},  $L(\ul{x};a_{T}) = \emptyset$ for all almost all $\ul{x} \in \bR^2$ and sufficiently large $T$ (depending on $\ul{x}$),  and thus $L(\ul{x};\bt)\cap [1,T) = Q_\Omega(\ul{x})$.  Hence, Theorem \ref{Cor_Cnt} follows from Theorem \ref{Thm_Cnt}
and the volume formula \eqref{eq:v}.
\end{proof}

\appendix

\section{Volume estimates}
\label{sec:vol}

In this section we discuss some basic facts concerning the volumes of the sets
\[
\Omega
= \left\{ (\ul{x},y) \in \bR^2 \times [1,T) \,  : \,  
\begin{array}{c}
\max(|x_1|,|x_2|) \leq \ct \\[0.2cm]
\at < |x_1x_2| \cdot y \leq \bt
\end{array}
\right\}
\]
with $0<\at<\bt< 1$ and $\ct\le 1/2$. 
We observe that these domains can be represented in terms of 
more basic sets 
\begin{equation}
\label{Def_xigamma}
\Xi(\gamma) := \left\{ \ul{x} \in [\shortminus 1,1]^2 \,  : \,  |x_1 x_2| \leq \gamma \right\},\quad \gamma>0.
\end{equation}
Direct computation gives 
	\begin{equation}\label{VolXigamma}
	\Vol_2(\Xi(\gamma)) = 4\max\big(1,\gamma \cdot (1 - \ln\gamma)\big)\quad \hbox{for $\gamma>0$}.
	\end{equation}
In particular, it follows from the Mean Value Theorem that 
\begin{equation}\label{eq:mean}
\Vol_2(\Xi(\gamma_2))-\Vol_2(\Xi(\gamma_1))\le 4|\ln\min(1,\gamma_1)|\cdot (\gamma_2-\gamma_1) \quad \hbox{for $\gamma_2>\gamma_1>0$}.
\end{equation}
We observe that the $y$-sections
\[
\Omega(y) := \{ \ul{x} \in \bR^2 \,  : \,  (\ul{x},y) \in \Omega \},  \quad \hbox{for $y \in [1,T).$}
\]
can be written as 
\[
\Omega(y) = \ct\cdot\left(\Xi\left( \frac{\bt}{\ct^2 \cdot y}\right) \setminus \Xi\left(\frac{\at}{\ct^2 \cdot y}\right)\right),  
\]
and thus
\[
\Vol_2(\Omega(y)) = \ct^2 \cdot \left( \Vol_2\left(\Xi\left( \frac{\bt}{\ct^2 \cdot y}\right)\right)-\Vol_2\left(\Xi\left( \frac{\at}{\ct^2 \cdot y}\right)\right)\right),  \quad \textrm{for all $y \in [1,T)$}.
\]
In view of this, the following lemma can be deduced from \eqref{VolXigamma} by a direct computation:

\begin{lemma}
	\label{Lemma_VolOmegaTy}
	For $\max(1,\frac{b}{c^{2}}) \leq y < T$,  
	\begin{align*}
	\Vol_2(\Omega(y)) 
	=& \frac{4 \cdot \ln y}{y} \cdot (\bt-\at) \\
	& + \frac{4}{y} \cdot \Big( (\bt-\at) \cdot (1 + 2 \cdot \ln \ct) - \bt \ln \bt + \at \ln \at \Big),
	\end{align*}
	   so that  when $b\ll c^2$,
	\begin{align*}
	\Vol_3(\Omega) =& 2 \cdot (\ln T)^2 \cdot (\bt-\at) \\
	&+ 
	4 \cdot \ln T \cdot \Big( (\bt-\at) \cdot (1 + 2 \cdot \ln \ct) - \bt \ln \bt + \at \ln \at \Big)+O(1).
	\end{align*}
	In particular,  if in addition $\at \gg (\ln T)^{-\theta}$ for some $\theta > 0$,  then
	\begin{align}\label{eq:v}
	\Vol_3(\Omega) = 2 \cdot (\ln T)^2 \cdot (\bt - \at) + O\Big(\ln T(\ln\ln T)\cdot(\bt-\at)+1\Big).
	\end{align}
\end{lemma}


\section{An auxiliary double sum}

This is a largely technical section where we collect some estimates on certain multi-parameter sums that are used in Section \ref{Sec:MeanCnt}.  This part can be safely skipped 
on a first read.  \\

We begin with a simple observation.  For $\ul{u} = (u_1,u_2) \in \bR^2_{+}$,  let
\begin{equation}
N(\ul{u}) := \left| 
\bZ^d \cap \left([\shortminus u_1,u_1] \times [\shortminus u_2,u_2] \right)
\right|.
\end{equation}
A simple counting argument that we leave to the reader shows that $N(\ul{u}) \ll G(\ul{u})$,  where
\begin{equation}
\label{Def_G}
G(\ul{u})
=
\left\{ 
\begin{array}{cl}
1 & \textrm{if $\lceil \ul{u} \rceil < 1$} \\[0.2cm]
\lceil \ul{u} \rceil & \textrm{if $\lfloor \ul{u} \rfloor < 1 \leq \lceil \ul{u} \rceil$} \\[0.2cm]
u_1 u_2 & \textrm{if $1 \leq \lfloor \ul{u} \rfloor$}
\end{array}
\right..
\end{equation}
Let us fix a constant $M > 0$ for the rest of the section.  For $t = (t_1,t_2) \in \bR^{2}_{+}$,  we define the function
\begin{equation}
\label{Def_Ft}
F_t(q) = \frac{G(2Mqe^{-t_1},2Mqe^{-t_2})}{q^2} \cdot e^{-(t_1 + t_2)},  \quad q \geq 1.
\end{equation}
The explicit formula for $G$ above tells us that
\begin{equation}
\label{Ft_explicit}
F_t(q) = 
\left\{
\begin{array}{cl}
\frac{e^{-(t_1 + t_2)}}{q^2}
& \textrm{if \, $q < \frac{e^{\lfloor t \rfloor}}{2M}$} \\[0.2cm]
2M\frac{e^{-(2\lfloor t \rfloor + \lceil t \rceil)}}{q} 
& \textrm{if \,  $\frac{e^{\lfloor t \rfloor}}{2M} \leq q < \frac{e^{\lceil t \rceil}}{2M}$} \\[0.2cm]
4M^{2}e^{-2(t_1 + t_2)} 
& \textrm{if \, $\frac{e^{\lceil t \rceil}}{2M} \leq q$} 
\end{array}
\right..
\end{equation}

Our main goal in this section is to prove the following upper bound on an auxiliary double sum which involves the function $F_t$.
\begin{lemma}
\label{Lemma_AuxilliarySum}
Let $0 < \alpha < \beta \leq M$ and let $t = (t_1,t_2) \in \bR^2_{+}$.  Define
\[
\alpha_t = \alpha \cdot e^{t_1 + t_2} \qand \beta_t = \beta \cdot e^{t_1 + t_2}
\]
and suppose that $\alpha<1$, $\alpha_t \geq 1$.  Then,  
\[
\sum_{q_1,  q_2 = \alpha_t}^{\beta_t} F_t\left( \frac{\max(q_1,q_2)}{\gcd(q_1,q_2)}\right) 
\ll_M
e^{-(t_1+t_2)} + (\beta-\alpha) \cdot \max\left(1,\ln\left(\frac{\beta}{\alpha}\right)\right) \cdot \max(1,t_1 + t_2),
\]
where the implicit constants depend only on $M$. 
\end{lemma}

\begin{remark}
We adopt the following sum convention: If $1 \leq \gamma < \delta$,  and
$m$ and $n$ are integers such that
\[
m < \gamma \leq m+1 \qand n \leq \delta < n+1,
\]
then $\sum_{q = \gamma}^\delta := \sum_{q = m+1}^n$,  where the right-hand side
is defined to be zero if $m = n$.  
\end{remark}

\subsection{Proof of Lemma \ref{Lemma_AuxilliarySum}}

The following standard estimates will be used in the proof. For 
$1 \leq \gamma < \gamma+1 <\delta $ we have
\begin{align}
\sum_{q = \gamma}^\delta \frac{1}{q} &=\ln\left(\frac{\delta}{\gamma}\right)+\mathcal{O}\left(\frac{1}{\gamma}\right), \label{sumq1} 
\end{align}
where the implicit constants are independent of $\gamma$ and $\delta$.  In addition, the following elementary upper bound on the sum-of-divisors function holds
\begin{equation}
\label{estsigma}
\frac{1}{n} \sum_{\substack{m=1 \\ m \mid n}}^{n} m \ll \ln(n).
\end{equation}
Let us begin with the proof.  We fix $t = (t_1,t_2) \in \bR^2_{+}$ and $0 < \alpha < \beta \leq M$ such that 
\[
1 > \alpha  \qand  \alpha_t \geq 1 \qand e^{\lfloor t \rfloor} > 2M,
\]
and we want to bound the double sum
\[
S_t(\alpha,\beta) := 
\sum_{q_1, q_2 = \alpha_t}^{\beta_t}
F_t\left( \frac{\max(q_1,q_2)}{\gcd(q_1,q_2)} \right)
\]
from above.  Note that $F_t(1)\ll_{M} e^{-(t_1 + t_2)}$. \\

We first consider the case when $\beta_t - \alpha_t < 1$.  Then, the double sum above contains at most one term (necessarily with $q_1 = q_2$),  and thus 
\begin{equation}
\label{Stbam1}
S_t(\alpha,\beta) \leq F_t(1) \ll_M e^{-(t_1 + t_2)}.
\end{equation}
Let us from now on assume that $\beta_t - \alpha_t \geq 1$,  and define the sets
\begin{align*}
\cE_0 &:= 
\left\{ (q_1,q_2) \in [\alpha_t,\beta_t]^2 \cap \bN^2 \,  : \, 
\frac{\max(q_1,q_2)}{\gcd(q_1,q_2)} < \frac{e^{\lfloor t \rfloor}}{2M} \right\}, \\[0.2cm]
\cE_1 &:= 
\left\{ (q_1,q_2) \in [\alpha_t,\beta_t]^2 \cap \bN^2 \,  : \, 
\frac{e^{\lfloor t \rfloor}}{2M} \leq \frac{\max(q_1,q_2)}{\gcd(q_1,q_2)} < 
\frac{e^{\lceil t \rceil}}{2M} \right\}, \\[0.2cm]
\cE_2 &:= 
\left\{ (q_1,q_2) \in [\alpha_t,\beta_t]^2 \cap \bN^2 \,  : \, 
\frac{e^{\lceil t \rceil}}{2M} \leq \frac{\max(q_1,q_2)}{\gcd(q_1,q_2)}  \right\},
\end{align*} 
and the functions
\[
S_t^{(k)}(\alpha,\beta) := \sum_{(q_1,q_2) \in \cE_k} F_t\left(\frac{\max(q_1,q_2)}{\gcd(q_1,q_2)} \right),  \quad \textrm{for $k=0,1,2$}.
\]
Note that 
\[
S_t(\alpha,\beta) = S_t^{(0)}(\alpha,\beta) + S_t^{(1)}(\alpha,\beta) + S_t^{(2)}(\alpha,\beta).
\]
We will estimate these three partial sums separately below.  

\subsubsection*{\textbf{An upper bound for $S_t^{(0)}(\alpha,\beta)$}}

Note that if $\gcd(q_1,q_2) = d$,  then
\[
\left(\frac{\gcd(q_1,q_2)}{\max(q_1,q_2)}\right)^2 = \frac{d^2}{\max(q_1,q_2)^2} \leq \frac{d}{q_1 \cdot q_2'}, 
\]
where $q_2 = d \cdot q_2'$.  Hence,  
\begin{align*}
S_t^{(0)}(\alpha,\beta) 
&= \sum_{(q_1,q_2) \in \cE_o} \left(\frac{\gcd(q_1,q_2)}{\max(q_1,q_2)}\right)^2 \cdot e^{-(t_1 + t_2)}
\leq \sum_{q_1,q_2 = \alpha_t}^{\beta_t} \left(\frac{\gcd(q_1,q_2)}{\max(q_1,q_2)}\right)^2 \cdot e^{-(t_1 + t_2)} \\[0.2cm]
&\leq 2 \cdot \sum_{q_1 = \alpha_t}^{\beta_t} 
\left( \sum_{\substack{d=1 \\ d \mid q_1}}^{q_1} \,  
\sum_{q_2' = \frac{\alpha_t}{d}}^{\frac{\beta_t}{d}} \frac{d}{q_1 \cdot q_2'} \right) \cdot e^{-(t_1 + t_2)}.
\end{align*}
By \eqref{sumq1} (with $\gamma = \frac{\alpha_t}{d}$ and $\delta = \frac{\beta_t}{d}$),
we now see that 
\begin{align*}
S_t^{(0)}(\alpha,\beta) 
&\ll 
\sum_{q_1 = \alpha_t}^{\beta_t} 
\left( \frac{1}{q_1} 
\sum_{\substack{q_1 = 1 \\ d \mid q}}^{q_1} d 
\right) \cdot
\left( 1 + \ln\left(\frac{\beta}{\alpha}\right)
\right) \cdot e^{-(t_1 + t_2)},
\end{align*}
and by 
\eqref{estsigma}, 
\begin{align*}
\sum_{q = \alpha_t}^{\beta_t} \left( \frac{1}{q} 
\sum_{\substack{q = 1 \\ d \mid q}} d \right) \cdot e^{-(t_1 + t_2)} 
&\ll \left( \sum_{q=\alpha_t}^{\beta_t} \ln(q) \right) \cdot e^{-(t_1 + t_2)}
\leq (\beta_t - \alpha_t) \cdot  \ln(\beta_t) \cdot e^{-(t_1 + t_2)} \\[0.2cm]
&\ll_M (\beta-\alpha) \cdot \max(1,t_1 + t_2),
\end{align*}
since $\beta \leq M$.  
We conclude that
\begin{equation}
\label{St0}
S_t^{(0)}(\alpha,\beta) 
\ll (\beta-\alpha) \cdot \max\left(1,\ln\left(\frac{\beta}{\alpha}\right)\right) \cdot \max(1,t_1 + t_2).
\end{equation}

\subsubsection*{\textbf{An upper bound for $S_t^{(1)}(\alpha,\beta)$}}

To simplify notation,  let us assume that $t_1 \leq t_2$ so that $\lfloor t \rfloor = t_1$
and $\lceil t \rceil = t_2$.  Then,
\[
S_t^{(1)}(\alpha,\beta) = \left( \sum_{(q_1,q_2) \in \cE_1} \frac{\gcd(q_1,q_2)}{\max(q_1,q_2)} \right) \cdot e^{-(2t_1 + t_2)},
\]
If $(q_1,q_2) \in \cE_1$,  then
\[
\frac{e^{t_1}}{2M} \leq \frac{\max(q_1,q_2)}{\gcd(q_1,q_2)} \qand \max(q_1,q_2) \leq \beta_t = \beta \cdot e^{t_1 + t_2},
\]
and thus $\gcd(q_1,q_2) \leq 2M \cdot \beta \cdot e^{t_2}$.  Let $\gamma_t = \min(\alpha_t,2M \cdot \beta \cdot e^{t_2})$,  and note that
\begin{align*}
S_t^{(1)}(\alpha,\beta)
&\leq \sum_{d=1}^{\gamma_t} \left( 
\sum_{q_1,q_2 = \frac{\alpha_t}{d}}^{\frac{\beta_t}{d}} \frac{1}{\max(q_1,q_2)} \right) \cdot e^{-(2t_1 + t_2)} \\[0.2cm]
&+
\sum_{d=\alpha_t}^{\beta_t} \left( 
\sum_{q_1,q_2 = 1}^{\frac{\beta_t}{d}} \frac{1}{\max(q_1,q_2)} \right) \cdot e^{-(2t_1 + t_2)} \\[0.2cm]
&\ll
\left( \sum_{d=1}^{\gamma_t} \frac{\beta_t - \alpha_t}{d} \right) \cdot e^{-(2t_1 + t_2)}
+ \left( \sum_{d=\alpha_t}^{\beta_t} \frac{\beta_t}{d} \right) \cdot e^{-(2t_1 + t_2)} \\[0.2cm]
&\ll
(\beta-\alpha) \cdot (1 + \ln(\gamma_t)) \cdot e^{-t_1} + \beta \cdot \left(\ln\left(\frac{\beta}{\alpha}\right)+\mathcal{O}\left(\frac{1}{\alpha_{t}}\right) \right) \cdot e^{-t_1},
\end{align*}
where we in the last inequality have used \eqref{sumq1}  (with parameters $\gamma = 1, \delta = \gamma_t$ and $\gamma = \alpha_t,  \delta = \beta_t$).  Since
\[
\ln(\gamma_t) \ll 1 + t_2,
\]
we have
$$(\beta-\alpha) \cdot (1 + \ln(\gamma_t)) \cdot e^{-t_1}\ll_{M}(\beta-\alpha) \cdot \max(1,t_1 + t_2) \cdot e^{-\lfloor t \rfloor}.$$
For the term
$$ \beta \cdot \left(\ln\left(\frac{\beta}{\alpha}\right)+\mathcal{O}\left(\frac{1}{\alpha_{t}}\right) \right) \cdot e^{-t_1},$$
we consider two separate cases. If $\beta/\alpha\geq e$, then $\ln(\beta/\alpha)$ dominates, and since $\beta-\alpha\gg\beta$, we find
$$\beta \cdot \left(\ln\left(\frac{\beta}{\alpha}\right)+\mathcal{O}\left(\frac{1}{\alpha_{t}}\right) \right) \cdot e^{-t_1}\ll_{M}(\beta-\alpha)\cdot\ln\left(\frac{\beta}{\alpha}\right)\cdot e^{-t_{1}}.$$
If $\beta/\alpha<e$, on the other hand, we deduce
$$\beta \cdot \left(\ln\left(\frac{\beta}{\alpha}\right)+\mathcal{O}\left(\frac{1}{\alpha_{t}}\right) \right) \cdot e^{-t_1}\ll_{M}\beta\cdot\ln\left(\frac{\beta}{\alpha}\right)\cdot e^{-t_{1}}+\frac{\beta}{\alpha}\cdot e^{-2t_{1}-t_{2}}.$$
On observing that for $1<\beta/\alpha\leq e$
$$\beta\cdot \ln\left(\frac{\beta}{\alpha}\right)\ll(\beta-\alpha),$$
we now obtain
$$\beta \cdot \left(\ln\left(\frac{\beta}{\alpha}\right)+\mathcal{O}\left(\frac{1}{\alpha_{t}}\right) \right) \cdot e^{-t_1}\ll_{M}(\beta-\alpha)\cdot e^{-t_{1}}+e^{-2t_{1}-t_{2}}.$$
Then, by combining the previous estimates, we conclude that
\begin{equation}
\label{St1}
S_t^{(1)}(\alpha,\beta) \ll_M (\beta-\alpha) \cdot \max\left(1,\ln\left(\frac{\beta}{\alpha}\right)\right) \cdot \max(1,t_1 + t_2) \cdot e^{-\lfloor t \rfloor}+e^{-(t_{1}+t_{2})}.
\end{equation}

\subsubsection*{\textbf{An upper bound for $S_t^{(2)}(\alpha,\beta)$}}

The following crude estimate will suffice: 
\begin{align}
S^{(2)}_t(\alpha,\beta) 
&= 
\left( \sum_{(q_1,q_2) \in \cE_2} 1 \right) \cdot e^{-2(t_1 + t_2)} 
\leq
\left( \sum_{q_1=\alpha_t}^{\beta_t} \sum_{q_2 = \alpha_t}^{\beta_t} 1\right) \cdot e^{-2(t_1 + t_2)} \nonumber \\[0.2cm]
&\ll (\beta_t-\alpha_t)^2 \cdot e^{-2(t_1 + t_2)} \ll_M (\beta-\alpha)^2 \ll \beta-\alpha, \label{St2}
\end{align}
since $\beta \leq M$.  

\subsubsection*{\textbf{Putting it all together}}

If we now combine \eqref{Stbam1},  \eqref{St0},  \eqref{St1} and \eqref{St2},  
we get
\[
S_t(\alpha,\beta) \ll_M e^{-(t_1+t_2)} + (\beta-\alpha) \cdot \max\left(1,\ln\left(\frac{\beta}{\alpha}\right)\right) \cdot \max(1,t_1 + t_2),
\]
where the implicit constants are independent of $\alpha$ and $\beta$,  and hence the
proof of Lemma \ref{Lemma_AuxilliarySum} is complete.

\end{document}